\documentclass[11pt,a4paper]{article}
\usepackage{makeidx}
\usepackage{amsmath}
\usepackage{amssymb}
\usepackage{amsfonts}
\usepackage{amsthm}
\usepackage{relsize}
\usepackage{setspace}
\usepackage[left=3cm, right=3cm]{geometry}
\usepackage{enumerate}
\usepackage{url}
\usepackage{xspace}
\usepackage{mathrsfs}
\usepackage{tocloft}

\usepackage{authblk}
\usepackage[colorlinks,bookmarksopen=true, pagebackref=true]{hyperref}
\usepackage[english]{babel}
\usepackage{amssymb}
\usepackage{latexsym}
\usepackage{mathrsfs}
\usepackage{xspace}
\usepackage{enumerate, paralist}
\newtheorem{prop}{Proposition}[section]
\newtheorem{thm}[prop]{Theorem}
\newtheorem*{thm*}{Theorem}

\newtheorem*{addendum*}{Addendum}
\newtheorem{cor}[prop]{Corollary}
\newtheorem{lem}[prop]{Lemma}
\newtheorem{thmintro}{Theorem}

\newtheorem{corintro}[thmintro]{Corollary}

\newtheorem{claim}{Claim}
\newtheorem*{claim*}{Claim}
\newtheorem*{convention*}{Convention}
\theoremstyle{definition}
\newtheorem*{defn*}{Definition}
\newtheorem{defn}[prop]{Definition}
\newtheorem{remark}[prop]{Remark}
\newtheorem*{remark*}{Remark}

\newtheorem*{scholium*}{Scholium}
\theoremstyle{remark}

\newtheorem*{example*}{Example}
\numberwithin{equation}{section}

\newcommand{\vareps}{\varepsilon}
\newcommand{\fhi}{\varphi}
\newcommand{\ro}{\varrho}
\newcommand{\teta}{\vartheta}

\newcommand{\NN}{\mathbf{N}}

\newcommand{\RR}{\mathbf{R}}

\newcommand{\sC}{\mathscr{C}}
\newcommand{\sF}{\mathscr{F}}
\newcommand{\sK}{\mathscr{K}}

\newcommand{\sS}{\mathscr{S}}

\newcommand{\inv}{^{-1}}

\newcommand{\se}{\subseteq}

\newcommand{\lra}{\longrightarrow}

\newcommand{\Td}{d_\mathrm{T}}
\newcommand{\Creg}{C_{\mathrm{reg}}}
\newcommand{\Link}{\mathrm{Link}}
\newcommand{\cat}{{\upshape CAT(0)}\xspace}
\newcommand{\catun}{{\upshape CAT(1)}\xspace}
\newcommand{\catmun}{{\upshape CAT($-1$)}\xspace}
\newcommand{\catca}{{\upshape CAT($\kappa$)}\xspace}
\newcommand{\tangle}[2]% angle de Tits
{\angle(#1,#2)}
\newcommand{\aangle}[3]% angle d'Alexandrov
{\angle_{#1}(#2,#3)}
\newcommand{\cangle}[3]% angle de comparaison
{\overline{\angle}_{#1}(#2,#3)}
\DeclareMathOperator{\proj}{proj} 
\DeclareMathOperator{\Stab}{Stab} \DeclareMathOperator{\Fix}{Fix} 
 \DeclareMathOperator{\Ker}{Ker}

\DeclareMathOperator{\Id}{Id}

\DeclareMathOperator{\Isom}{Is}
\DeclareMathOperator{\Ch}{Ch}

 \DeclareMathOperator{\rad}{rad} 
\newcommand{\bd}{\partial} % mathoperator ajoute beaucoup d'espace....

\def\Aut{\mathop{\mathrm{Aut}}\nolimits}

\def\max{\mathop{\mathrm{max}}\nolimits}
\def\Op{\mathop{\mathrm{Opp}}\nolimits}
\def\Ant{\mathop{\mathrm{Ant}}\nolimits}
\makeindex
\begin{document}

%\title{\cat spaces with cocompact amenable isometry groups}
\title{An indiscrete Bieberbach theorem:\\ from amenable \cat groups to Tits buildings}

\author[1]{Pierre-Emmanuel Caprace\thanks{F.R.S.-FNRS research associate, supported in part by the ERC (grant \#278469)}}
%\ead{pe.caprace@uclouvain.be}
\author[2]{Nicolas Monod\thanks{Supported in part by the Swiss National Science Foundation and by the ERC}}

%\affil[1]{Universit\'e catholique de Louvain, IRMP, Chemin du Cyclotron 2, bte L7.01.02, 1348 Louvain-la-Neuve, Belgique}
\affil[1]{UCLouvain, 1348 Louvain-la-Neuve, Belgium}
\affil[2]{EPFL, 1015 Lausanne, Switzerland}
%\email{nicolas.monod@epfl.ch}

\date{}
%\date{February 5, 2015}%  ``\today'' est dangereux si on met le papier sur Arxiv, voir http://arxiv.org/help/faq/today
%\subjclass{20F65; 20E42, 20F50, 43A07} % AMS classification numbers
%\keywords{Non-positive curvature, \cat space, locally compact group, amenable group}
%

\maketitle

\begin{abstract}
Non-positively curved spaces admitting a cocompact isometric action of an amenable group are investigated. A classification is established under the assumption that there is no global fixed point at infinity under the full isometry group. The visual boundary is then a spherical building. When the ambient space is geodesically complete, it must be a product of flats, symmetric spaces, biregular trees and Bruhat--Tits buildings.

We  provide moreover a sufficient condition for a spherical building arising as the visual boundary of a proper \cat space to be Moufang, and deduce that an irreducible locally finite Euclidean building of dimension~${\geq 2}$ is a Bruhat--Tits building if and only if its automorphism group acts cocompactly and chamber-transitively at infinity.
\end{abstract}

%\let\languagename\relax  % TO FIX A BUG IN RUNNING HEADERS AND BABEL

%%%%%%%%%%%%%%%%%%%%%%%%%%%%%%%%%%%%%%%%%%%%%%%%%%%%%%%%%%%%%%%%%%%%%%%%%%%%%%
\section{Introduction}
%%%%%%%%%%%%%%%%%%%%%%%%%%%%%%%%%%%%%%%%%%%%%%%%%%%%%%%%%%%%%%%%%%%%%%%%%%%%%%

The meeting ground between non-positive curvature and amenability is shaped by flatness. This principle emerged in the 1970s for Riemannian geometry~\cite{Avez}, \cite{Gromoll-Wolf}, \cite{Lawson-Yau} and culminated in 1998 as a definitive metric statement: The flat Euclidean spaces are the only geodesically complete locally compact \cat spaces admitting a proper cocompact isometric action of an amenable discrete group~\cite[Cor.~C]{AB98}. In particular, amenable discrete \cat groups are all Bieberbach groups. Earlier versions include~\cite[Thm.~2]{BurgerSchroeder87} and~\cite{Anderson87}.

However, as soon as we broaden our view from discrete to locally compact \cat groups, the landscape becomes much more scenic. To wit, symmetric spaces and Bruhat--Tits buildings support a cocompact action of the corresponding minimal parabolic groups, which are amenable (indeed soluble-by-compact). A taxonomy was missing, even in the classical case of Riemannian manifolds.

The main objective of this article is to establish the classification of geodesically complete \cat spaces without global fixed point at infinity that admit a cocompact amenable group of isometries: they are all products of symmetric spaces, Bruhat--Tits buildings and trees (Theorem~\ref{cor:CoctAmen} below). The assumption that there be no fixed point at infinity under the \emph{entire} (typically non-amenable) group of isometries is necessary. Otherwise, the variety of possible spaces becomes richer, even among manifolds: uncountably many homogeneous non-positively curved manifolds that are not symmetric spaces arise as soluble Lie groups endowed with an invariant Riemannian metric of non-positive sectional curvature (see~\cite{Heintze74} and~\cite{AzencottWilson}). 

In conclusion, we submit that symmetric spaces, buildings and trees are really the next flattest \cat spaces after Euclidean spaces: our result could be seen as a non-discrete Bieberbach theorem. The associated full isometry groups are for instance semi-simple algebraic groups with Kazhdan's property; although being traditionally viewed as ``very non-amenable'', in the present geometric setting their salient property is rather to contain a cocompact amenable subgroup.

\begin{remark*}
The threshold to flatness is situated slightly further than discreteness: Corollary~C of~\cite{AB98} holds more generally for \emph{unimodular} amenable  locally compact groups, as shown by combining~\cite[Main Theorem]{AB98} with~\cite[Theorem M]{amenis}.
\end{remark*}

The major part of this article will consist in establishing the classification \emph{at infinity}, as follows.

\begin{thmintro}[Classification at infinity]\label{thm:NotGoedComplete}
Let $X$ be a proper  \cat space without a global fixed point at infinity. Assume that  $X$ admits a cocompact isometric action of an amenable locally compact group. 

\medskip
Then $\bd X$ is a metric spherical building and each of its irreducible factors of dimension~${\geq 1}$ is Moufang.

Moreover, the stabiliser of every point of $\bd X$ acts cocompactly on $X$. 
\end{thmintro}

A sufficient condition for the absence of global fixed points at infinity is that $\Isom(X)$ be unimodular,  see~\cite[Theorem M]{amenis}.
This is   automatic if  $\Isom(X)$ contains a lattice, e.g.\ if $X$ is the universal cover of a compact locally \cat space. 

A \textbf{metric spherical building} is a spherical building in the sense Kleiner--Leeb~\cite{Kleiner-Leeb} (see \S\ref{sec:Buildings} below).  It may be seen as the \catun metric realisation of a combinatorial spherical building (see Lemma~\ref{lem:MetricVsCombinatorial}). Degenerate examples are provided by round spheres (those are the \emph{thin} spherical buildings) and by discrete \catun spaces in which any two points are antipodal (those are the $0$-dimensional spherical  buildings). If $\Isom(X)$ acts  \textbf{minimally}\index{minimal action} on $X$, in the sense that $X$ is the only non-empty closed convex $\Isom(X)$-invariant subset, then the spherical factors of $\bd X$ correspond to the flat factors of $X$, while the $0$-dimensional factors of $\bd X$ correspond  to the Gromov hyperbolic factors of $X$. Theorem~\ref{thm:NotGoedComplete} ensures that any other irreducible factor of $\bd X$ is a spherical building of dimension~$\geq 1$ and satisfies the \textbf{Moufang condition} (see \S\ref{sec:Moufang} below).

Moufang buildings have been completely classified by Tits~\cite{Tits74} and Tits--Weiss~\cite{TitsWeiss} in terms of algebraic data. In particular, the irreducible Moufang buildings appearing at infinity of $X$ are precisely the spherical buildings associated with simple algebraic groups over (possibly Archimedean) local fields (see \S\ref{sec:Moufang} below). It thus follows from Theorem~\ref{thm:NotGoedComplete} that the  isometry type of the \catun visual boundary $\bd X$ runs over a countable family of standard possibilities: this is the classification at infinity.

\medskip
If we impose on $X$ the (rather common) assumption of \textbf{geodesic completeness}\index{geodesic completeness}, i.e.\ that every geodesic segment extends to a bi-infinite geodesic line, then $X$ itself is also restricted to a countable list of standard geometries (up to the unavoidable scalings of the various irreducible factors).

\begin{thmintro}[Classifying the spaces]\label{cor:CoctAmen}
Let $X$ be a  geodesically complete locally compact \cat space without a global fixed point at infinity. Assume that  $X$ admits a cocompact isometric action of an amenable locally compact group. 

Then $X$ is a product of flats, symmetric spaces of non-compact type, Bruhat--Tits buildings and biregular trees.
\end{thmintro}

The key tool in deducing Theorem~\ref{cor:CoctAmen} from Theorem~\ref{thm:NotGoedComplete} is provided by Leeb's work~\cite{Leeb}, which ensures that  a  geodesically complete locally compact \cat space whose boundary is an irreducible metric spherical building of dimension~$\geq 1$ must be a symmetric space or a Euclidean building. 

It cannot be expected that the sharper conclusions of Theorem~\ref{cor:CoctAmen} hold without the hypothesis of  geodesic completeness, even if one assumes that the full isometry group $\Isom(X)$ acts minimally on $X$.  The simplest illustration is provided by the following one-parameter family of ``equivariant deformations'' of the trivalent tree $T$ with edge-length~$1$. Replacing the ball of radius~$0<\vareps\leq 1/2$ around each vertex by an equilateral triangle of side length~$\vareps$, one obtains a proper \cat space $\widetilde G_\vareps$ (the universal cover of a G-string\index{G-string}) which is not geodesically complete. Nonetheless, its isometry group $\Isom(\widetilde G_\vareps)$ is naturally isomorphic to $\Aut(T)$ and acts minimally. In particular $\widetilde G_\vareps$ admits  cocompact action of an amenable group, but no isometric embedding of $T$.  A more fancy-clad construction gives a continuous deformation of the classical hyperbolic plane~\cite{Monod-Py}
 . The latter provides a one-parameter family of (homothety classes of) non-geodesically complete proper cocompact minimal  \catmun spaces whose isometry group is $\mathrm{PGL}_2(\RR)$. Each of those spaces therefore admits a cocompact isometric action of the amenable group $ax +b$.

However, despite this broader diversity of \cat spaces allowed for by Theorem~\ref{thm:NotGoedComplete}, on the level of the full isometry groups we retain the same rigidity as imposed by geodesic completeness on the level of the spaces:

\begin{corintro}[Classifying the groups]\label{cor:ModelSpace}
Under the hypotheses of Theorem~\ref{thm:NotGoedComplete}, there exist:
\begin{itemize}
\item[---] a geodesically complete locally compact \cat  space $X_{\mathrm{model}}$ which is a product of flats, symmetric spaces of non-compact type, Bruhat--Tits buildings and biregular trees,
\item[---] a continuous, proper,  cocompact isometric action of $\Isom(X)$ on $X_{\mathrm{model}}$,
\item[---] an $\Isom(X)$-equivariant isometry $\bd X \cong \bd X_{\mathrm{model}}$.
\end{itemize}
In particular, modulo a compact kernel, $\Isom(X)$ is a closed subgroup of a product of algebraic groups and automorphisms groups of biregular trees.
\end{corintro}

The proof of Theorem~\ref{thm:NotGoedComplete} occupies the major part of this paper. The absence of fixed points at infinity is used once at the first step of the proof, in order to invoke structural results from~\cite{Caprace-Monod_structure}  recalled in Theorem~\ref{thm:structure} below. This provides a reduction to the case  where $\Isom(X)$ is totally disconnected. We then focus on proper \cat spaces admitting a cocompact isometric action of a totally disconnected amenable locally compact group. This part of the work (which represents about half of the proof of Theorem~\ref{thm:NotGoedComplete}) is valid in full generality: global fixed points at infinity are allowed. The main result provides a sharp description of the boundary at infinity; in particular we establish the existence of spherical caps of full dimension in $\bd X$ consisting of points with a cocompact stabiliser (Propositions~\ref{prop:CregNonempty} and~\ref{prop:regular:1}). A posteriori, those spherical 
 caps will be identified with the chambers of the spherical building at infinity. Those general results are then confronted to the absence of fixed points at infinity: using the aforementioned spherical caps, we construct fully maximal spheres all of whose points have a cocompact stabiliser, and then establish, using  a criterion due to Balser--Lytchak~\cite{BaLy_BuildingLikeSpaces} and recalled in Theorem~\ref{thm:BalserLytchak}, that the subset $C$ of $\bd X$ consisting of those  points at infinity whose stabiliser is cocompact, forms a metric spherical building. It is finally shown that $C = \bd X$. 

The last step in the proof of Theorem~\ref{thm:NotGoedComplete} is to establish the Moufang property. This property will be guaranteed by   the following. 

\begin{thmintro}[Classifying thick boundary buildings]\label{thm:Moufang}
Let $X$ be a proper \cat space whose boundary $\bd X$ is a thick irreducible  metric spherical building of dimension~$\geq 1$. If for each $\xi \in \bd X$, the stabiliser $\Isom(X)_\xi$ acts cocompactly on $X$, then  $\bd X$  is Moufang. 

Moreover the image of the natural map $\Isom(X) \to \Aut(\bd X)$ contains all root subgroups. 
\end{thmintro}

We point out the following building-theoretic consequence. 

\begin{corintro}\label{cor:Building}
Let $X$ be a  thick irreducible locally finite Euclidean building of dimension~${\geq 2}$. If   $\Aut(X)$ (including potential non-type-preserving automorphisms) acts cocompactly on $X$ and transitively on the set of chambers of the spherical building $\bd X$, then  $\bd X$ is Moufang. In particular $X$ is a Bruhat--Tits building. 
\end{corintro}

The Moufang condition is automatic for irreducible spherical buildings of dimension~$\geq 2$  (see~\cite[Satz~1]{Tits77} or~\cite[Theorem~11.6]{Weiss03}), so the relevance of Corollary~\ref{cor:Building} is for $2$-dimensional Euclidean buildings. The conclusion of Corollary~\ref{cor:Building} has been proved in the special case of $\tilde A_2$-buildings  by H.~Van Maldeghem and K.~Van Steen~\cite{VanMaldeghem-VanSteen} (under the a priori stronger hypothesis that $X$ admits a type-preserving automorphism group acting \textbf{Weyl-transitively}\index{Weyl-transitive} on $X$ in the sense of~\cite[Definition~6.10]{AbramenkoBrown}). The case of $\tilde C_2$- and $\tilde G_2$-buildings had remained open since then. The proof presented below   is insensitive to the type of the ambient  building.

\subsection*{Acknowledgements}
This work was finalised while the first-named author was spending a trimester at EPFL, whose hospitality was greatly appreciated. We are grateful to Richard Weiss for a useful comment on the classification of locally finite Bruhat--Tits buildings.

\bigskip

\itshape
The following table of contents is provided as a reading guide; an index collects new and classical terminology.
\upshape

\tableofcontents

%%%%%%%%%%%%%%%%%%%%%%%%%%%%%%%%%%%%%%%%%%%%%%%%%%%%%%%%%%%%%%%%%%%%%%%%%%%%%%
\section{Geometric preliminaries}
%%%%%%%%%%%%%%%%%%%%%%%%%%%%%%%%%%%%%%%%%%%%%%%%%%%%%%%%%%%%%%%%%%%%%%%%%%%%%%

We first fix some basic terminology and notation that will be used throughout this paper. We follow the book~\cite{Bridson-Haefliger} to which we refer for further details. 

Let $X$ be a \cat space. The \textbf{visual boundary}\index{visual boundary} of $X$, consisting of asymptote classes of geodesic rays, is denoted by $\bd X$. Given $\xi, \eta \in \bd X$, the \textbf{angular distance}\index{angular distance}\index{distance! angular} between $\xi$ and $\eta$ is defined by 
$$\tangle \xi \eta = \sup_{x \in X} \aangle x \xi \eta,$$
where $\aangle x \xi \eta$ denotes the Alexandrov angle formed by the geodesic rays $[x, \xi)$ and $[x, \eta)$ at the point $x$. The set $\bd X$ endowed with the angular distance is a metric space. The \textbf{Tits distance}\index{Tits distance}\index{distance! Tits}  $\Td$ on $\bd X$ is defined as the length metric associated with the angular distance. The metric space $(\bd X, \Td)$ is called the \textbf{Tits boundary}\index{Tits boundary} of $X$. It is a \catun space which is complete if $X$ is so (see~\cite[II.9.20]{Bridson-Haefliger}); there is a customary abuse of notation here since $\Td$ may take infinite values. For all $\xi, \eta \in \bd X$, we have $\tangle \xi \eta\leq \Td(\xi, \eta )$ by definition; moreover, the equality $ \tangle \xi \eta = \Td( \xi, \eta) $ holds as soon as  $\tangle \xi \eta < \pi$.

Endowing the set of geodesic segments and  rays emanating from a base point $x \in X$ with the topology of uniform convergence on compacta, we obtain a topology on  set $\overline X= X \sqcup \bd X$, which happens to be independent on  the base point $x$. This is called the \textbf{c\^one topology}\index{cone topology@c\^one topology}. If $X$ is proper, the c\^one topology is compact and $\overline X$ is then a compactification of $X$. The c\^one topology generally differs  from the topology induced by the Tits distance. 

Given $\xi \in \bd X$, the \textbf{Busemann function}\index{Busemann function} based at $\xi$ is defined by 
$$B_\xi \colon X \times X \to \RR : (x, y) \mapsto \lim_{t \to \infty} d(r(t), y) - t,$$
where $r\colon \RR_+ \to X$ is the geodesic ray joining $x$ to $\xi$. For all $x, y, z \in X$, we have
$$B_\xi(x, z) = B_\xi(x, y) + B_\xi(y, z).$$
It follows that for any isometry $g \in \Isom(X)$ fixing $\xi$, the number 
$$\beta_\xi(g) = B_\xi(x, gx)$$
is independent of $x$. Moreover the map $\beta_\xi \colon \Isom(X)_\xi \to \RR$ is a homomorphism, called the \textbf{Busemann character}\index{Busemann character} at $\xi$. It is continuous when $X$ is proper and $\Isom(X)$ is endowed with the compact-open topology. 

It is customary to fix a basepoint $x_0 \in X$ and to call the function
$$b_\xi \colon X \to \RR : x \mapsto B_\xi(x_0, x)$$
the \textbf{Busemann function}  at $\xi$. The dependence on the choice of a base point $x_0 \in X$ is thus implicit in that notation.

%%%%%%%%%%%%%%%%%%%%%%%%%%%%%%%%%%%%%%%%%%%%%%%%%%%%%%%%%%%%%%%%%%%%%%%%%%%%%%
\subsection{The interplay of angles and Busemann functions}
%%%%%%%%%%%%%%%%%%%%%%%%%%%%%%%%%%%%%%%%%%%%%%%%%%%%%%%%%%%%%%%%%%%%%%%%%%%%%%

Consider two distinct points $x,y\in X$ and at point at infinity $\xi\in\bd X$.

\begin{lem}[The asymptotic angle formula]\label{lem:asy}\index{asymtotic angle formula}
If $r$ is the geodesic ray pointing towards $\xi$ with $r(0)=x$, then
$$\lim_{t\to\infty}\cos\cangle x{r(t)}y = \frac{b_{\xi}(x) - b_{\xi}(y)}{d(x,y)},$$
where the comparison angle in the limit is a non-decreasing function of $t$.
\end{lem}

\begin{proof}
The angle is non-decreasing by comparison~\cite[II.3.1]{Bridson-Haefliger}. Normalize the Busemann function by $b_\xi(x)=0$. For any given $t$, the cosine law gives
$$\cos\cangle x{r(t)}y = \frac{t^2 + d^2(x,y) - d^2(r(t), y)}{2td(x,y)}.$$
This can be re-written
$$\frac{d^2(x,y) -2t\Big(d(r(t), y)-t\Big) -\Big(d(r(t), y)-t\Big)^2}{2td(x,y)}.$$
Since $d(r(t), y)-t$ converges to $b_\xi(y)$, the statement follows.
\end{proof}

\begin{lem}\label{lem:flat-half}\index{flat!half-strip}
If $b_\xi$ is constant on the segment $[x,y]$, then $x$, $y$ and $\xi$ span a Euclidean half-strip in $X$.
\end{lem}

\begin{proof}
By the ``ideal'' version of the flat triangle lemma, or equivalently by applying the flat quadrilateral lemma~\cite[II.2.11]{Bridson-Haefliger} at points far out on $[x,\xi)$ and $[y,\xi)$, it suffices to prove
$$\aangle xy\xi, \aangle yx\xi\geq\pi/2.$$
For any $z\in [x,\xi)$ and any $t\in[x,y]$ we have
$$d(t,z) \geq |b_\xi(t) - b_\xi(z)| = b_\xi(x) - b_\xi(z) = d(x,z).$$
It follows that the orthogonal projection of $z$ to $[x,y]$ is $x$, so that $\aangle xy\xi = \aangle xyz \geq \pi/2$ as required. The other angle is treated the same way.
\end{proof}

%%%%%%%%%%%%%%%%%%%%%%%%%%%%%%%%%%%%%%%%%%%%%%%%%%%%%%%%%%%%%%%%%%%%%%%%%%%%%%
\subsection{Radial limits of Alexandrov angles}
%%%%%%%%%%%%%%%%%%%%%%%%%%%%%%%%%%%%%%%%%%%%%%%%%%%%%%%%%%%%%%%%%%%%%%%%%%%%%%
Tits angles and Alexandrov angles have various semi-continuity properties. The most basic one is the following:

\begin{lem}\label{lem:lower:semi}
The Tits angle is lower semi-continuous with respect to the c\^one topology in the sense that for any convergent sequences $\zeta_n\to\zeta$ and $\teta_n\to\teta$ in $\bd X$ we have
$$\tangle\zeta\teta \leq \liminf_{n\to\infty} \tangle{\zeta_n}{\teta_n}.$$
\end{lem}

\begin{proof}
See Proposition~II.9.5(2) in~\cite{Bridson-Haefliger}.
\end{proof}

However, we shall repeatedly need a \emph{lower} bound on the Tits angle between limits of points at infinity. This will be provided by Alexandrov angles measured along radial limits, as follows. We  say that a sequence of isometries $(g_n)$ is \textbf{radial}\index{radial sequence} for a point $\zeta \in\bd X$ if for some (hence all) $p \in X$, the sequence $g_n^{-1}p$ converges to $\zeta$ and remains in a bounded neighbourhood of the geodesic ray $[p, \zeta)$.

\begin{remark}\label{rem:radial}
Observe that $(g_n)$ is radial for $\zeta$ if and only if for some (hence any) ray $\ro$ pointing to $\zeta$, there is a sequence $t_n\to+\infty$ such that $g_n \ro(t_n)$ remains bounded.
\end{remark}

As pointed out to us by Eric Swenson, the following result appears in \cite[Lemma~7]{Swenson13}. 

\begin{prop}\label{prop:upper:semi}
Let $\zeta,\teta\in\bd X$ and let $(g_n)$ be a radial sequence for $\zeta$. Let $\zeta', \teta'$ be accumulation points of the sequences $(g_n \zeta)$ and $(g_n \teta)$ respectively. Then we have
$$\tangle{\zeta'}{\teta'} = \aangle x{\zeta'}{\teta'} = \tangle\zeta\teta$$
whenever $x\in X$ is an accumulation point of a sequence $g_n \ro(t_n)$ where $\ro\colon\RR_+\to X$ is a ray pointing to $\zeta$ and $t_n\to+\infty$.
\end{prop}

\begin{proof}
Upon extracting, we assume that the accumulation points under consideration are the limit of the corresponding sequence. By definition, $\tangle{\zeta'}{\teta'} \geq \aangle x{\zeta'}{\teta'}$. The latter is bounded below by
$$\limsup_{n\to\infty} \aangle{g_n \ro(t_n)}{g_n \zeta}{g_n \teta} = \limsup_{n\to\infty} \aangle{\ro(t_n)}{\zeta}{\teta}$$
because of the upper semi-continuity of Alexandrov angles, see Proposition~II.9.2(2) in~\cite{Bridson-Haefliger}. The right hand side is actually a limit and equals $\tangle{\zeta}{\teta}$ by Proposition~II.9.8(2) from~\cite{Bridson-Haefliger}. Finally,
$$\tangle{\zeta}{\teta}=\tangle{g_n \zeta}{g_n \teta} \geq \tangle{\zeta'}{\teta'}$$
by Lemma~\ref{lem:lower:semi}, and hence equality holds everywhere.
\end{proof}

%%%%%%%%%%%%%%%%%%%%%%%%%%%%%%%%%%%%%%%%%%%%%%%%%%%%%%%%%%%%%%%%%%%%%%%%%%%%%%
\subsection{Joins and links}\label{sec:JoinsLinks}
%%%%%%%%%%%%%%%%%%%%%%%%%%%%%%%%%%%%%%%%%%%%%%%%%%%%%%%%%%%%%%%%%%%%%%%%%%%%%%
We use the symbol $\circ$ for the operation of \textbf{spherical join}\index{spherical join}\index{join} of two metric spaces (see~\cite[I.5.13]{Bridson-Haefliger}). The following criterion allows to recognize when a subset of a \catun space decomposes as a spherical join. By convention, a subset of a \catun space is called \textbf{convex}\index{convex} if it is $\pi$-convex, i.e. if for any two points at distance~$< \pi$, the  geodesic segment joining them is contained in the set.

\begin{lem}\label{lem:Join}
Let $Z$ be a \catun space and  $Z_1, Z_2 \subset Z$ be closed convex subsets such that $d(z_1, z_2) = \pi/ 2$ for all $(z_1, z_2)   \in Z_1 \times Z_2$. Then  the natural map $f \colon Z_1 \circ Z_2 \to Z$ sending the geodesic from $z_1$ to $z_2$ isometrically to the geodesic segment $[z_1, z_2]$ in $Z$  for all $(z_1, z_2)    \in Z_1 \times Z_2$  is an isometric embedding.
\end{lem}

\begin{proof}
See~\cite[Lemma~4.1]{Lytchak_RigidityJoins}.
\end{proof}

Given a  pair $\{\xi, \xi'\}$ in a \catun space $Z$ with $d(\xi, \xi') \geq \pi$, we define its \textbf{link}\index{link} as
$$\Link(\xi, \xi')  = \{z \in Z \mid d(z, \xi) = d(z, \xi') = \frac \pi 2\}.$$
It is a closed convex subset, since it coincides with the intersection of the two closed balls of radius~$\pi/ 2$ respectively centred at $\xi$ and $\xi'$. We also define the set\index{$Pi$@$\Pi(\xi, \xi')$}  
$$\Pi(\xi, \xi') $$
as the union of $\{\xi, \xi'\}$ and all geodesic segments of length $\pi$ joining $\xi$ to $\xi'$. Notice that if $d(\xi, \xi') > \pi$, then $\Link(\xi, \xi')  = \varnothing$ and $\Pi(\xi, \xi') = \{\xi, \xi'\}$.

\begin{lem}\label{lem:Links}
Let $Z$ be a \catun space. Given $\xi, \xi' \in Z$ with $d(\xi, \xi')\geq  \pi$, the set $\Pi(\xi, \xi') $   is closed and convex, and decomposes as the spherical join $\{\xi, \xi'\}  \circ \Link(\xi, \xi')$.
\end{lem}
\begin{proof}
The fact that $\Pi(\xi, \xi') $  is convex and isometric to the spherical join $\{\xi, \xi'\}  \circ \Link(\xi, \xi')$ is a consequence of Lemma~\ref{lem:Join}. Its closedness follows since $\Link(\xi, \xi')$ is closed.
\end{proof}

When $Z$ is the visual boundary of a \cat space $X$, the link of a pair of opposite points has another interpretation. In order to formulate it, let us recall following~\cite{Caprace-Monod_structure}  that a point $\xi'
\in \bd X$ is \textbf{opposite}\index{opposite}\index{boundary point!opposite}\index{point at infinity!opposite}  to $\xi$ if there is a geodesic
line in $X$ whose extremities are $\xi$ and $\xi'$.  We say that
$\xi' $ is \textbf{antipodal}\index{antipodal}\index{boundary point!antipodal}\index{point at infinity!antipodal}  to $\xi$ if $\tangle \xi { \xi'} =
\pi$ or, equivalently, if $\Td(\xi, \xi') \geq \pi$. Thus the set $\Op(\xi)$\index{$Op$@$\Op(\xi)$} of points opposite to $\xi$ is contained in the set $\Ant(\xi)$\index{$Ant$@$\Ant(\xi)$} consisting of all antipodes of $\xi$.

Given $\xi' \in \Op(\xi)$, we denote by\index{$P$@$P(\xi, \xi')$}   
$$P(\xi, \xi')$$
the union of all geodesic lines joining $\xi$ to $\xi'$. Recall that $P(\xi, \xi')$ is a closed convex subset of $X$; moreover there is a canonical isometric identification
\begin{equation}\label{eq:P}
P(\xi, \xi') \cong \RR \times Y
\end{equation}
for some  \cat space $Y$ which is complete if $X$ is so (see~\cite[II.2.14]{Bridson-Haefliger}), where the boundary of the line factor is precisely the pair $\{\xi, \xi'\}$. In particular, this decomposition of $P(\xi, \xi')$ yields a canonical isometric embedding of $\bd Y$ into $\Link(\xi, \xi')$. It turns out that this embedding is surjective, hence an isometry.

\begin{lem}\label{lem:P=Pi}
Let $X$ be a \cat space. Given two opposite points  $\xi, \xi' \in \bd X$, we have 
$$\Pi(\xi, \xi') = \bd P(\xi, \xi').$$ 
In particular  $\Link(\xi, \xi')$ is canonically isometric to the visual boundary of the factor $Y$ in the decomposition (\ref{eq:P}).
\end{lem}

\begin{proof}
Let $\ell \colon \RR \to X$ be a geodesic line  joining $\xi$ to $\xi'$. 

If $\Td(\xi, \xi') > \pi$, then $P(\xi, \xi')$ is contained in a bounded neighbourhood of $\ell(\RR)$ by~\cite[Prop.~II.9.21]{Bridson-Haefliger}, and the result is clear. 

If $\Td(\xi, \xi') = \pi$, let $c \colon [0, \pi] \to \bd X$ be a geodesic segment from $\xi$ to $\xi'$. Let $\mu = c(\pi/2)$ be its midpoint. For any $t \in \RR$, we have $\aangle{\ell(t)} {\xi} \mu \leq \tangle  {\xi} \mu = \pi/2$ and analogous inequalities for  $\xi'$. On the other hand we have $\aangle{\ell(t)} {\xi} \mu  + \aangle{\ell(t)} {\xi'} \mu = \pi$. Hence $\aangle{\ell(t)} {\xi} \mu = \tangle  {\xi} \mu$ and  $\aangle{\ell(t)} {\xi'} \mu = \tangle  {\xi'} \mu$. It then follows from~\cite[Cor.~II.9.9]{Bridson-Haefliger} that the convex hull of $[\ell(t), \xi)$ (resp. $[\ell(t), \xi')$) and $[\ell(t), \mu)$ is a flat sector. Since this holds for all $t$, we infer that $\ell(\RR)$ bounds a flat half-plane $H$ whose visual boundary coincides with the geodesic segment $c([0, \pi])$. Since $H \subseteq P(\xi, \xi')$, it follows that $\Pi(\xi, \xi')  \subseteq \bd P(\xi, \xi')$. The reverse inclusion is clear from the product decomposition  (\ref{eq:P}).
\end{proof}

%%%%%%%%%%%%%%%%%%%%%%%%%%%%%%%%%%%%%%%%%%%%%%%%%%%%%%%%%%%%%%%%%%%%%%%%%%%%%%
\subsection{Lunule and parachute}
%%%%%%%%%%%%%%%%%%%%%%%%%%%%%%%%%%%%%%%%%%%%%%%%%%%%%%%%%%%%%%%%%%%%%%%%%%%%%%

\begin{flushright}
\begin{minipage}[t]{0.5\linewidth}\itshape\small
Tu bouscules, tu circules,\\
O ma lune majuscule,\\
Tu m'\'ecules, tu m'annules\\
Je ne suis que ta lunule.\\

\hfill\upshape (Norge, \emph{La Langue verte}, 1954)
\end{minipage}
\end{flushright}

The following results  on \catun geometry will play a crucial role.

\begin{lem}[Lunule Lemma]\index{lunule}
Let $Z$ be a \catun space and $\zeta, \xi \in Z$ be a pair of points with $d_Z(\zeta, \xi) = \pi$, and let $\sigma, \sigma'$ be two geodesic segments joining $\zeta$ to $\xi$. If the angle $\alpha_1$ formed by $\sigma $ and $\sigma'$ at $\zeta$ is strictly smaller than $\pi$, then the angle they form at $\xi$ equals $\alpha_1$ and the geodesic biangle $\sigma \cup \sigma'$ bounds a spherical lunule. 
\end{lem}
\begin{proof}
See~\cite[Lemma~2.5]{BallmannBrin_Duke99}.
\end{proof}

Two points $\xi_1, \xi_2$ of a \catun space $Z$ are called \textbf{antipodal}\index{antipodal} if $d(\xi_1, \xi_2)\geq \pi$. When $Z$ is the Tits boundary of a \cat space, this definition coincides with the one given above. The \textbf{dimension}\index{dimension} of a \catun space  is the geometric dimension in the sense of Kleiner~\cite{Kleiner}. The Tits boundary $\bd X$ of a proper \cat space $X$ with a cocompact isometry group is finite-dimensional; its dimension coincides with the largest dimension of a \textbf{round sphere}\index{sphere!round} (i.e.  a unit sphere of a Euclidean space equipped with the angular metric) in $\bd X$ which bounds a Euclidean flat in $X$ (see~\cite[Theorem~C]{Kleiner}). 

\begin{lem}\label{lem:AntipodalPt}
Let $Z$ be a finite-dimensional \catun space and let $\xi \in Z$. Given a round sphere $S \subset Z$ with $\dim(S) = \dim(Z)$ and  any point $\eta \in S$ with $d(\xi, \eta) < \pi$, there is $\eta' \in S$ antipodal to $\xi$ and a  geodesic segment from $\xi$ to $\eta'$ passing through $\eta$.  In particular, $S$ contains a point antipodal to $\xi$.
\end{lem}
\begin{proof}
See~\cite[Lemma~3.1]{BalserLytchak_Centers}.
\end{proof}

\begin{lem}[Parachute Lemma]\label{lem:para}\index{parachute lemma}
Let $Z$ be a finite-dimensional \catun space and $\zeta, \xi \in Z$ be a pair of points with $d_Z(\zeta, \xi) = \pi$. Let also $U$ be a spherical cap of dimension $\dim Z$ containing $\xi$ in its interior.

If $d_Z(z, \zeta) < \pi$ for all $z \in U \setminus \{\xi\}$, then there is a round sphere $S$ with $\dim S = \dim Z$ such that $S$ contains  $\zeta$ and that the spherical cap $S \cap U$ still contains  $\xi$ in its interior. 
\end{lem}

\begin{proof}
Let $V \subset U$ be the subset consisting of all those points $z \in U$ such that $d_Z(\xi, z) + d_Z(z , \zeta) = \pi$. Thus $\xi \in V$. We claim that $V$ contains a neighbourhood of $\xi$ in $U$. Indeed, suppose the contrary. Then there exist points in $ U \setminus V$ which are arbitrary close to $\xi$. Pick such a point $z$ such that the ball of radius $d_Z(z, \xi)$ around $z$ is entirely contained in $U$. By Lemma~\ref{lem:AntipodalPt}, the geodesic segment joining $\zeta$ to $z$ can be locally prolonged in $U$. Since the ball of radius $d_Z(z, \xi)$ entirely lies in $U$, this prolongation can be extended until it hits some antipode $x$ of $\zeta$, with $x \in U$. Notice that $x \neq \xi$ since otherwise we would have $z \in V$. Therefore $x$ is a point of $U \setminus \{\xi\}$ with $d_Z(\zeta, x) = \pi$, which contradicts the hypotheses. The claim stands proven. 

Let now $S$ be the union of all geodesic segments joining $\xi$ to $\zeta$ and containing some point of $V \setminus \{\xi\}$. It now follows from the Lunule Lemma that $S$ is isometric to a round sphere. Moreover, since $S$ contains an open subset of $U$ we have $\dim S = \dim U  = \dim Z$. 
\end{proof}

%%%%%%%%%%%%%%%%%%%%%%%%%%%%%%%%%%%%%%%%%%%%%%%%%%%%%%%%%%%%%%%%%%%%%%%%%%%%%%
\subsection{Tricycles}
%%%%%%%%%%%%%%%%%%%%%%%%%%%%%%%%%%%%%%%%%%%%%%%%%%%%%%%%%%%%%%%%%%%%%%%%%%%%%%

A \textbf{tricycle}\index{tricycle} of dimension $d$ is the \catun space obtained by gluing three unit $d$-dimensional closed hemispheres along their boundary equators. Alternatively, it can be defined as the spherical join
$$S \circ \{1, 2, 3\} $$
of a unit $(d-1)$-dimension sphere $S$ with a discrete set of three elements that are pairwise at distance~$\pi$ from one another. The sphere $S$, viewed as a subset of the tricycle $S \circ \{1, 2, 3\} $, is called its \textbf{equator}\index{equator!of a tricycle}\index{tricycle!equator}. It is an equator of each of the three top-dimensional spheres contained in the tricycle and coincides with the intersection of those three spheres.

Tricycles may be recognised by applying the following criterion in the case $n=3$.

\begin{lem}\label{lem:criterion:tricycle}
Let $Z$ be a finite-dimensional complete \catun space  and let $S_1, \dots, S_n \subset Z$ be  round spheres with $\dim S_i   = \dim Z > 0$ for all $i$. For each $i$, let  $H_i \subset S_i$ be a hemisphere with centre $s_i$. Assume that the  $H_i$ are pairwise distinct and have a common boundary equator, say $E$. Then $d(s_i, s_j)= \pi $ for all $i \neq j$, and  
 $H_1 \cup \dots \cup H_n$ is isometric to the spherical join  $\{s_1, \dots, s_n\} \circ E$.  
\end{lem}
\begin{proof}
Let $i \neq j$. Every point of $S_i \cup S_j$ is at distance at most~$\pi/2$ from the common boundary equator $E$ of $H_i$ and $H_j$. Moreover $s_i$ and $s_j$ are both at distance exactly $\pi/2$ from any point of $E$.   In particular $d(s_i, s_j) \leq \pi$ and $d(s_i, u) < \pi$ for all $u \in H_j \setminus \{s_j\}$. By hypothesis $H_i \neq H_j$, hence $s_i \neq s_j$. If $d(s_i, s_j) < \pi$, then Lemma~\ref{lem:AntipodalPt} ensures that there exists a geodesic segment  $\gamma$ (in $Z$) of length $\pi$ containing $s_j$ and joining $s_i$ to   a point $u \in S_j \setminus \{s_j\}$ which is antipodal to $s_i$. We must have  $u \in H_2$, since otherwise  the geodesic $[s_j, u]$ would meet the equator $E$ in some point $p$. Since $[s_i, s_j] \cup [s_j, p]$ is a subsegment of the geodesic $\gamma$ we would infer that 
$$\frac \pi 2 = d(s_i, p) = d(s_i, s_j) + d(s_j, p) = d(s_i, s_j) + \frac \pi 2,$$
contradicting that $s_i \neq s_j$. Thus we have constructed a point $u \in H_2 \setminus \{s_j\}$ with $d(s_i, u) = \pi$, a contradiction. This proves that $d(s_i, s_j) =\pi$.

This  implies that the set $\{s_1, \dots,  s_n\}$ is a closed $\pi$-convex subset of $Z$, each of whose points lies at distance $\pi/2$ from each point of $E$. Therefore there is a natural surjective map from the spherical join $\{s_1, \dots,  s_n\} \circ E$ to $H_1 \cup \dots \cup H_n$. This map  is an isometry by   Lemma~\ref{lem:Join}.
\end{proof}

%%%%%%%%%%%%%%%%%%%%%%%%%%%%%%%%%%%%%%%%%%%%%%%%%%%%%%%%%%%%%%%%%%%%%%%%%%%%%%
\subsection{A de Rham decomposition}
%%%%%%%%%%%%%%%%%%%%%%%%%%%%%%%%%%%%%%%%%%%%%%%%%%%%%%%%%%%%%%%%%%%%%%%%%%%%%%

A \catun space is called  \textbf{irreducible}\index{irreducible!\catun space} if it does not admit any non-trivial decomposition as a spherical join. 
It is a general fact that the boundary   of a product \cat space $X = X_1 \times X_ 2$ is the spherical join $\bd X \cong \bd X_1 \circ \bd X_2$. In particular $\bd X$ is not irreducible if $\bd X_1$ and $\bd X_2$ are both non-empty. It is natural to expect that this statement has a converse: a join decomposition of the \catun boundary of $X$ should come from a product decomposition. This cannot hold in full generality without any further assumption on $X$ (one may destroy the product structure of a \cat space without altering the visual boundary by growing hair on $X$). However, the statement does hold if one assumes that $X$ is geodesically complete, see~\cite[Theorem~II.9.4]{Bridson-Haefliger}. For our purposes, we need the following variation. 

\begin{prop}\label{prop:JoinDec}
Let $X$ be a proper \cat space such that $\Isom(X)$ acts cocompactly and minimally. 
Given a spherical join decomposition $\bd X = Z_1 \circ Z_2$, there is a \cat product decomposition $X = X_1 \times X_2$ such that $\bd X_1 = Z_1$ and $\bd X_2 = Z_2$.
\end{prop}

\begin{proof}
See~\cite[Proposition~III.10]{PCMI}.
\end{proof}

Structural properties of proper cocompact \cat spaces and their full isometry groups  have been established in~\cite{Caprace-Monod_structure}. The following summaries some of them; this will allow us in due course to reduce the proof of the Main Theorem to the case where the ambient space is irreducible and has a totally disconnected isometry group.

\begin{thm}\label{thm:structure}\index{de Rham!decomposition}
Let $X$ be a proper \cat space whose isometry group acts cocompactly without a fixed point at infinity. Then there is a canonical minimal $\Isom(X)$-invariant closed convex subset $X'$ with $\bd X' = \bd X$ admitting a canonical product decomposition
$$X'  \cong X_1 \times \dots \times X_p \times \RR^n \times Y_1 \times \cdots
\times Y_q,$$
invariant under $\Isom(X')$ up to permutations of isometric factors, and whose factors satisfy the following.  
\begin{itemize}

\item For each $i$, the \cat space $X_i$ is irreducible, its isometry group $\Isom(X_i)$ is a simple Lie group acting cocompactly on $X_i$. Moreover its visual boundary $\bd X_i$ is equivariantly isometric to the spherical building of $\Isom(X_i)$.

\item For each $j$, the \cat space $Y_j$ is irreducible, its isometry group $\Isom(Y_j)$ is totally disconnected and acts minimally,  cocompactly, without a fixed point at infinity. Moreover its visual boundary $\bd Y_j$ is an  irreducible  \catun space. 
  \end{itemize}
\end{thm}

\begin{proof}
All the statements are provided by~\cite[Theorem~1.6 and Addendum~1.8]{Caprace-Monod_structure}, except those concerning the visual boundaries of the factors. The fact that  $\bd X_i$ is isometric to the spherical of the simple Lie group $\Isom(X_i)$ follows from~\cite[Theorem~7.4]{Caprace-Monod_structure}. 
The fact that $\bd Y_j$ is an irreducible \catun space follows from Proposition~\ref{prop:JoinDec}.
\end{proof}

\subsection{Multiple antipodes}

\begin{flushright}
\begin{minipage}[t]{0.75\linewidth}\itshape\small
\ldots le terme d'antipodes trouvait dans ma pens\'ee un sens plus riche et plus na\"if que son contenu litt\'eral.
\upshape
\begin{flushright}
(Cl. L\'evi-Strauss, \emph{Tristes Tropiques}, 1955)
\end{flushright}
\end{minipage}
\end{flushright}

The results of this section and of the following one are irrelevant to the proof of Theorems~\ref{thm:NotGoedComplete} and~\ref{cor:CoctAmen}, but they will be used in the proof of Corollary~\ref{cor:ModelSpace}. 

Let $Z$ be a  finite-dimensional \catun space. Applying the de Rham
decomposition Theorem from~\cite{FoertschLytchakGAFA} to the Euclidean
c\^one over $Z$, we infer that $Z$ has a canonical join decomposition $Z = S \circ Z'$, invariant under $\Isom(Z)$, 
such that $S$ is the unique maximal round sphere factor of $Z$. That factor $S$ is called the  \textbf{spherical de Rham factor}\index{de Rham!spherical factor}\index{spherical de Rham factor} of $Z$. When $Z$ is the Tits-boundary of a proper \cat space whose isometry group $\Isom(X)$ acts cocompactly and minimally, the spherical de Rham factor of $Z = \bd X$ coincides with the boundary of the maximal Euclidean de Rham factor of $X$ (e.g.\ by Proposition~\ref{prop:JoinDec}).

\smallskip
It turns out that antipodes are in richer supply as soon as we leave this spherical world:

\begin{prop}\label{prop:ManyAntipodes}
Let $X$ be a proper cocompact \cat space. A point at infinity has a unique antipode if and only if it belongs to the spherical de Rham factor of $\bd X$. 
\end{prop}

We first record the existence of antipodes.

\begin{lem}
Any point at infinity $\xi$ of a proper cocompact \cat space $X$ has at least one antipode, which can moreover be chosen opposite to $\xi$.
\end{lem}

\begin{proof}
The existence follows from Lemma~\ref{lem:AntipodalPt} because the Tits-boundary $\bd X$ is finite-dimensional and contains a round sphere $S$ with $\dim(S) = \dim(\bd X)$, see~\cite[Theorem~C]{Kleiner}. For the stronger statement, let $r\colon\RR_+\to X$ by a geodesic ray pointing towards $\xi$. For each $n\in\NN$ there is a geodesic ray $r_n$ with $r_n(0)=r(n)$ and passing within uniformly bounded distance of $r(0)$, see Corollary~3 in~\cite{Geoghegan-Ontaneda}. Thus the sequence of $1$-Lipschitz maps $t\mapsto r_n(t-n)$ remains in compact sets of $X$ for $t$ in any bounded interval and hence has a limit point. This limit point is a bi-infinite geodesic line with one end pointing to $\xi$, as needed.
\end{proof}

\begin{proof}[Proof of Proposition~\ref{prop:ManyAntipodes}]
If $\bd X$ decomposes as a join $\bd X \cong Z \circ Z'$ with $Z$ a round sphere, then any point $\xi \in Z$ has a unique antipode in $\bd X$, namely its antipode in $Z$. Assume conversely that a point $\xi \in \bd X$ has a unique antipode in $\bd X$, say $\xi'$. Any sphere $S \subseteq \bd X$ with $\dim(S) = \dim(\bd X)$ must contain $\xi'$ by uniqueness and Lemma~\ref{lem:AntipodalPt}. Let now $\xi'' \in S$ be the unique antipode of $\xi'$ in $S$. Since  $\xi$ and $\xi''$ are not antipodal, we have $\Td(\xi, \xi'')< \pi$ and it follows from Lemma~\ref{lem:AntipodalPt} that the geodesic segment from $\xi$ to $\xi''$ can be prolonged in $S$ until it reaches an antipode of $\xi$, which has to be $\xi'$. Thus $\Td(\xi, \xi'') = \Td(\xi, \xi')- \Td(\xi', \xi'') = 0$, so that $\xi = \xi''$. This shows that both $\xi$ and $\xi'$ belong to every round sphere $S\subseteq \bd X$ with $\dim(S) = \dim(\bd X)$. 

The intersection of all those round spheres is a convex Tits-compact subset of $\bd X$ which is invariant under $\Isom(X)$. It follows that any sequence $(g_n) \subset \Isom(X)$ has a subsequence $(g_{n'})$ such that $(g_{n'}\xi)$ and $(g_{n'}\xi')$ converge to an antipodal pair. The desired conclusion now follows from a similar argument as in the proof of~\cite[Lemma~4.17]{GuralnikSwenson}. Note however that we cannot invoke the latter result directly, since it is stated under the hypotheses that $X$ admits a \emph{discrete} group of isometries acting cocompactly, and containing no rank one isometry. Thus, for the sake of completeness, we include the details. Our goal is to prove that $\bd X$ admits a join decomposition of the form $\bd X \cong \{\xi, \xi'\} \circ Z$.  To this end, it suffices by Lemma~\ref{lem:Links} to prove that $\bd X = \Pi(\xi, \xi')$, i.e. that every point $\zeta \in \bd X$ lies on a geodesic of length $\pi$ joining $\xi$ to $\xi'$. If that were not the case, there would be a point $\zeta \in \bd X$ with $\Td(\xi, \zeta) + \Td(\xi', \zeta)> \pi$. In particular $\zeta \not \in \{\xi, \xi'\}$. We then choose a sequence $(g_n) \subset \Isom(X)$ which is radial for $\zeta$ and assume, after extraction, that for some (hence all) $x \in X$, the sequence $(g_n x)$ converges to a point $\zeta' \in \bd X$. A further extraction ensures that $(g_n \xi)$ and $(g_n \xi')$ converge to an antipodal pair $\eta, \eta'$. 
The $\pi$-convergence Theorem of Papasoglu--Swenson~\cite[Lemma~19]{PapasogluSwenson} then implies that $\Td(\zeta', \eta) \leq \max\{0, \pi-\Td(\xi, \zeta)\}$ and $\Td(\zeta', \eta')\leq \max\{0, \pi - \Td(\xi', \zeta)\}$. We infer that 
$$\Td(\eta, \eta') \leq \Td(\eta, \zeta') + \Td(\zeta', \eta')  < \pi,$$
contradicting that $\eta$ and $\eta'$ are antipodal. This proves (ii). 
\end{proof}

We point out that the $\pi$-convergence used above (and again for Proposition~\ref{prop:BoundaryAction} below) is stated in~\cite{PapasogluSwenson} for discrete groups acting properly; however, its proof does not use this assumption.

\subsection{Properness of the action at infinity}

\begin{prop}\label{prop:BoundaryAction}
Let $X$ be a proper cocompact \cat space with empty spherical de Rham factor in $\bd X$. Endow $\bd X$ with the c\^one topology and $\mathrm{Homeo}(\bd X)$ with the compact-open topology. 

Then then canonical continuous homomorphism from $\Isom(X)$ to $\mathrm{Homeo}(\bd X)$ has compact kernel and closed image. 
\end{prop}

We begin with a general observation.

\begin{lem}\label{lem:faithful}
Let $Y$ be a proper \cat space with trivial Euclidean de Rham factor and finite-dimensional Tits boundary. If $\Isom(Y)$ act minimally on $Y$, then it acts faithfully on $\bd Y$.
\end{lem}

\begin{proof}
According to Proposition~1.5(i) in~\cite{Caprace-Monod_structure}, the space $Y$ is \emph{boundary-minimal} in the sense that $Y$ is the only closed convex subspace $Z\subseteq Y$ with $\bd Z=\bd Y$. Let $g \in \Isom(Y)$ be an isometry acting trivially on $\bd Y$. Then the displacement function $z\mapsto d(gz, z)$ is bounded on any geosedic ray, hence non-increasing on rays by convexity. Choosing any $y\in Y$, it follows that the closed convex set $Z = \{z \in Y \mid d(gz,z) \leq d(gy,y)\}$ has full boundary and hence $Z=Y$. Now Lemma~3.9(ii) in~\cite{Caprace-Monod_structure} implies that the displacement function of $g$ is constant on $Y$. Thus $g$ is a Clifford translation, and must be trivial since the Euclidean de Rham factor of $Y$ is trivial by hypothesis.
\end{proof}

\begin{proof}[Proof of Proposition~\ref{prop:BoundaryAction}]
Since $\Isom(X)$ acts cocompactly, there exists a minimal $\Isom(X)$-invariant closed convex subset $Y \subseteq X$. The kernel of the $\Isom(X)$-action on $Y$ is compact and we have $\bd X = \bd Y$. Thus we can apply Lemma~\ref{lem:faithful} to $Y$ and deduce the compactness of the kernel.

Assume now that the image of the homomorphism is not closed. Then there exists a sequence $(g_n)$ in $\Isom(X)$ leaving every compact subset, and such that $(g_n)$ converges in $\mathrm{Homeo}(\bd X)$ to some $\varphi$. Upon extracting, we may assume that for some (hence all) $x \in X$, the sequence $(g_n\inv x)$ converges to a point $\xi \in \bd X$. By Proposition~\ref{prop:ManyAntipodes}, the point $\xi$ has at least two antipodes, say $\eta_1 \neq\eta_2$. The $\pi$-convergence Theorem of Papasoglu--Swenson~\cite[Lemma~19]{PapasogluSwenson}  implies that the converging sequences $(g_n \eta_1)$ and $(g_n \eta_2)$ have the same limit in $\bd X$, which contradicts that $\varphi = \lim_n g_n \in \mathrm{Homeo}(\bd X)$ is injective.
\end{proof}

%%%%%%%%%%%%%%%%%%%%%%%%%%%%%%%%%%%%%%%%%%%%%%%%%%%%%%%%%%%%%%%%%%%%%%%%%%%%%%
\section{Spherical buildings}
%%%%%%%%%%%%%%%%%%%%%%%%%%%%%%%%%%%%%%%%%%%%%%%%%%%%%%%%%%%%%%%%%%%%%%%%%%%%%%

The goal of this section, that the experts may wish to skip, is to review of  basic definitions and facts on spherical buildings.  

%%%%%%%%%%%%%%%%%%%%%%%%%%%%%%%%%%%%%%%%%%%%%%%%%%%%%%%%%%%%%%%%%%%%%%%%%%%%%%
\subsection{Metric and combinatorial buildings}\label{sec:Buildings}
%%%%%%%%%%%%%%%%%%%%%%%%%%%%%%%%%%%%%%%%%%%%%%%%%%%%%%%%%%%%%%%%%%%%%%%%%%%%%%

The spherical buildings appearing in this paper are  \emph{metric} spherical buildings, as defined by Kleiner--Leeb~\cite[\S 3.2]{Kleiner-Leeb}. In the following, we briefly recall the basic notions as well as some facts relevant to our purposes. We refer to loc.~cit. for any further detail.

Let $S$ be a round sphere. An \textbf{equator}\index{equator} of $S$ is a subsphere of codimension~$1$. A \textbf{reflection}\index{reflection} of $S$ is an isometry of order~$2$ fixing pointwise an equator. A \textbf{spherical Coxeter complex}\index{spherical Coxeter complex} is a pair $(S, W)$ consisting of a round sphere $S$ and a finite subgroup $W \leq \Isom(S)$ generated by reflections, and called the \textbf{Weyl group}\index{Weyl group}\index{group!Weyl}. A \textbf{wall}\index{wall} of $(S, W)$ is the fixed equator of a reflection belonging to $W$. Each connected component of the complement in $S$ of the union of all walls is an open convex subset called a \textbf{chamber}\index{chamber}. The group $W$ acts freely and transitively on the set of chambers (except in the trivial case where $\dim S = 0$ and $W = \{1\}$, in which $W$ has two orbits of chambers).

A \catun space $Z$ is called a \textbf{metric spherical building}\index{spherical building!metric} modelled on a spherical Coxeter complex $(S,W)$ if there exists a collection $\mathscr A$, called  \textbf{atlas}\index{atlas},  of isometric emdeddings $\iota \colon S \to Z$ satisfying the following  axioms. 
\begin{description}
\item[(SB0)] For all $\iota \in \mathscr A$ and $w \in W$, we have $\iota \circ w \in \mathscr A$. 

\item[(SB1)] For each pair $z_1, z_2 \in Z$ there exists $\iota \in \mathscr A$ whose image contains both $z_1$ and $z_2$. 

\item[(SB2)] For all $\iota_1, \iota_2 \in \mathscr A$ with respective images $A_1, A_2$, the map $\iota_1\inv \circ \iota_2$ defined on $\iota_2\inv(A_1 \cap A_2) \subset S$ is the restriction of an element of $W$. 
\end{description}

The image of an element of $\mathscr A$ is called an \textbf{apartment}\index{apartment}. By~\cite[Corollary~3.5.2]{Kleiner-Leeb},  each top-dimensional sphere in a metric spherical building is an apartment. A \textbf{chamber}\index{chamber} (resp. a \textbf{wall})  in a metric spherical building is the image of a chamber (resp. wall) under an element of $\mathscr A$. 
Two chambers are called \textbf{adjacent}\index{chamber!adjacent} if their intersection is of codimension~$\leq 1$. A \textbf{panel}\index{panel} is the intersection of two distinct adjacent chambers. 
A \textbf{half-apartment}\index{half-apartment}\index{apartment!half-} is a hemisphere whose boundary equator is a wall. The building $Z$ is called \textbf{thick}\index{spherical building!thick}\index{thick spherical building} if every wall is contained in at least~$3$ half-apartments.

\begin{lem}\label{lem:Chamber=Intersection}
In a metric spherical building, if a point $\xi$ is contained in  the interior of a chamber, then it is contained in a unique chamber. 
If moreover the building is thick, that chamber coincides with the intersection of all apartments containing $\xi$. 
\end{lem}

\begin{proof}
The first assertion follows from~\cite[Lemma~3.4.2]{Kleiner-Leeb}. For the second, we first remark that in a spherical Coxeter complex, the chamber containing a given regular point is the intersection of all half-apartments containing it. Hence the same holds in any spherical building. Therefore it suffices to show that each half-apartment in a thick spherical building is an intersection of apartments. This can be established as follows. Given a half-apartment $H_0$ in a thick building, there exist two other half-apartments $H_1, H_2$, each sharing its boundary wall with $H_0$. By Lemma~\ref{lem:criterion:tricycle}, the union $A_i = H_0 \cup H_i$ is isometric to a sphere, hence it is an apartment. The intersection $A_1 \cap A_2$ is a proper convex subset of a sphere containing the hemisphere $H_0$, whence $A_1 \cap A_2 = H_0$. 
\end{proof}

We say that a spherical Coxeter complex $(S, W)$ is \textbf{irreducible}\index{irreducible!Coxeter complex}\index{spherical Coxeter complex!irreducible} if its chambers are simplices of diameter~$<\pi/2$. A spherical building is called \textbf{irreducible}\index{irreducible!building}\index{spherical building!irreducible} if its model Coxeter complex is irreducible. 

\begin{lem}\label{lem:IrredBuilding}
Every spherical building $Z$ admits a canonical decomposition as the join $Z = Z_0 \circ Z_1 \circ \dots \circ Z_p$, where $Z_0$ is a  (possibly empty) round sphere and  $Z_i$ is  a thick irreducible spherical building for each $i>0$. 

In particular, if $Z$ is irreducible as a \catun space, then $Z$ is a thick irreducible spherical building or a $0$-dimensional sphere.
\end{lem}
\begin{proof}
Let $(S, W)$ be the model Coxeter complex and $\mathscr A$ be the atlas of $Z$. 
We first apply the \emph{thick reduction} from~\cite[Proposition~3.7.1]{Kleiner-Leeb} to $Z$: this provides a canonical way to replace $W$ be a subgroup $W'$ and $\mathscr A$ by a subset $\mathscr A'$ in such a way that $Z$ equipped with $\mathscr A'$ is a thick building modelled on $(S, W')$. 
By~\cite[Proposition~3.3.1]{Kleiner-Leeb} and the discussion following its proof, every  spherical building $Z$ admits a canonical decomposition as the join $Z = Z_0 \circ Z_1 \circ \dots \circ Z_p$, where $Z_0$ is a  (possibly empty) round sphere and  $Z_i$ is  an irreducible spherical building for each $i>0$. In a product building, the Weyl group is the product of the Weyl groups of the factors; therefore the walls of the model Coxeter complex is the union of the walls of the factors. This implies that a product building is thick if and only if each of its factors is thick. So is thus each $Z_i$  (in particular the Weyl group of the spherical factor $Z_0$ is trivial). The first assertion follows.

Assume now that $Z$ is irreducible  as a \catun space. Then $Z$ is an irreducible spherical building or a $0$-dimensional round sphere and the second assertion follows.
\end{proof}

The canonical spherical factor $Z_0$ afforded by Lemma~\ref{lem:IrredBuilding} coincides with the {spherical de Rham factor} of $Z$.

%%%%%%%%%%%%%%%%%%%%%%%%%%%%%%%%%%%%%%%%%%%%%%%%%%%%%%%%%%%%%%%%%%%%%%%%%%%%%%
%\subsection{Combinatorial structure}
%%%%%%%%%%%%%%%%%%%%%%%%%%%%%%%%%%%%%%%%%%%%%%%%%%%%%%%%%%%%%%%%%%%%%%%%%%%%%%

The definition of metric spherical buildings recalled above is due to Kleiner--Leeb and differs from the original definition of Tits, which is more combinatorial. We will be led to apply classical results from the literature on buildings in a metric context, and therefore need to clarify the connection between the two notions. A spherical building in the sense of~\cite[Definition~4.1]{AbramenkoBrown} will be called a \textbf{combinatorial spherical building}\index{spherical building!combinatorial}.

\begin{lem}\label{lem:MetricVsCombinatorial}
Every (irreducible, thick) combinatorial spherical building  has a   metric realisation which is a(n irreducible, thick) metric spherical building. 

Conversely, a  metric spherical building without a spherical de Rham factor carries a canonical  structure of a  simplicial complex, whose simplices are the chambers and their intersections,  which is a  combinatorial spherical   building. 
\end{lem}

\begin{proof}
Given a combinatorial spherical building, the existence of a metric realisation as a \catun space  is proved in~\cite[Proposition~12.29 and Example~12.39]{AbramenkoBrown}. The other verifications are then straightforward. For the converse, the arguments are provided in~\cite[\S 3.4]{Kleiner-Leeb}.
\end{proof}

%%%%%%%%%%%%%%%%%%%%%%%%%%%%%%%%%%%%%%%%%%%%%%%%%%%%%%%%%%%%%%%%%%%%%%%%%%%%%%
\subsection{The Moufang condition}\label{sec:Moufang}
%%%%%%%%%%%%%%%%%%%%%%%%%%%%%%%%%%%%%%%%%%%%%%%%%%%%%%%%%%%%%%%%%%%%%%%%%%%%%%

Let $Z$ be a thick (metric) spherical building. Given a half-apartment $\alpha \subset Z$, we define the \textbf{root group}\index{root group}\index{group!root} associated with $\alpha$ as the subgroup  $U_\alpha \leq \Isom(Z)$ consisting of those elements fixing pointwise each   chamber having a panel  contained in $\alpha$ but not contained in the boundary wall of $\alpha$. We say that $Z$ satisfies the \textbf{Moufang condition}\index{Moufang condition} (or simply is \textbf{Moufang}) if for each half-apartment $\alpha$, the root group $U_\alpha$ acts transitively on the set of apartments containing $\alpha$. We remark that this condition is empty if $\dim Z = 0$. There is a suitable adaption of the Moufang condition for $0$-dimensional buildings, namely the notion of a \emph{Moufang set}, but we do not recall that here since it will not play any role.

Spherical buildings associated with non-compact simple Lie groups, and more generally with isotropic simple algebraic groups over arbitrary fields, are all Moufang (see e.g.~\cite[\S 7.9.3]{AbramenkoBrown}). In particular, the spherical buildings appearing as the visual boundaries of symmetric spaces of non-compact type are Moufang (and associated with simple algebraic groups over the reals). 

Following R.~Weiss~\cite{Weiss09}, a \textbf{Bruhat--Tits building}\index{Bruhat--Tits building} is defined as a thick Euclidean building whose spherical building at infinity is Moufang. The classification of irreducible Bruhat--Tits buildings of dimension~$\geq 2$ has been obtained by Bruhat and Tits and can be consulted in~\cite{Weiss09}. In particular, it turns out that the locally finite (equivalently, locally compact) Bruhat--Tits buildings are precisely the Bruhat--Tits buildings associated with simple algebraic groups over non-Archimedean local fields. This is indeed proved in Chapter~28 from~\cite{Weiss09}, to which the following additional subtle but important clarification should be added. The buildings of mixed type appearing in Table 28.4 from~\cite{Weiss09} happen to be split, and isomorphic to buildings appearing in Table~28.5. They are thus also associated with simple algebraic groups over local fields. This is explained in Remark~28.19 from~\cite{MuhlherrPeterssonWeiss}. In the case of buildings of dimension~$\geq 3$, the same remark is also made by J.~Tits in \S15 of the R\'esum\'e de Cours 1983--1984 that can be consulted in \cite{Tits_resume}.

%%%%%%%%%%%%%%%%%%%%%%%%%%%%%%%%%%%%%%%%%%%%%%%%%%%%%%%%%%%%%%%%%%%%%%%%%%%%%%
\subsection{A metric characterisation of  spherical buildings}
%%%%%%%%%%%%%%%%%%%%%%%%%%%%%%%%%%%%%%%%%%%%%%%%%%%%%%%%%%%%%%%%%%%%%%%%%%%%%%

In proving that the visual boundary $\bd X$ is a spherical building in Theorem~\ref{thm:NotGoedComplete}, we will not check the building axioms directly, but will rather invoke the following criterion, due to  A.~Balser and A.~Lytchak.

\begin{thm}\label{thm:BalserLytchak}
Let $Z$ be a \catun space of dimension $d< \infty$ containing at least one antipodal pair. Assume that every antipodal pair is contained in a $d$-dimensional round sphere $S \subset Z$. If $Z$ contains a non-empty open subset with compact closure, then $Z$ is a metric spherical building. 
\end{thm}
\begin{proof}
See Theorem~1.6 from~\cite{BaLy_BuildingLikeSpaces}.
\end{proof}

%%%%%%%%%%%%%%%%%%%%%%%%%%%%%%%%%%%%%%%%%%%%%%%%%%%%%%%%%%%%%%%%%%%%%%%%%%%%%%
\section{Orthogonal projections at infinity}\label{sec:proj}
%%%%%%%%%%%%%%%%%%%%%%%%%%%%%%%%%%%%%%%%%%%%%%%%%%%%%%%%%%%%%%%%%%%%%%%%%%%%%%

Our goal in this section is to prove a purely geometric statement about projections at infinity. Proposition~\ref{prop:ControlProj} below establishes this result under a technical assumption. We then proceed to present a density criterion which implies a variant of this assumption, Proposition~\ref{prop:DecrBusemann}. This variant will turn out to be sufficient when we will apply Proposition~\ref{prop:ControlProj}.

%%%%%%%%%%%%%%%%%%%%%%%%%%%%%%%%%%%%%%%%%%%%%%%%%%%%%%%%%%%%%%%%%%%%%%%%%%%%%%
\subsection{Projecting on an intersection at infinity}
%%%%%%%%%%%%%%%%%%%%%%%%%%%%%%%%%%%%%%%%%%%%%%%%%%%%%%%%%%%%%%%%%%%%%%%%%%%%%%

Let $X$ be a \cat space. 
Given  $\xi \in \bd X$, we shall denote by $b_\xi \colon X \to \RR$ the Busemann function centred at $\xi \in \bd X$ and normalised at some chosen base point of $X$. We denote by $\proj_A\colon X\to A$ the nearest point projection to a non-empty complete convex subset $A\se X$. We use the same notation $\proj_A$ for the projection in $\bd X$ to a non-empty complete convex subset $A\se \bd X$ but we recall that it is only defined for points at Tits-distance~$<\pi/ 2$ of $A$, see~\cite[II.2.6]{Bridson-Haefliger}.

\begin{prop}\label{prop:ControlProj}
Let $X$ be a proper \cat space and $\mathscr Y$  a filtering family of closed convex subsets of $X$ with empty intersection. Set $D = \bigcap_{Y \in \mathscr Y} \bd Y$. 
Assume that there exists $\xi \in D$ and $\lambda > 0$ such that for all $Y \in \mathscr Y$ and $y \in X$ with $d(y, Y) \geq 1$, we have 
$$b_\xi(y) - b_\xi(\proj_{Y}(y))  \geq  \lambda d(y, Y).$$
Then for all $\eta \in \bd X$ with $\Td(\eta, D) \in (0, \pi/2)$, we have $\proj_{D}(\eta) \neq \xi$. 
\end{prop}

The proof  requires a number of preparations; it will be given at the end of the section. 
For ease of reference, we start by recording the following basic fact.

\begin{lem}\label{lem:tangle}
Let $X$ be a \cat space and $r, r':\RR_+\to X$ two geodesic rays with $r(0)=r'(0)$. Then $\cangle{r(0)}{r(t)}{r'(t')}$ is non-decreasing in both $t$ and $t'$, and it converges to $\tangle{r(\infty)}{r'(\infty)}$ as $t, t'\to\infty$.

Moreover, if $\{y_n\}$, $\{y'_n\}$ are any sequences in $X$ converging to $\eta, \eta'\in\bd X$, then
$$\tangle\eta{\eta'} \leq \liminf_{n\to\infty} \cangle x{y_n}{y'_n}$$
for all $x\in X$.
\end{lem}

\begin{proof}
For the first statement, see II.3.1 and II.9.8(1) in~\cite{Bridson-Haefliger}. The second is~\cite[II.9.16]{Bridson-Haefliger}; here is a shorter proof. 
It suffices to prove that for any $x \in X$, the right hand side dominates $\aangle x\eta{\eta'}$. This holds because $\aangle x\eta{\eta'}=\lim_{n\to\infty} \aangle x{y_n}{y'_n}$ by the continuity of the Alexandrov angle for $x$ fixed, see~\cite[II.9.2(1)]{Bridson-Haefliger}.
\end{proof}

\begin{lem}\label{lem:projections}
Let $X$ be a \cat space, $Y\se X$ a complete convex  subset and $\eta\in\bd X$ such that $\tangle\eta{\bd Y} < \pi/2$. Choose $x\in Y$.

Then given a sequence $(t_n) \subset [x, \eta)$, any accumulation  point in $\bd X$ of the sequence of projections $\big(\proj_Y(t_n)\big)_n$ is $\proj_{\bd Y}(\eta)$.
\end{lem}

\begin{proof}
Let $y_n$ be a sequence along $[x, \eta)$ such that $y'_n:=\proj_Y(y_n)$ converges to some $\eta'\in\bd X$. Since $\eta'\in\bd Y$, it suffices to show that $\tangle{\eta'}\eta \leq \tangle{\zeta}\eta$, where $\zeta=\proj_{\bd Y}(\eta)$. Let $v_n$ be the point of $[x,\zeta)$ with $d(x,v_n)=d(x, y'_n)$. Since $d(y'_n, y_n)\leq d(v_n, y_n)$ we deduce $\cangle x{y'_n}{y_n} \leq \cangle x{v_n}{y_n}$. Now Lemma~\ref{lem:tangle} implies
$$\tangle{\eta'}\eta \leq \liminf_{n\to\infty}\cangle x{y'_n}{y_n} \leq  \lim_{n\to\infty}\cangle x{v_n}{y_n} =  \tangle{\zeta}\eta$$
as was to be shown.
\end{proof}

\begin{lem}\label{lem:fellow}
Let $X$ be a \cat space, $x\in X$ and $\xi, \eta\in\bd X$. If $\tangle\xi\eta\leq\pi/2$, then $b_\xi(y)\leq b_\xi(x)$ for any $y\in[x, \eta)$. In particular, there is a unique $z_y\in[x, \xi)$ with $b_\xi(z_y)= b_\xi(y)$. If $\tangle\xi\eta <\pi/2$, then in addition
$$\lim_y \frac{d(x,z_y)}{d(x,y)} = \cos\tangle\xi\eta$$
as $y$ tends to $\eta$ along $[x, \eta)$.
\end{lem}

\begin{proof}
Let $r:\RR_+\to X$ be the geodesic ray from $x$ to $\xi$. In view of Lemma~\ref{lem:tangle}, we have $\cangle x{r(t)}y \leq \tangle\xi\eta$ for all $t$. Therefore, the asymptotic formula of Lemma~\ref{lem:asy}
\begin{equation}\label{eq:fellow}
b_\xi(x) - b_\xi(y)   = d(x,y) \lim_{t\to\infty} \cos \cangle x{r(t)}y
\end{equation}
is non-negative, which implies the first statement. Assume now that $\tangle\xi\eta <\pi/2$. Then Lemma~\ref{lem:tangle} and~(\ref{eq:fellow}) imply that $b_\xi(y)$ tends to $-\infty$ as $y$ tends to $\eta$ along $[x, \eta)$. Since $d(x,z_y) = b_\xi(x)- b_\xi(y)$, Lemma~\ref{lem:tangle} implies moreover that $\lim \cangle xy{z_y} =\tangle\eta\xi$ along $y$ and~(\ref{eq:fellow}) yields the desired conclusion.
\end{proof}

We now establish the  key ingredient for the proof of Proposition~\ref{prop:ControlProj}.

\begin{prop}\label{prop:Tits:proj}
Let $X$ be a proper \cat space, $Y\se X$ a closed convex subset, $\xi\in\bd Y$ and $\lambda>0$. Assume that for all $y\in X$ with $d(y, Y)\geq 1$ we have
$$b_\xi(y) - b_\xi(\proj_Y(y))  \geq  \lambda d(y, Y).$$
Then every $\eta\in\bd X$ with $\tangle\eta{\bd Y} \in (0, \pi/2)$ satisfies
\begin{equation}\label{eq:trigo}
\cos\tangle\eta{\eta'} \geq \cos\tangle\eta\xi + \lambda \sin\tangle\eta{\eta'},
\end{equation}
where $\eta' = \proj_{\bd Y}(\eta)$.

In particular, $\tangle{\eta'}\xi$ is bounded below by a positive constant depending only on $\lambda$ and on $\tangle\eta\xi$.
\end{prop}

\begin{proof}%[Proof of Proposition~\ref{prop:Tits:proj}]
Let $\eta$ be as in the statement, choose $x\in Y$ and normalise $b_\xi$ at $x$. Let $\{y_n\}$ be a sequence converging to $\eta$ along the ray $[x, \eta)$. Define $y'_n=\proj_Y(y_n)$ and $\eta'=\proj_{\bd Y}(\eta)$.

Since $\eta\notin\bd Y$, the convexity of the function $d(-, Y)$ forces it to grow at least linearly along a subray of $[x, \eta)$. Therefore, we can assume $d(y_n, Y)\geq 1$. Moreover, it follows from the assumption that $y'_n$ tends to infinity; since $Y$ is proper Lemma~\ref{lem:projections} implies that $y'_n$ converges to $\eta'$. Let $z_n=z_{y_n}\in[x, \xi)$ be the point with $b_\xi(z_n)= b_\xi(y_n)$ given by Lemma~\ref{lem:fellow}. Define $\teta_n=\cangle x{y_n}{z_n}$ and $\fhi_n=\cangle x{y_n}{y'_n}$, so that $\lim_{n\to\infty} \teta_n = \tangle\eta\xi$ and $\fhi:=\liminf_{n\to\infty}\fhi_n \geq \tangle\eta{\eta'}$ by Lemma~\ref{lem:tangle}.

Notice first that the cosine law applied to $\cangle{y'_n}x{y_n}\geq \pi/2$ implies 
$$\cos\fhi_n \geq d(x, y'_n) / d(x, y_n).$$ 
Moreover,
$$d(x, y'_n) \geq -b_\xi(y'_n) \geq -b_\xi(y_n) + \lambda d(y_n, Y) = d(x, z_n) + \lambda d(y_n, y'_n)$$
and hence
$$\cos\fhi_n \geq \frac{d(x, z_n) + \lambda d(y_n, y'_n)}{d(x, y_n)}.$$
The sine law ensures $d(y_n, y'_n) / d(x, y_n) \geq \sin\fhi_n$ and thus we conclude
$$\cos\fhi_n - \lambda \sin\fhi_n \geq \frac{d(x, z_n)}{d(x, y_n)}.$$
The right hand side converges to $\cos\tangle\eta\xi$ by Lemma~\ref{lem:fellow} and thus
\begin{equation*}%\label{eq:cosT}
\cos\fhi - \lambda \sin\fhi \geq \cos\tangle\eta\xi.
\end{equation*}
This implies~(\ref{eq:trigo}) stated in the proposition since $\tangle\eta{\eta'} \leq \fhi\leq \pi/2$.

For the additional statement, we write
$$\tangle{\eta'}\xi \geq \tangle\eta\xi - \tangle\eta{\eta'}.$$
The right hand side is non-negative by definition of $\eta'$. It cannot become arbitrarily small when $\tangle\eta\xi$ and $\lambda$ remain fixed because otherwise~(\ref{eq:trigo}) would imply that $\lambda \sin \tangle\eta{\eta'}$ and hence also $\lambda \sin \tangle\eta\xi$ become arbitrarily small.
\end{proof}

Consider a filtering family  of bounded closed convex sets in a complete \cat space. It is known that the intersection is non-empty, but in general its circumradius and circumcentre are not the limit of the radii and centres of the sequence. In contrast, the following lemma shows that projections are better behaved.

\begin{lem}\label{lem:proj:limit}
Let $X$ be a complete \cat space and let $\sC$ be a filtering family  of bounded closed convex sets of $X$. Define $D=\bigcap \sC$. Then for all $x\in X$ we have $d(x, D) = \lim_{C\in\sC} d(x, C)$ and $\proj_D(x)  = \lim_{C\in\sC} \proj_C(x)$.

The same statement holds for $\pi$-convex sets if $X$ is \catun and $d(x, D)<\pi/2$.
\end{lem}

\begin{proof}
Recall first that $D$ is non-empty (see e.g.\ the weak compactness statement~\cite[Thm.~14]{Monod_superrigid}). By construction, the net $d(x, C)$ is non-decreasing in $C$ and bounded above by $d(x, D)$; it thus converges. If the limit $r$ were less than $d(x, D)$, then the closed ball $\bar B(x, r)$ would meet every $C\in\sC$. Applying again the weak compactness, but to $C\cap \bar B(x, r)$, we would obtain a point in $D\cap \bar B(x, r)$, which is absurd.

For the convergence of the projections, consider the midpoint $m$ between $\proj_C(x)$ and $\proj_D(x)$. The \cat inequality gives
$$4 d^2(x, m) \leq 2 d^2(x, C) + 2 d^2(x, D) - d^2(\proj_C(x), \proj_D(x)).$$
On the other hand, $d(x, m) \geq d(x, C)$. Therefore,
$$d^2(\proj_C(x), \proj_D(x)) \leq 2 d^2(x, D) - 2 d^2(x, C)$$
which we have already shown to converge to zero.

The proof in the \catun setting is identical upon using \catun comparison inequalities.
\end{proof}

We are now ready to prove the main result of this section.

\begin{proof}[Proof of Proposition~\ref{prop:ControlProj}]
By hypothesis $\mathscr C = \{\bd Y\}_{Y \in \mathscr Y}$ is a filtering family of closed convex subsets of the \catun space $\bd X$. Set $D = \bigcap \mathscr C$ and let $\eta \in \bd X$ with $\tangle\eta{D} \in (0, \pi/2)$. 

By Lemma~\ref{lem:proj:limit}, we have $\Td(\eta, D) = \lim_{C\in\sC} \Td(\eta,  C)$, so there is a cofinal subfamily $\mathscr C' \subseteq \mathscr C$ such that $\Td(\eta,   C) \in (0, \pi /2)$ for all $C \in \mathscr C'$.

Given $C \in \mathscr C'$, set $\eta_C  =  \proj_C(\eta)$. By Proposition~\ref{prop:Tits:proj}, the distance $\Td(\eta_C,  \xi)$ is bounded below by a  constant $\vareps >0$ which is independent of $C$.  By Lemma~\ref{lem:proj:limit}, we have $\lim_{C\in\sC} \eta_C = \proj_D(\eta)$ in the Tits topology, so that $\Td({\proj_D(\eta)}, \xi) = \lim_{C\in\sC} \Td({\eta_C}, \xi) \geq \vareps >0$. In particular $\proj_D(\eta) \neq \xi$.
\end{proof}

%%%%%%%%%%%%%%%%%%%%%%%%%%%%%%%%%%%%%%%%%%%%%%%%%%%%%%%%%%%%%%%%%%%%%%%%%%%%%%
\subsection{Transversely dense spaces}
%%%%%%%%%%%%%%%%%%%%%%%%%%%%%%%%%%%%%%%%%%%%%%%%%%%%%%%%%%%%%%%%%%%%%%%%%%%%%%

The \textbf{transverse space}\index{transverse space} $X_\xi$\index{$X$@$X_\xi$} associated with a point $\xi \in \bd X$ is defined as follows. On the set $X^*_\xi$ of all geodesic rays $r\colon\RR_+\to X$ pointing to $\xi$, the infimal distance
$$d_\xi(r, r')\ = \ \inf_{t, t'\geq 0} d(r(t), r'(t'))\index{$d$@$d_\xi$}$$
is a pseudo-metric. The transverse space  $(X_\xi, d)$ is then defined as the  metric completion of the quotient metric space of $(X^*_\xi, d)$. It is a complete \cat space (see~\cite[\S 3.A]{amenis}). Moreover, the canonical map associating to each $x \in X$ the geodesic ray $[x, \xi) \in X^*_\xi$ induces a map
$$\pi_\xi \colon X \to X_\xi\index{$pi$@$\pi_\xi$}$$
which is $1$-Lipschitz.

The following criterion will help us to apply Proposition~\ref{prop:ControlProj}.

\begin{prop}\label{prop:DecrBusemann}
Let  $Y \subset X$ be a closed convex subset and let $\xi \in \bd Y$. If the canonical image $\pi_\xi(Y)$ of $Y$ in $X_\xi$ is dense, then $b_\xi(\proj_{Y}(x)) <b_\xi(x)$ for each $x \in X \setminus Y$.
\end{prop}

We shall use the following basic and well-known result on $4$-tuples of points in arbitrary \cat spaces.

\begin{lem}\label{lem:4point}
Let $X$ be any \cat space. Suppose $x, x', y', y\in X$ satisfy $\cangle xy{x'} \geq\pi/2$ and $\cangle yx{y'} \geq \pi/2$. Then $d(x', y')\geq d(x,y)$.
\end{lem}

\begin{proof}[Proof of the lemma]
By the $4$-point condition (see Proposition~II.1.11 in~\cite{Bridson-Haefliger}),  there exists a quadrilateral $(\bar x, \bar x', \bar y', \bar y)$ in $\RR2$ whose side lengths equal the corresponding distances amongst $\{x, x', y', y\}$ and whose diagonals are not shorter. This implies that $\aangle {\bar x}{\bar y}{\bar x'} \geq \pi/2$ and that $\aangle {\bar y}{\bar x}{\bar y'} \geq \pi/2$. We deduce $d(x', y') = d(\bar x', \bar y') \geq  d(\bar x,\bar y) = d(x,y)$ as desired.
\end{proof}

\begin{proof}[Proof of Proposition~\ref{prop:DecrBusemann}]
As a first step, we shall prove that $b_\xi(\proj_{Y}(x)) \leq b_\xi(x)$ holds for all $x \in X \setminus Y$.  Let $r \colon \RR_+\to X$ be the ray from $x$ to $\xi$ and write $y_t =  \proj_{Y}(r(t))$. Suppose for a contradiction that $b_\xi(y_0) > b_\xi(x)$. Then the asymptotic angle formula Lemma~\ref{lem:asy} implies $\cangle x{y_0}{r(t)} >\pi/2$ for all sufficiently large $t$; in particular it is at least $\pi/2$ for such $t$. On the other hand, $\cangle {y_0}x{y_t} \geq \pi/2$ since $y_0$ is the projection of $x$. Therefore, Lemma~\ref{lem:4point} applied to $x, r(t), y_t, y_0$ implies $d(r(t), y_t)\geq d(x, y_0)$. This is absurd because the assumption on $Y$ implies that $d(r(t), y_t)\to 0$.

In order to prove the proposition, we need to show that the equality $b_\xi(y_0) = b_\xi(x)$ leads also to a contradiction. Consider any point $x'$ on the segment $[x, y_0]$. Its projection to $Y$ is still $y_0$, see~\cite[II.2.4(2)]{Bridson-Haefliger}. Therefore, applying the first step to $x'$, we find $b_\xi(x') \geq b_\xi(y_0) =b_\xi(x)$. By convexity of Busemann functions, we deduce that $b_\xi$ is constant on $[x, y_0]$. Thus Lemma~\ref{lem:flat-half} implies that $x$, $y_0$ and $\xi$ span a Euclidean half-strip. In particular, we have $\cangle x{y_0}{r(t)} =\pi/2$ for all $t$ and hence we reach a contradiction by the same argument as in the first step.
\end{proof}

%%%%%%%%%%%%%%%%%%%%%%%%%%%%%%%%%%%%%%%%%%%%%%%%%%%%%%%%%%%%%%%%%%%%%%%%%%%%%%
\section{Boundary points with a cocompact stabiliser}
%%%%%%%%%%%%%%%%%%%%%%%%%%%%%%%%%%%%%%%%%%%%%%%%%%%%%%%%%%%%%%%%%%%%%%%%%%%%%%

Given a proper \cat space $X$, the full isometry group $\Isom(X)$, endowed with the compact-open topology, is a second countable locally compact group. If it acts cocompactly on $X$, then it is compactly generated. 

Throughout the rest of this paper, we fix a \textbf{locally compact \cat group}\index{CAT(0) group@\cat group}\index{group!CAT(0)@\cat}, namely a pair $(X, G)$ consisting of  a proper \cat space $X$ and a closed subgroup $G \leq \Isom(X)$ acting cocompactly on $X$. Notice that $(X, \Isom(X))$ is a locally compact \cat group if and only if $\Isom(X)$ acts cocompactly. If $(X, G)$ is a locally compact \cat group, then so is $(X, \Isom(X))$. Conversely, if $(X, \Isom(X))$  a locally compact \cat group, then  for any closed subgroup $G \leq \Isom(X)$ the pair $(X, G)$ is so if and only if $G$ is a cocompact subgroup of  $\Isom(X)$. 

%%%%%%%%%%%%%%%%%%%%%%%%%%%%%%%%%%%%%%%%%%%%%%%%%%%%%%%%%%%%%%%%%%%%%%%%%%%%%%
\subsection{Cocompact points and c\^one-closed orbits}
%%%%%%%%%%%%%%%%%%%%%%%%%%%%%%%%%%%%%%%%%%%%%%%%%%%%%%%%%%%%%%%%%%%%%%%%%%%%%%

A boundary point $\xi \in \bd X$ is called \textbf{cocompact}\index{cocompact point at infinity}\index{point at infinity!cocompact}\index{boundary point!cocompact} (relative to $G$) if its stabiliser $G_\xi$ acts cocompactly on $X$.

There is a basic characterisation of cocompact points in terms of the dynamics of the $G$-action on the boundary $\bd X$ endowed with c\^one topology:

\begin{lem}\label{lem:CocptOrbit}
For each   $\xi \in \bd X$, the following assertions are equivalent: 
\begin{enumerate}[(i)]
\item $ \xi$ is cocompact.

\item The $G$-orbit of $\xi$ is closed in $\bd X$ (with respect to the c\^one topology).
\end{enumerate}
\end{lem}

Thus we see that if there is a cocompact point, then the boundary $\bd X$ contains $G$-invariant closed subsets consisting entirely of cocompact points.

\begin{proof}
(i) $\Rightarrow$ (ii) 
Since $G_\xi$ acts cocompactly on $X$, there exists some compact
subset $K \subset G$ such that $G= K \cdot G_\xi$. Thus we have
$G.\xi = K.\xi \subset \bd X$, and the Lemma follows since the
$G$-action on $\bd X$ is continuous with respect to the c\^one
topology.

\medskip \noindent
(ii) $\Rightarrow$ (i) 
The $G$-orbit of $\xi$, which we denote by $\Omega$, is closed, hence compact. By  the corollary to Theorem~8 in~\cite{Arens46}, this implies that continuous surjective map $G/G_\xi \to \Omega$ induced by the orbit map is a homeomorphism. In particular $G_\xi$ is cocompact in $G$ and, hence, acts cocompactly on $X$. 
\end{proof}

%%%%%%%%%%%%%%%%%%%%%%%%%%%%%%%%%%%%%%%%%%%%%%%%%%%%%%%%%%%%%%%%%%%%%%%%%%%%%%
\subsection{Levi decompositions}
%%%%%%%%%%%%%%%%%%%%%%%%%%%%%%%%%%%%%%%%%%%%%%%%%%%%%%%%%%%%%%%%%%%%%%%%%%%%%%

An important tool in analysing cocompact points and the structure of their stabilisers is provided by the geometric Levi decomposition theorem that we established in~\cite{amenis}. In order to recall its statement, we first introduce some further notation and terminology. 

Given $\xi \in \bd X$, recall that $\beta_\xi$ denotes the Busemann character centred at $\xi$. Moreover $G_\xi^{\mathrm u}$ denotes the closed normal subgroup of $G_\xi$ defined by
$$G_\xi^\mathrm{u} := \Big\{\ g\in G \ : \ \lim_{t\to\infty} d\big(g\cdot r(t), r(t)\big)=0 \ \ \forall\, r \text{\ with\ } r(\infty)=\xi \Big\}\index{$G$@$G_\xi^{\mathrm u}$}$$
(where $r\colon \RR\to X$ are geodesic rays). This group may be viewed as intersection of $\Ker(\beta_\xi)$ with the kernel of the $G_\xi$ action on the  transverse space  $X_\xi$.

Recall from \S\ref{sec:JoinsLinks} that the set $P(\xi, \xi')$ associated with an opposite pair $\xi, \xi' \in \bd X$ is the union of all geodesic lines whose endpoints are $\xi, \xi'$. We have moreover a canonical product decomposition $P(\xi, \xi') \cong \RR \times Y$, where the boundary of the line factor is the pair $\{\xi, \xi'\}$.
By restricting the canonical maps $\pi_\xi \colon X \to X_\xi$ and $\pi_{\xi'} \colon X \to X_{\xi'}$  to $P(\xi, \xi')$, one obtains isometric embeddings $Y \to X_\xi$ and $Y \to X_{\xi'}$. The following shows in particular that these are onto. 

\begin{prop}\label{prop:lifting}
Let  $\xi \in \bd X$ be cocompact and $\xi' \in \Op(\xi)$. Then
$$P(\xi, \xi') \cong \RR \times X_\xi \cong \RR \times X_{\xi'}$$
and the action of $G_{\xi, \xi'}$ on $P(\xi, \xi')$ is given by $g\colon (r, x) \mapsto (r + \beta_\xi(g), g.x)$.

Moreover, the $G_{\xi, \xi'}$-action on $P(\xi, \xi')$ is cocompact. 
\end{prop}

\begin{proof}
See Theorem~3.9, Proposition~3.10, and Proposition~K from~\cite{amenis}.
\end{proof}

The following consequence is immediate.

\begin{cor}\label{cor:OppositePair}
For each cocompact point $\xi$ and each $\xi' \in \Op(\xi)$, the
spaces $X_\xi$ and $X_{\xi'}$ are isometric.\qed
\end{cor}

The following result was obtained in~\cite{amenis} and will be a key tool in the subsequent developments; it provides an analogue of the Levi decomposition of parabolic subgroups in semi-simple algebraic groups. The main case is $N = G_\xi$; notice that the \emph{existence} of a sequence $\{h_n\}$ as below is then trivial.

\begin{thm}[Levi decomposition]\label{thm:Levi}\index{Levi decomposition}
Let  $\xi \in \bd X$ be cocompact.  Let $N<G_\xi$ be a subgroup that is normalised by some radial sequence $\{h_n\}$ in $G_\xi$ for $\xi$. Let $\xi'\in\bd X$ be a limit point of $\{h_n\}$; hence $\xi' \in \Op(\xi)$.

Then, writing  $N^\mathrm{u}=N\cap G_\xi^\mathrm{u} $, we have a decomposition
$$N = N_{\xi'} \cdot N^\mathrm{u} $$
which is almost semi-direct in the sense that $N_{\xi'} \cap N^\mathrm{u} $ has compact closure. In particular $N^\mathrm{u}$ acts transitively on the $N$-orbit of $\xi'$ in $\Op(\xi)$. 

Moreover $G_\xi^\mathrm{u}$ is amenable, and  locally elliptic if $G_\xi$ is totally disconnected.
\end{thm}

We recall that a locally compact group is \textbf{locally elliptic}\index{locally elliptic} if every finite subset is contained in a compact subgroup. 

\begin{proof}
See~\cite[Theorem 3.12]{amenis}. The amenability of $G_\xi^\mathrm{u}$ (and hence that of $N^\mathrm{u}$ when $N$ is closed) follows from~\cite[Corollary~4.5]{amenis}. The  local ellipticity  of $G_\xi^\mathrm{u}$ (assuming $G_\xi$ totally disconnected) follows from Propositions~4.2 and~4.4 in~\cite{amenis} applied to $G_\xi$.
\end{proof}

We shall often use the following basic part of the theorem, which we therefore isolate:

\begin{cor}\label{cor:transitive}
For any cocompact point $\xi \in \bd X$, the group  $G_\xi^\mathrm{u}$ acts transitively on the non-empty set $\Op(\xi)$.
\end{cor}

\begin{proof}
As recalled above,  $G_\xi$ acts transitively on the non-empty set $\Op(\xi)$, see Proposition~7.1 in~\cite{Caprace-Monod_structure}. Thus the statement follows from the case $N=G_\xi$ of Theorem~\ref{thm:Levi}.
\end{proof}

\subsection{Orbits and fixed points of the Levi factor}
Theorem~\ref{thm:Levi} gives in particular an almost semi-direct Levi decomposition $G_\xi = G_{\xi, \xi'} \cdot G_\xi^\mathrm{u}$. We call $G_{\xi, \xi'}$ the \textbf{Levi factor}\index{Levi factor} of this decomposition. The following result will play an important role in the sequel.

\begin{prop}\label{prop:2CocptPts}
Let $\xi \in \bd X$ be  cocompact  and $\xi' \in \Op(\xi)$.  For any $\eta \in \bd X$, the closure
$\overline {G_{\xi, \xi'}.\eta}$ contains a point $\eta' \in \bd P(\xi, \xi')$ such that $\tangle\xi { \eta'} = \tangle \xi\eta$.
\end{prop}

\begin{proof}
Pick a base point $p \in P(\xi, \xi')$ and let $\ro\colon \RR \to X$
be the geodesic line such that $\ro(0)=p$, $\ro(+\infty) = \xi$
and $\ro(-\infty) = \xi'$. By Proposition~\ref{prop:lifting}, there is $g_n \in G_{\xi, \xi'}$ be such that
$g_n. \ro(n)$ remains bounded. Upon extracting, the sequence
$(g_n.\ro(n))$ converges to some $p'' \in P(\xi, \xi')$ and
$(g_n.\eta)$ converges to some $\eta'' \in \bd X$. Applying Proposition~\ref{prop:upper:semi}, we have
\begin{equation}\label{eq:2CocptPts}
\aangle{p''} \xi {\eta''} = \tangle \xi {\eta''} = \tangle \xi {\eta}.
\end{equation}
We now choose a sequence $(g'_n)$ in $G_{\xi, \xi'}$ such that $d(\ro(-n), g'_n.p'')$ remains bounded; as it happens $g'_n=g_n$ would do since $(g_n)$ keeps the entire line $\ro(\RR)$ at bounded distance of itself. Upon extracting, the sequence $(g'_n.\eta'')$ converges to some $\eta' \in \bd X$; by construction, $\eta'\in\overline {G_{\xi, \xi'}.\eta}$. Moreover, another application of Proposition~\ref{prop:upper:semi} shows that $\tangle \xi {\eta'}= \tangle \xi {\eta''} = \tangle \xi {\eta}$. It only remains to show that $\eta'$ belongs to $\bd P(\xi, \xi')$.

Consider first the special case $\tangle \xi\eta = \pi$. Then $\xi$ and $\eta''$ are opposite in view of~(\ref{eq:2CocptPts}). The construction of $\eta'$ now in fact forces $\eta'=\xi'$ which finishes this case.

We assume henceforth $\tangle \xi\eta <\pi$. Then~(\ref{eq:2CocptPts}) implies that the triple $(p'', \xi, \eta'')$ spans a flat sector which we denote by $S$, see Proposition~II.9.9 in~\cite{Bridson-Haefliger}. By construction, the sequence of sectors $g'_n.S$ subconverges to a half-flat $H$ whose boundary line is parallel to $\ro$. Therefore $H$ is entirely contained in $P(\xi, \xi')$. Since $\eta'' \in \bd S$, it follows that $\eta' \in \bd H $, whence $\eta' \in \bd P(\xi, \xi')$ as desired.
\end{proof}

In particular, we obtain the following.

\begin{cor}\label{cor:FixedPts:Levi}
Let $\xi \in \bd X$ be cocompact and $\xi' \in \Op(\xi)$. Then 
$$(\bd X)^{G_{\xi}} \subset  (\bd X)^{G_{\xi, \xi'}} \subset \bd P(\xi,
\xi').$$
\end{cor}

\begin{proof}
Immediate from Proposition~\ref{prop:2CocptPts}.
\end{proof}

%%%%%%%%%%%%%%%%%%%%%%%%%%%%%%%%%%%%%%%%%%%%%%%%%%%%%%%%%%%%%%%%%%%%%%%%%%%%%%
\subsection{Cocompact points and maximal flats}
%%%%%%%%%%%%%%%%%%%%%%%%%%%%%%%%%%%%%%%%%%%%%%%%%%%%%%%%%%%%%%%%%%%%%%%%%%%%%%
We recall from~\cite[Theorem~C]{Kleiner} that if $X$ admits a cocompact isometry group, then $X$ contains a flat $F $ such that $\dim(\bd X) = \dim F -1$. In particular $\dim(\bd X) < \infty$. Any such flat will be called a \textbf{fully maximal flat}\index{fully maximal!flat}\index{flat!fully maximal}. The set of all fully maximal flats is denoted by $\sF X$\index{$FX$@$\sF X$}. 

Similarly, one defines a  \textbf{fully maximal sphere}\index{fully maximal!sphere}\index{sphere!fully maximal} to be a round sphere $S \subseteq \bd X$ such that  $\dim(S) = \dim(\bd X)$. The set of all fully maximal spheres is denoted by  $\sS X$\index{$SX$@$\sS X$}. 
 Clearly any fully maximal flat is bounded by a fully maximal sphere. More importantly, the converse also holds:
 
\begin{prop}\label{prop:Leeb}
Any fully maximal sphere bounds a fully maximal flat. 

Moreover, the union of all fully maximal flats bounded by a given fully maximal sphere is a closed convex subset which splits as a product $\RR^{d+1} \times Y$, with $Y$ bounded and $d = \dim( \bd X)$. In particular, any fully maximal sphere bounds a \emph{canonical} fully maximal flat. 
\end{prop}

\begin{proof}
The first assertion follows from~\cite[Proposition~2.1]{Leeb}. For the second we refer to Proposition~3.6 in~\cite{Caprace-Monod_structure}; the \emph{canonical} flat corresponds to the circumcenter of $Y$ (as in Corollary~3.10 of~\cite{Caprace-Monod_structure}).
\end{proof}

The following fact is an essential first step in showing that the set of all cocompact points of $\bd X$ is a spherical building.

\begin{prop}\label{prop:MaxFlat}\index{cocompact point at infinity!is contained in a fully maximal sphere}
Every cocompact point is contained in a fully maximal sphere.
\end{prop}

\begin{proof}
We have just recalled that $\bd X$ contains some fully maximal sphere which bounds a fully maximal flat $F$. 
Let $\xi \in \bd X$ be cocompact.
By Lemma~\ref{lem:AntipodalPt}, the sphere $\bd F$
contains some point $\eta$ which is antipodal to $\xi$. Pick a base
point $p \in F$ and a geodesic ray $\ro \colon \RR_+ \to X$ emanating
from $p$ and pointing to $\eta$. Choose moreover a sequence $(g_n)$
of elements of $G_\xi$ such that $g_n.\ro(n)$ remains bounded. Upon
extracting we may assume that $g_n.\ro(n) $ converges to some $x
\in X$. Up to a further extraction, the sequence of flats $g_n.F$
converges to a flat $F'$ containing $x$. Since $\dim F' = \dim F$,
we have $F' \in \sF X$. Now, the sequence $g_n.\eta$ converges
to a point $\eta' \in \bd F'$. Proposition~\ref{prop:upper:semi} implies
$\aangle x \xi {\eta'} = \tangle\xi\eta = \pi$. In other
words, the point $x$ belongs to a geodesic line joining $\xi$ to
$\eta'$ and hence $\eta' \in \Op(\xi)$.

Now the image of the flat $F'$ in $X_{\eta'}$ is a flat of dimension
$\dim F'-1$. In view of Proposition~\ref{prop:lifting}, the latter flat
lifts to a flat $F'' \subset P({\xi, \eta'})$ such that $\xi \in \bd
F''$ and  $\dim F''=\dim F'$. In particular $F''$ is a fully
maximal flat of $X$.
\end{proof}

In~\cite[Theorem~M]{amenis}, we proved that if a closed \emph{unimodular} subgroup $H < \Isom(X)$ acts cocompactly on $X$, then its fixed point set $(\bd X)^H$ is contained in the boundary of the maximal Euclidean factor of a minimal $H$-invariant subspace of $X$. In particular  $(\bd X)^H$ is contained in a fully maximal sphere. The   latter property holds even without assuming that $H$ is unimodular.

\begin{prop}\label{prop:FixedPtSet}
Let $H < \Isom(X)$ act cocompactly on $X$. Then the (possibly empty) fixed point set $(\bd X)^H$ is contained in a fully maximal sphere. 
\end{prop}

\begin{proof}
Let $G = \overline H$ be the closure of $H$ in $\Isom(X)$. Thus $(X, G)$ is a locally compact \cat group. The proof proceeds by induction on $\dim(\bd X)$. If $\dim(\bd X)=0$, then $X$ does not contain any flat of dimension~$\geq 2$. It follows that $\sS X$ consists exactly of all sets of two distinct points at infinity, see Theorem~II.9.33 in~\cite{Bridson-Haefliger}. Moreover, $X$ is Gromov-hyperbolic~\cite[III.H.1.5]{Bridson-Haefliger}. Thus, saying that the fixed point set $(\bd X)^H$ is contained in a fully maximal sphere amounts to saying that it contains at most two points. This is indeed true, since if  $(\bd X)^H$ contained at least~$3$ points, then $H$ would have a bounded orbit in $X$ by hyperbolicity, hence $X$ would be bounded since $H$ acts cocompactly, hence $\bd X$ would be empty, which is absurd.  The induction can start.

For the induction step, we may assume that $(\bd X)^H$ is non-empty. Let then $\xi \in (\bd X)^H$; thus $\xi$ is cocompact. By Proposition~\ref{prop:MaxFlat}, the point $\xi$ is contained in a fully maximal sphere, say $S$. Let $\xi' \in S$ be antipodal to $\xi$. Then $\xi'$ is opposite $\xi$ by Proposition~\ref{prop:Leeb}.  Corollary~\ref{cor:FixedPts:Levi} ensures that $(\bd X)^H = (\bd X)^G \subset (\bd X)^{G_{\xi'}} \subset \bd P(\xi, \xi')$, while $S \subset \bd P(\xi, \xi')$ by Proposition~\ref{prop:Leeb}. 

From Proposition~\ref{prop:lifting}, we infer that $(\bd X)^{G_{\xi'}} \cong \{\xi, \xi'\} * (\bd X_\xi)^{G_{\xi'}}$; moreover $G_{\xi'}$ acts cocompactly on $X_\xi$. By induction, it follows that $(\bd X_\xi)^{G_{\xi'}}$ is contained in a fully maximal sphere, say $S_\xi$. Since $S \subset \bd P(\xi, \xi')$, we have $\dim(\bd X_\xi) = \dim S -1 = \dim(\bd X)-1$. Now the sphere $S' =  \{\xi, \xi'\} * S_\xi \subset \bd P(\xi, \xi')$ is a fully maximal sphere of $\bd X$ that contains   $(\bd X)^{G_{\xi'}} \supset (\bd X)^H$.
\end{proof}

The following property of  convex sets of round spheres is  well known.

\begin{lem}\label{lem:RadSphericalConvex}
A closed convex subset $Z$ of a unit round sphere $S$ is either a sphere or has intrinsic radius~$\leq \pi/2$.
\end{lem}
\begin{proof}
We work by induction on $\dim(S)$, the base case $\dim(S)=0$ being trivial.

We may assume that $Z$ is properly contained in $S$, so   it is contained in a closed hemisphere $H$. If $Z$ is contained in the boundary equator $E$ of $H$, then the conclusion follows by induction. We assume henceforth that  $Z$ meets the interior  of $H$; we shall prove by contradiction that $\rad(Z) \leq \pi/2$.  

Assume on the contrary that  $\rad(Z) > \pi/2$. Then $Z$ contains at least one antipodal pair by~\cite[Lemma~1.7]{AB98}, which must be contained in $E$. Since $Z$ meets the interior of $H$, we infer that $Z$ contains the centre of $H$. Since $Z \subset H$, we conclude that $\rad(Z) \leq \pi/2$, a contradiction. 
\end{proof}

\begin{prop}\label{prop:RadFixedPt}
Let $H < \Isom(X)$ act cocompactly on $X$. Assume that the fixed point set $(\bd X)^H$ is non-empty.

Then $\rad((\bd X)^H) \leq  \pi/ 2$ unless $\bd X$ decomposes as a join with  a non-empty spherical factor. 

If moreover $H$ is amenable, then $\rad((\bd X)^H) \leq  \pi/ 2$ unless $\bd X$  is a round sphere.
\end{prop}

\begin{proof}
By Proposition~\ref{prop:FixedPtSet}, the set $(\bd X)^H$ is a closed convex subset of some $S\in \sS X$. If $\rad((\bd X)^H) >  \pi/ 2$, then $(\bd X)^H$ is a subsphere of $S$ by Lemma~\ref{lem:RadSphericalConvex}. By Proposition~\ref{prop:Leeb}, that sphere bounds a flat in $X$. The union of all those flats is a closed convex $H$-invariant subset $Y \subset X$ which decomposes $H$-equivariantly as a product $Y \cong \RR^{d+1} \times Y'$, where $d = \dim((\bd X)^H)$ and the boundary of the flat factor $\RR^{d+1}$ coincides with the sphere $(\bd X)^H$, see Proposition~3.6 in~\cite{Caprace-Monod_structure}. Since $H$ acts cocompactly on $X$, we have $\bd X = \bd Y$, and the product decomposition of $Y$ induces a join decomposition of $\bd X \cong (\bd X)^H \circ \bd Y'$ having the non-empty sphere $(\bd X)^H$ as a factor. 

Assuming in addition that $H$ is amenable, we consider the $H$-action on the factor $Y'$ in the above decomposition of $Y$. The group $H$ has no fixed point in $\bd Y'$ since all $H$-fixed points at infinity lie in the boundary of the flat factor $\RR^{d+1}$. Therefore the Main Theorem of~\cite{AB98} ensures that $H$ stabilises a flat in $Y'$. Since the $H$-action is cocompact, it follows that $\bd Y'$ is a round sphere, hence so is $\bd X$.
\end{proof}

%%%%%%%%%%%%%%%%%%%%%%%%%%%%%%%%%%%%%%%%%%%%%%%%%%%%%%%%%%%%%%%%%%%%%%%%%%%%%%
\section{Totally disconnected amenable \cat  groups}\label{sec:TD}
%%%%%%%%%%%%%%%%%%%%%%%%%%%%%%%%%%%%%%%%%%%%%%%%%%%%%%%%%%%%%%%%%%%%%%%%%%%%%%
%Throughout this section, we let $G$ be a closed subgroup of $\Isom(X)$ and $A < G$ be an amenable subgroup acting cocompactly on $X$. Since the closure $\overline A$ is  amenable and cocompact as well, there is no loss of generality in assuming that $A$ is closed.

%%%%%%%%%%%%%%%%%%%%%%%%%%%%%%%%%%%%%%%%%%%%%%%%%%%%%%%%%%%%%%%%%%%%%%%%%%%%%%
\subsection{Orthogonal projection to the  fixed-point-set at infinity}
%%%%%%%%%%%%%%%%%%%%%%%%%%%%%%%%%%%%%%%%%%%%%%%%%%%%%%%%%%%%%%%%%%%%%%%%%%%%%%

The proof of the following result combines the Levi decomposition theorem with the geometric facts established in Section~\ref{sec:proj}. 

\begin{prop}\label{prop:TitsNb}
Let $(X, G)$ be a locally compact \cat group with $G$ totally disconnected, and let $\xi \in \bd X$ be cocompact. 

Then, for every $\eta \in \bd X \setminus (\bd X)^{G_\xi^\mathrm{u}}$ with $d_T(\eta, (\bd X)^{G_\xi^\mathrm{u}})< \pi/2$, the projection of $\eta$ to $(\bd X)^{G_\xi^\mathrm{u}}$ is distinct from $\xi$.
\end{prop}

Proposition~\ref{prop:TitsNb}  will be mostly used through the following.

\begin{cor}\label{cor:proj}
Let $(X, G)$ be a locally compact \cat group with $G$ totally disconnected. Let $\xi \in \bd X$ be cocompact and   $\xi' \in \Op(\xi)$. 
Then for each $\eta \neq \xi$ with $\eta \in (\bd X)^{G_{\xi, \xi'}} $ lying at distance~$< \pi /2$ from $ (\bd X)^{G_\xi}$, we have $\proj_{ (\bd X)^{G_\xi}}(\eta) \neq \xi$. 
\end{cor}

\begin{proof}[Proof of Corollary~\ref{cor:proj} assuming Proposition~\ref{prop:TitsNb}]
There is no loss of generality in assuming $\eta \not \in (\bd X)^{G_\xi}$.
Since $G_\xi^\mathrm{u}$ is normal in $G_\xi$, it follows that $ (\bd X)^{G_\xi^\mathrm{u}}$ is invariant under $G_\xi$. In particular the point $\proj_{ (\bd X)^{G_\xi^\mathrm{u}}}(\eta)$ is fixed by $G_{\xi, \xi'}$. Since $G_\xi = G_{\xi, \xi'} G_\xi^\mathrm{u}$ by Theorem~\ref{thm:Levi}, we infer that  $\proj_{ (\bd X)^{G_\xi^\mathrm{u}}}(\eta)$ is in fact fixed by $G_\xi$. Since  $ (\bd X)^{G_\xi^\mathrm{u}}$ contains  $ (\bd X)^{G_\xi}$, we infer that 
$$ \proj_{ (\bd X)^{G_\xi}}(\eta) =  \proj_{ (\bd X)^{G_\xi^\mathrm{u}}}(\eta),$$
and the inequality $\proj_{ (\bd X)^{G_\xi}}(\eta) \neq \xi$ now follows from Proposition~\ref{prop:TitsNb}.
\end{proof}

The proof of Proposition~\ref{prop:TitsNb} will be given at the end of the subsection, as it necessitates a number of preparations.

\begin{lem}\label{lem:DecrBusemann:cpt}
Let $\xi \in \bd X$ and $K \leq \Isom(X)$ be a compact subgroup fixing $\xi$ and acting trivially on $X_\xi$. Let $Y = X^K$ be the fixed point set of $K$. Then $\xi \in \bd Y$ and  $\pi_\xi(Y)$ is dense in $X_\xi$. In particular,  $b_\xi(\proj_{Y}(x)) <b_\xi(x)$ for each $x \in X \setminus Y$.
\end{lem}

\begin{proof}
Since $K$ is compact, its fixed point set $Y$ is non-empty and thus contains a geodesic ray pointing to $\xi$. Hence $\xi \in \bd Y$. 

Let now $r \colon\RR_+\to X$ be any geodesic ray pointing to $\xi$. Let $\delta_t$ denote the diameter of the $K$-orbit of $r(t)$. We claim that $\lim_{t \to \infty} \delta_t =0$. Indeed, there would otherwise exist  a positive constant $\vareps > 0$ and sequences $(k_n) \subset K$ and $(t_n)  \subset \RR_+$ with $t_n \to \infty$ such that $d(r(t_n), k_n r(t_n)) \geq \vareps$ for all $n$. Let $k$ be any limit point of $(k_n)$ in $K$. We have
%
%$$d(r(t_n), k r(t_n)) - d(r(t_n), k_n r(t_n))\geq - d(k r(t_n), k_n r(t_n))   \geq - d(k r(0), k_n r(0)) $$
%
$$d(r(t_n), k r(t_n))\geq  d(r(t_n), k_n r(t_n)) - d(k r(t_n), k_n r(t_n))   \geq  \vareps - d(k r(0), k_n r(0)) $$
by convexity of the metric, so that $\liminf d(r(t_n), k r(t_n)) \geq \vareps$. Since $k$ acts trivially on the point $\pi_\xi(r(0))$ in $X_\xi$, there is a sequence $(s_n)$ in $\RR_+$ with $d(r(s_n), k r(t_n))\to0$. In particular, $b_\xi(r(s_n)) - b_\xi(k r(t_n)) \to 0$. Since $K$ is compact, $k$ is in the kernel of the Busemann character $\beta_\xi$ and thus $b_\xi(k r(t_n))= b_\xi (r(t_n))$ for all $n$. Therefore $t_n-s_n\to 0$ and hence $d(r(t_n), k r(t_n))\to0$, contradicting our earlier lower bound. This proves the claim.

Since the circumcenter of any $K$-orbit is fixed by $K$, it follows that $\lim_{t \to \infty} d(r(t), Y)=0$. This confirms that  $\pi_\xi(Y)$ is dense in $X_\xi$. We conclude using Proposition~\ref{prop:DecrBusemann}.
\end{proof}

In the case of cocompact points, we obtain the following uniform version of Lemma~\ref{lem:DecrBusemann:cpt} with a constant $\lambda$ independent of the group $K$.

\begin{lem}\label{lem:UnifDecrBusemann}
Let $(X, G)$ be a locally compact \cat group, and  $\xi \in \bd X$ be cocompact.

Then for each  $\delta >0$, there exists $\lambda = \lambda(\delta)>0$ such that for every compact subgroup $K < G_\xi^\mathrm{u}$  and all $x \in X$ with $d(x, X^K)\geq \delta$, we have
$$ b_\xi(x)  - b_\xi(\proj_{X^K}(x)) \geq  \lambda \, d(x, X^K).$$
\end{lem}

\begin{proof}
Suppose the contrary. Then there exist $\delta>0$, a sequence $\{\lambda_n\}$ of positive real numbers with $\lim_n \lambda_n =0$, a sequence of compact subgroups $K_n \leq G_\xi^\mathrm{u}$ and a sequence $\{x_n\}$ of points of $X$ with $d(x_n, X^{K_n}) \geq \delta$ satisfying
$$b_\xi(\proj_{{X^{K_n}}}(x_n)) > b_\xi(x_n) -  \lambda_n \, d(x_n, {X^{K_n}}) \kern5mm \forall\,n.$$
Set $y_n = \proj_{{X^{K_n}}}(x_n)$. 

Let $x'_n$ now be the point of $[x_n, y_n]$ at distance $\delta$ of $y_n$. Then the projection $\proj_{{X^{K_n}}}(x'_n)$ is still $y_n$, see~\cite[II.2.4(2)]{Bridson-Haefliger}. The convexity of $b_\xi$ implies
\begin{align*}
b_\xi(x'_n) &\leq \frac{d(x_n, {y_n}) - \delta}{d(x_n, {y_n})} b_\xi(y_n) + \frac{\delta}{d(x_n, {y_n})} b_\xi(x_n)\\
&< \frac{d(x_n, {y_n}) - \delta}{d(x_n, {y_n})} b_\xi(y_n) + \frac{\delta}{d(x_n, {y_n})} \Big(b_\xi(y_n) + \lambda_n d(x_n, {y_n})\Big) = b_\xi(y_n) + \lambda_n\delta.
\end{align*}

\smallskip

Using the cocompactness of the $G_\xi$-action, we find a sequence $\{g_n\}$ of elements of $G_\xi$ such that $\{g_n.x'_n\}$ is bounded. Hence $\{g_n.y_n\}$ remains bounded as well. Upon extracting, we may then assume that 
$g_n.x'_n$ and $g_n. y_n$ converge respectively  to some $x'$ and $y \in X$, and that moreover $g_n K_n g_n\inv$ converges in the Chabauty topology to some closed subgroup $K \leq  G_\xi^\mathrm{u}$,  which must be compact since it fixes $y$.

We have $b_\xi(g_nx'_n) - b_\xi(g_ny_n) = b_\xi(x'_n) - b_\xi(y_n)>0$ for all $n$ by Lemma~\ref{lem:DecrBusemann:cpt}. Therefore   $b_\xi(x') = b_\xi(y)$. We claim that $y = \proj_{{X^K}} (x')$. This will yield a contradiction with Lemma~\ref{lem:DecrBusemann:cpt} and hence finish the proof.

We need to prove that for all $p \in {X^K}$, we have $d(x,p ) \geq d(x, y) = \delta$. Suppose for a contradiction that there is some $p \in {X^K}$ with $d(x,p) < d(x, y)$. Set $p_n = g_n\inv.p$. Then there is $\vareps>0$ such that $d(x_n, p_n) \leq \delta-\vareps$ for all sufficiently large $n$; in particular $d(p_n, X^{K_n}) \geq \vareps$ for all large $n$. Since the circumcentre of the orbit $K_n.p_n$ is $K_n$-fixed, it follows that the diameter of $K_n.p_n$ is at least~$\vareps$. Therefore there is $k_n\in K_n$ with $d(k_n p_n, p_n)\geq \vareps$. The sequence $\{g_n k_n g_n\inv \}$ is bounded since by construction $g_n k_n g_n\inv y\to y$. Therefore it has an accumulation point $k$. Now $k$ belongs to $K$ but $d(k p, p)\geq \vareps$, a contradiction.
\end{proof}

We are now able to complete the proof of Proposition~\ref{prop:TitsNb}. 

\begin{proof}[Proof of Proposition~\ref{prop:TitsNb}]
By Theorem~\ref{thm:Levi}, the group $G_\xi^\mathrm{u}$ is locally elliptic; hence its compact subgroups form a directed system $\sK$ whose union is $G_\xi^\mathrm{u}$. It follows that the family of fixed-point-sets $\mathscr Y = \{X^K\}_{K \in \sK}$ is filtering. 

If $\bigcap \mathscr Y$ is non-empty, then  $G_\xi^\mathrm{u}$ fixes points of $X$ and is thus compact. Hence $G_\xi^\mathrm{u}$ acts trivially on any minimal closed convex $G_\xi$-invariant subset of $X$; in particular $G_\xi^\mathrm{u}$ acts trivially on $\bd X$. The desired assertion is thus vacuous in that case. We assume henceforth that $\bigcap \mathscr Y = \varnothing$.

Remark that the fixed point set $(\bd X)^{G_\xi^\mathrm{u}}$ is  the intersection of the filtering family  of closed $\pi$-convex sets $\bd X^K$, $K\in\sK$. We use the simple notation $\bd X^K$ since $\bd(X^K) = (\bd X)^K$ by compactness of $K$.

Therefore the desired conclusion follows from Proposition~\ref{prop:ControlProj}, whose hypothesis is satisfied in view of 
Lemma~\ref{lem:UnifDecrBusemann}.
\end{proof}

%%%%%%%%%%%%%%%%%%%%%%%%%%%%%%%%%%%%%%%%%%%%%%%%%%%%%%%%%%%%%%%%%%%%%%%%%%%%%%
\subsection{Finite Weyl groups}
%%%%%%%%%%%%%%%%%%%%%%%%%%%%%%%%%%%%%%%%%%%%%%%%%%%%%%%%%%%%%%%%%%%%%%%%%%%%%%
The following  consequence of Bieberbach's Theorem will be used frequently in the sequel. 

\begin{lem}\label{lem:Bieberbach}\index{Weyl group!is finite}
Let $X$ be a proper \cat space and $G \leq \Isom(X)$ be a closed totally disconnected subgroup. For any fully maximal sphere $S$, the group $W= \Stab_G(S)/\Fix_G(S)$ is finite.
\end{lem}

\begin{proof}
Proposition~\ref{prop:Leeb} implies that  $\Stab_G(S)$ stabilises a fully maximal flat $F$ with $\bd F =S$. Since $\Stab_G(S)$ is closed in $\Isom(X)$, it acts properly on $X$, hence on $F$, so that its image in $\Isom(F)$ is closed. Any closed totally disconnected subgroup of the Lie group $\Isom(F)$ is discrete. Therefore Bieberbach's Theorem ensures that $\Stab_G(S)$ has a finite index subgroup acting by translations on $F$ (the specific version of Bieberbach's Theorem used here is given, for instance, as Theorem~2 in~\cite{Oliver}). That subgroup is thus contained in $\Fix_G(S)$, and the result follows.
\end{proof}

%%%%%%%%%%%%%%%%%%%%%%%%%%%%%%%%%%%%%%%%%%%%%%%%%%%%%%%%%%%%%%%%%%%%%%%%%%%%%%
\subsection{Regular points at infinity}
%%%%%%%%%%%%%%%%%%%%%%%%%%%%%%%%%%%%%%%%%%%%%%%%%%%%%%%%%%%%%%%%%%%%%%%%%%%%%%

We consider the boundary as a \catun space in its Tits distance. We shall introduce the following notion motivated by the case of buildings: a point at infinity $\xi \in \bd X$ will be called \textbf{regular}\index{regular point}\index{boundary point!regular}\index{point at infinity!regular} if the Tits boundary $\bd X_\xi$ of its transverse space is a round sphere. By Proposition~\ref{prop:Leeb}, this is equivalent to requiring that the space $X_\xi$   contains a flat $F$ with $\bd F = \bd X_\xi$.  By convention, a point $\xi \in \bd X$ such that $X_\xi$ is bounded (and hence $\bd X_\xi$ is empty) will be considered regular as well.

We first record the following basic fact. 

\begin{lem}\label{lem:regular:amen_stab}
Let $\xi \in \bd X$ be cocompact and regular. Then  the stabiliser $G_\xi$ is amenable.
\end{lem}

\begin{proof}
Pick a point $\xi' \in \Op(\xi)$ and apply Theorem~\ref{thm:Levi} to $G_\xi$. This yields a decomposition $G_\xi = G_{\xi, \xi'}   G_\xi^\mathrm{u}$, where $G_\xi^\mathrm{u}$ is amenable and $G_{\xi, \xi'}$ acts cocompactly on $P(\xi, \xi')$. In view of Proposition~\ref{prop:Leeb}, the regularity of $\xi$ ensures that $G_{\xi, \xi'}$ stabilizes a flat $F \subset P(\xi, \xi')$, on which it acts properly since $G_{\xi, \xi'}$ is closed in $G$. The amenability of $G_{\xi, \xi'}$, and hence that of $G_\xi$, follows. 
\end{proof}

Regular points may   be identified using the following criterion. 

\begin{lem}\label{lem:regular:criterion}
Let $\xi \in \bd X$ be cocompact. For any $\xi' \in \Op(\xi)$, the following assertions are equivalent.
\begin{enumerate}[(i)]
\item $\xi$ is regular. 

\item $\bd P(\xi, \xi')$ is a round sphere. 

\item $\bd P(\xi, \xi')$ is a fully maximal sphere. 

\item  $\Link(\xi, \xi')$ is a round sphere. 
\end{enumerate}
\end{lem}

\begin{proof}
The equivalence between (i) and (ii) follows from Proposition~\ref{prop:lifting}. 

By Proposition~\ref{prop:MaxFlat}, any cocompact point   $\xi$ is contained in a fully maximal sphere, which contains some $\xi'' \in \Op(\xi)$ by Proposition~\ref{prop:Leeb}. Since  $G_\xi$ is transitive on $\Op(\xi)$ by Corollary~\ref{cor:transitive}, it follows that $\bd P(\xi, \xi')$ contains a fully maximal sphere. The equivalence between (ii) and (iii) follows

By Lemmas~\ref{lem:Links} and~\ref{lem:P=Pi}, there is a canonical join decomposition $\bd P(\xi, \xi') \cong  \{\xi, \xi'\} \circ  \Link(\xi, \xi')$. 
The equivalence between (ii) and (iv) follows. 
\end{proof}

A crucial step in the proof of Theorem~\ref{thm:NotGoedComplete} is to show the existence of  points in $\bd X$ that are regular and cocompact. This is provided by the following.

\begin{prop}\label{prop:CregNonempty}
Let $(X, A)$ be a locally compact  \cat group. Assume that $A$ is amenable and totally disconnected. Then the set 
$$\Creg = \{\xi \in \bd X \mid \xi \text{ is regular and cocompact}\} \index{$Creg$@$\Creg$}$$
is non-empty.
\end{prop}

This will be deduced from the following result, whose statement requires an additional point of terminology. Following A.~Lytchak and V.~Schroeder~\cite{LytchakSchroeder}, we say that a point $z$ in a \catca space $Z$ is \textbf{geometrically inner}\index{geometrically inner} if there is some $\vareps >0$ such that for all $y \in Z$ with $0<d(y, z)\leq \vareps$, there is some $y' \in Z$ with $d(y', z) = \vareps$ such that $z$ lies on the geodesic segment joining $y$ to $y'$ (one has to take $\vareps$ suitably small when $\kappa>0$). An important fact (see Theorem~1.5 in \emph{loc.~cit.}) is that if $Z$ is (locally) finite-dimensional, then the set of geometrically inner points is dense, and in particular non-empty.

\begin{prop}\label{prop:regular:2}
Let $(X, A)$ be a locally compact  \cat group. Assume that $A$ is amenable and totally disconnected. 
Then any geometrically inner point of the fixed point set $(\bd X)^A$ is cocompact and regular. 
\end{prop}

\begin{proof}
We can of course assume that $(\bd X)^A$ is non-empty. Let $\xi$ be geometrically inner in  $(\bd X)^A$. By definition $\xi$ is cocompact; we need to prove that it is regular. 

We can pick some $\xi' \in \Op(\xi)$ since $\xi$ is cocompact.
Consider the restriction of the projection $\pi_\xi \colon X \to X_\xi$ to $P(\xi, \xi')$. Since $P(\xi, \xi') \cong \RR \times X_\xi$ (see Proposition~\ref{prop:lifting}) and hence $\bd P(\xi, \xi') \cong \{\xi, \xi'\} \circ \bd X_\xi$, we remark that $\pi_\xi$ induces a well defined $A_{ \xi'}$-equivariant map
$$ \bd P(\xi, \xi') \setminus \{\xi, \xi'\} \lra \bd X_\xi$$
which we still denote by $\pi_\xi$.  Moreover, the join decomposition $\bd P(\xi, \xi') \cong \{\xi, \xi'\} \circ \bd X_\xi$ also yields a canonical isometric embedding $\sigma_\xi \colon \bd X_\xi \to \bd P(\xi, \xi')$ whose image coincides with $\Link(\xi, \xi')$. The composite map
$$\tau_\xi = \sigma_\xi \circ \pi_\xi \colon \bd P(\xi, \xi')  \setminus \{\xi, \xi'\}  \lra \bd P(\xi, \xi')$$
is $A_{\xi'}$-equivariant (and coincides with the orthogonal projection on the spherical factor $\bd X_\xi$). Notice furthermore that $\pi_\xi \circ \tau_\xi = \pi_\xi$ because $\pi_\xi\circ \sigma_\xi=\Id$.

By Corollary~\ref{cor:FixedPts:Levi}, we have $(\bd X)^A \subset \bd P(\xi, \xi')$. In particular Theorem~\ref{thm:Levi} implies that $\pi_\xi$ may be viewed as an $A$-equivariant map $(\bd X)^A  \setminus \{\xi, \xi'\} \to \bd X_\xi$. Thus every $A$-fixed point $\eta \neq \xi, \xi'$ yields an $A$-fixed $\pi_\xi(\eta) \in \bd X_\xi$ and an $A_{\xi'}$-fixed point $\tau_\xi(\eta) \in \bd P(\xi, \xi')$ at distance~$\pi /2$ of $\xi$.

\medskip
Suppose in order to obtain a contradiction that $\xi$ is not regular. Then the space $\bd X_\xi$ is non-empty and not isometric to a sphere.  It then follows from Proposition~\ref{prop:RadFixedPt}   that $(\bd X_\xi)^A$ has intrinsic circumradius~$\leq  \pi/ 2$. We denote by $c\in (\bd X_\xi)^A$ its canonical circumcentre (see~\cite[Proposition~1.4]{BalserLytchak_Centers}) and set $c' = \sigma_\xi(c)$. Then $c'$ is fixed by $A_{\xi'}$. In particular, so is the midpoint $m$ of $c'$ and $\xi$. Since $\tangle m \xi = \pi /4$, the projection $\eta=\proj_{(\bd X)^A}(m)$ is well defined.

We are now in a position to invoke Corollary~\ref{cor:proj} and we conclude that   $\eta\neq \xi$.% In particular, $(\bd X)^A \neq \{\xi\}$. 

Since $\xi$ is a geometrically inner point of $(\bd X)^A$ and  since  $\eta \in (\bd X)^A$ is different from $\xi$, there exists $\eta' \in (\bd X)^A$ different from $\xi$ such that $\xi$ lies on the geodesic segment joining $\eta$ to $\eta'$. (We can indeed consider the geodesic segment $[\eta,\xi]$ since $\tangle\eta\xi\leq\tangle m\xi=\pi/4<\pi$.) It follows that $\tau_\xi(\eta)$ and $\tau_\xi(\eta')$ are two antipodal points of $\bd P(\xi, \xi')$. Therefore, so are $\pi_\xi(\eta) = \pi_\xi(\tau_\xi(\eta))$ and  $\pi_\xi(\eta') = \pi_\xi(\tau_\xi(\eta'))$. Thus we have found a pair of antipodal points in  $(\bd X_\xi)^A$. This implies  that the intrinsic circumradius of $(\bd X_\xi)^A$ equals~$\pi /2$ and that in $\bd X_\xi$ we have
$$\tangle{\pi_\xi(\eta)}c  =  \frac \pi 2 = \tangle{\pi_\xi(\eta')}c$$
since $c$ denotes the canonical circumcentre of $(\bd X_\xi)^A$. On the other hand, since $\eta$ and $c'$ lie in $\bd P(\xi, \xi') \cong \{\xi, \xi'\} \circ \bd X_\xi$ and since $\pi_\xi(c')=c$, we have
$$\aangle \xi \eta{c'} = \tangle{\pi_\xi(\eta)}c =  \frac \pi 2$$
and $(\xi,\eta,c')$ spans a spherical right triangle. Thus so does $(\xi,\eta,m)$ since $m$ is the midpoint of $[\xi, c']$. We now apply the Spherical Pythagorean Theorem which states that the cosine of the hypotenuse equals the product of the cosines of lengths of the other two sides. This implies $\tangle m \eta \geq \tangle m \xi$. Recalling that $\eta =  \proj_{(\bd X)^A}(m)$, we deduce $\eta = \xi$, which is absurd. This contradiction finishes the proof.
\end{proof}

\begin{proof}[Proof of Proposition~\ref{prop:CregNonempty}]
If $(\bd X)^A$ is non-empty, the conclusion directly follows from Proposition~\ref{prop:regular:2}. Otherwise, we know by~\cite{AB98} that $A$ stabilises a flat, say $F$. Since $A$ acts cocompactly on $X$, we have $\bd F = \bd X$.  By Lemma~\ref{lem:Bieberbach}, the group $\Fix_A(\bd X)$ is of finite index in $A$, so that every point of $\bd X$ is both regular and cocompact in this case. 
\end{proof}

Now that we have conditions ensuring the existence of cocompact regular points, we record a basic property of such points (for general \cat groups).

\begin{lem}\label{lem:bij}\index{$SXxi$@$\sS_\xi X$}
Let $(X, G)$ be a  locally compact  \cat group   and $\xi\in\bd X$ a cocompact regular point.

Then there is a canonical bijection between $\Op(\xi)$ and the set $\sS_\xi X$  of fully maximal spheres containing $\xi$. In particular, $G_\xi^{\mathrm u}$ acts transitively on $\sS_\xi X$.
\end{lem}

\begin{proof}
The set $\sS_\xi X$is non-empty by Proposition~\ref{prop:MaxFlat}. Every $S\in \sS_\xi X$ contains a unique antipode for $\xi$ and this antipode is opposite $\xi$ by Proposition~\ref{prop:Leeb}.

To go in the reverse direction, observe first that $\bd X_\xi$ is a sphere of dimension $\dim\bd X-1$ since $\xi$ is regular and contained in a fully maximal sphere. Thus any $\xi'\in\Op(\xi)$ gives rise to a fully maximal sphere $\bd P(\xi,\xi')$ since $P(\xi,\xi')\cong \RR \times X_\xi$ by Proposition~\ref{prop:lifting}. These maps are mutually inverse and the transitivity statement now follows from Corollary~\ref{cor:transitive}.
\end{proof}

%%%%%%%%%%%%%%%%%%%%%%%%%%%%%%%%%%%%%%%%%%%%%%%%%%%%%%%%%%%%%%%%%%%%%%%%%%%%%%
\subsection{Spherical caps of  cocompact regular points }
%%%%%%%%%%%%%%%%%%%%%%%%%%%%%%%%%%%%%%%%%%%%%%%%%%%%%%%%%%%%%%%%%%%%%%%%%%%%%%
Recall from Proposition~\ref{prop:MaxFlat} that every cocompact point $\xi \in \bd X$ is contained in a fully maximal sphere. Therefore, the following definition provides a non-empty set at least in that case.

\begin{defn}
The \textbf{spherical support}\index{spherical support}\index{$Sz$@$\Sigma(\xi)$} $\Sigma(\xi) = \bigcap \sS_\xi X\se \partial X$ of a point at infinity $\xi$ is the  intersection of all fully maximal spheres containing $\xi$.
\end{defn}

Thus $\Sigma(\xi)$ can be considered as a $\pi$-convex spherical set. Much more can be said if $\xi$ belongs to the set $\Creg$ of regular cocompact points:

\begin{prop}\label{prop:regular:1}
Let $(X, G)$ be a locally compact  \cat group. Assume that $G$ is totally disconnected.

Then the spherical support $\Sigma(\xi)$ of any regular cocompact point $\xi \in \Creg$ satisfies the following properties.
\begin{enumerate}[(i)]
\item $\dim(\Sigma(\xi)) = \dim(\bd X)$.

\item The set of regular points in $\Sigma(\xi)$ coincides with the set of geometrically inner points  of $\Sigma(\xi)$; in particular it is dense. 

\item For each regular $\zeta \in \Sigma(\xi)$, we have $\Sigma(\zeta) = \Sigma(\xi)$. 

\item The pointwise stabilizer of $\Sigma(\xi)$ is contained as a finite index subgroup of $G_\xi$. In particular every point of $\Sigma(\xi)$ is cocompact.\label{pt:regular:1:finind}

\item $\Sigma(\xi)$ contains finitely many points from the $G$-orbit of $\xi$. 
\end{enumerate}
\end{prop}

The above statement implies in particular that for each $\xi\in\Creg$ and each $S\in\sS_\xi X$, there is some spherical cap $S \cap B(\xi, \vareps)$ of full dimension consisting entirely of regular cocompact points.

\begin{proof}[Proof of Proposition~\ref{prop:regular:1}]
Let $\xi \in \Creg$.  Its stabiliser $A = G_\xi$ is amenable by Lemma~\ref{lem:regular:amen_stab}, so that $(X, A)$ is an amenable, totally disconnected, locally compact \cat group.

By Proposition~\ref{prop:MaxFlat}, the point $\xi$ is contained in a fully maximal sphere, say $S$. Let $\xi' \in S$ be the point antipodal (hence opposite by Proposition~\ref{prop:Leeb}) to $\xi$.  Since $\xi$ is regular, we have $S = \bd P(\xi, \xi')$ by Proposition~\ref{prop:lifting}, so that $A_{\xi'} = \Stab_A(S)$. 
By Lemma~\ref{lem:Bieberbach}, the group $\Fix_A(S)$ is of finite index in $\Stab_A(S)$. In view of Theorem~\ref{thm:Levi}, we infer that $A$ has a closed normal subgroup of finite index
$$B := \Fix_A(S) \cdot A^{\mathrm u} \ \lhd\  \Stab_A(S) \cdot A^{\mathrm u} =A$$
such that $\Stab_B(S) = B_{\xi'}$ acts trivially on $S  = \bd P(\xi, \xi')$. Applying Corollary~\ref{cor:FixedPts:Levi} to $B$, we have
$$(\bd X)^B \subset (\bd X)^{\Stab_B(S)}= S.$$ 
At this point, we observe that all desired conclusions follow easily in the special case $(\bd X)^B = S$. Indeed, $B$ then preserves a fully maximal flat by Proposition~\ref{prop:Leeb} and hence $\partial X=S$ because $B$ is cocompact. We assume henceforth that $(\bd X)^B$ is a proper subset of $S$. 

\begin{claim} \label{claim:1}
Each regular point $\zeta \in (\bd X)^B$ is contained in the $S$-interior of $(\bd X)^B$ (i.e. the interior of $(\bd X)^B$ viewed as a subspace of $S$). In other words, for each regular point $\zeta \in (\bd X)^B$,   there is $\vareps >0$ such that $S \cap B(\zeta, \vareps) \subset (\bd X)^B$. 
\end{claim}

Since $(\bd X)^B$ is closed,  convex and properly contained in $S$, it is an intersection of closed hemispheres. Therefore, if the claim fails, then there is a sequence $(H_n)$ of closed hemispheres of $S$ containing $(\bd X)^B$, such that the distance from $\zeta$ to the boundary equator of $H_n$ tends to~$0$ as $n$ tends to infinity. Upon extracting, we may assume that $(H_n)$ converges to a closed hemisphere $H$. By construction $H$ contains $(\bd X)^B$ and its boundary equator contains $\zeta$. We now pick  a point $\eta$ in $S \setminus H$ such that $\tangle \eta \zeta = \pi/ 4$ and  whose projection to $ H$ coincides with $\zeta$. Since $(\bd X)^B \subseteq H$, it follows that $\proj_{(\bd X)^B}(\eta) = \zeta$. Since $\Stab_B(S)$ acts trivially on $S$, it fixes $\eta$ as well as the unique antipode $\zeta'$ of $\zeta$ in $S$. Since $\zeta$ is regular, we have $S=\bd P(\zeta, \zeta')$ by Lemma~\ref{lem:regular:criterion} and hence $B_{\zeta'} = \Stab_B(S)$. Therefore, Corollary~\ref{cor:proj} yields the absurd conclusion that $\proj_{(\bd X)^B}(\eta) \neq \zeta$. This proves the claim.

\begin{claim}\label{claim:2}
For each regular point $\zeta \in (\bd X)^B$, we have $\Sigma(\zeta)= (\bd X)^B$. 
\end{claim}

Let $\zeta \in (\bd X)^B$ be regular. Since $B$ fixes $\zeta$ and acts cocompactly on $X$, we can apply Lemma~\ref{lem:bij} and deduce that $B$ acts transitively on the set $\sS_\zeta$ of fully maximal spheres containing $\zeta$, and we infer that $(\bd X)^B \subset \Sigma(\zeta)$. 

For the reverse inclusion, we suppose for a contradiction that some $z \notin (\bd X)^B$ belongs to $\Sigma(\zeta)$. In particular, $z\in S$ and we have a spherical geodesic segment $[\zeta, z]$. By Claim~\ref{claim:1}, this segment has an initial subsegment of positive length fixed by $B$. Let $b \in B$ which does not fix $z$; we conclude that $bz\notin S$ since otherwise $S$ would contain branching geodesics $[\zeta, z]$ and $[\zeta, bz]$. In particular, $bz\notin\Sigma(\zeta)$ and hence $z$ is not in $\Sigma(\zeta)$ because $\Sigma(\zeta)$ is $B$-invariant since $\zeta$ is $B$-fixed. This contradiction finishes the proof of the claim.

\medskip
Assertions (i), (iii) and (iv) are now immediate consequences of Claims~\ref{claim:1} and~\ref{claim:2}. For (ii), we know  that  every regular point of $(\bd X)^B$ is contained in the $S$-interior of $(\bd X)^B$  by Claim~\ref{claim:1}; in particular, it is geometrically inner in $\Sigma(\xi)$ since $\Sigma(\xi)=(\bd X)^B$. Conversely, every $S$-interior point of $(\bd X)^B$ is geometrically inner by Lemma~\ref{lem:AntipodalPt}, and hence regular and cocompact by Proposition~\ref{prop:regular:2}. Since the set of geometrically inner points is dense by~\cite[Theorem~1.5]{LytchakSchroeder}, we deduce that the set of regular points of $(\bd X)^B$
 is dense in $(\bd X)^B$.   This proves (ii). 

In order to prove~(v), let $g \in G$ with $g\xi \in \Sigma(\xi)$. Since  $g\xi$ is regular, we have $\Sigma(\xi) = \Sigma(g\xi) = g\Sigma(\xi)$ by (iii). Hence $g$ stabilises $\Sigma(\xi)$, so it suffices to prove that the orbit of $\xi$ under $\Stab_G(\Sigma(\xi))$ is finite. Since $B$ is cocompact, it acts transitively on $\sS_\xi$. Hence the $\Stab_G(\Sigma(\xi))$-orbit of $\xi$ is contained in its $\Stab_G(S)$-orbit. The assertion follows since $\xi \in S$ and since $\Stab_G(S)/\Fix_G(S)$ is finite by Lemma~\ref{lem:Bieberbach}.
\end{proof}

%%%%%%%%%%%%%%%%%%%%%%%%%%%%%%%%%%%%%%%%%%%%%%%%%%%%%%%%%%%%%%%%%%%%%%%%%%%%%%
\subsection{Amenable \cat groups and compactions}
%%%%%%%%%%%%%%%%%%%%%%%%%%%%%%%%%%%%%%%%%%%%%%%%%%%%%%%%%%%%%%%%%%%%%%%%%%%%%%

Following~\cite{CCMT}, we say that an automorphism $\alpha$ of a locally compact group $H$ is   \textbf{compacting}\index{compacting automorphism} or is a \textbf{compaction}\index{compaction} if there exists a compact subset $V \subset H$ such that for each $g \in H$ there exists $n_0 \geq 0$ with $\alpha^n(g) \in V$ for all $n \geq n_0$. It follows from~\cite[Theorem~A]{CCMT} that if a locally compact group is amenable and non-elementary hyperbolic, then it has a cocompact subgroup which is the semi-direct product of a closed subgroup $H$ by a cyclic group generated by a compaction. The following consequence of Proposition~\ref{prop:regular:1} is a \cat analogue of that fact, in the case of totally disconnected groups. We record this result but shall not need it for the main results of this article.

\begin{prop}\label{prop:Compaction}
Let $(X, A)$ be a locally compact  \cat group. Assume that $A$ is totally disconnected and amenable.

Then $A$ has  closed normal subgroups $U \leq T \leq A$ such that the quotient $A/T$ is finite and the quotient  $T/U$ is free abelian of rank $\dim(\bd X)+1$ generated by   elements   of $T$ acting  on $U$ as compactions. 
\end{prop}

\begin{proof}
If $(\bd X)^A$ is empty, then $A$ stabilises a flat $F$ by~\cite{AB98}. We may then define $U$ as the kernel of the $A$-action on $F$, which is compact. The desired assertions are then straightforward consequences of Bieberbach's theorem (as in the proof of Lemma~\ref{lem:Bieberbach}).

We assume henceforth that  $(\bd X)^A$ is non-empty. By Proposition~\ref{prop:regular:2}, the fixed point set $(\bd X)^A$ then contains a regular point, say $\xi$. As in the proof of Proposition~\ref{prop:regular:1}, we consider a fully maximal sphere $S$ containing $\xi$ and a $\Stab_A(S)$-invariant  fully maximal flat $F$ with $\bd F = S$. We set $U = \Fix_A(F) A_\xi^{\mathrm u}$ and  $B = \Fix_A(S) A_\xi^{\mathrm u}$, so that $U \leq B$ are  both closed normal subgroups of $A$. Moreover, $A/B$ is finite by Proposition~\ref{prop:regular:1}\eqref{pt:regular:1:finind} and Lemma~\ref{lem:Bieberbach}. Thus $B$ acts cocompactly on $X$ and $B/U$ is free abelian of rank $\dim(\bd X)+1$. By Theorem~\ref{thm:Levi}, the group $A_\xi^{\mathrm u}$ is locally elliptic. Thus $U$ is locally elliptic; in fact, $U$ coincides with the locally elliptic radical of $B$. Indeed, any compact subgroup of $\Fix_A(S)$ is contained in $\Fix_A(F)$ since it must fix a point of $F$.

The group $\Fix_A(S)$ acts on $F$ by translations, and the action is cocompact by Proposition~\ref{prop:lifting}. Consider a translation $t \in \Fix_A(S)$ whose repelling fixed point $\eta$ in $S$ belongs to the interior of set $\Sigma(\xi)$. We claim that $t$ acts on $U$ as a compaction. 

Let indeed $\eta' \in S$ be the antipode of $\eta$. The pointwise stabiliser  $ \Fix_B(F)$ is compact; thus it has a compact neighbourhood $V$ in $A$. Applying Theorem~\ref{thm:Levi} to the group $B$ and the point $\eta$, we obtain the decomposition $B =  B_\eta^{\mathrm u} B_{\eta'}$, where $B_{\eta}^{\mathrm u}$ is locally elliptic and hence in $U$. Since $\eta$ is regular, we deduce that $U= \Fix_B(F) B_\eta^{\mathrm u}$. Therefore, for each $g \in U$, the sequence $(t^n g t^{-n})_n$ is bounded in $B$. Any accumulation point  fixes also $\eta'$, and hence belongs to $B_{\eta'} \cap U = \Fix_B(F)$. This implies that all but finitely many elements of the sequence belong to $V$, whence the claim.

We finally define $T$ as the subgroup of $B$ generated by $U$ together with all those translations $t \in \Fix_G(S)$ whose repelling fixed point belongs to the interior of $\Sigma(\xi)$. All those translations are compactions of $U$ by the claim. Moreover $T$ is normal in $A$ because $A$ stabilises $\Sigma(\xi)$. That $B/T$ is finite follows from Proposition~\ref{prop:regular:1}(i).
\end{proof}

%%%%%%%%%%%%%%%%%%%%%%%%%%%%%%%%%%%%%%%%%%%%%%%%%%%%%%%%%%%%%%%%%%%%%%%%%%%%%%
\section{The visual boundary is a spherical building}\label{sec:SphBdg}
%%%%%%%%%%%%%%%%%%%%%%%%%%%%%%%%%%%%%%%%%%%%%%%%%%%%%%%%%%%%%%%%%%%%%%%%%%%%%%

The goal of this section is to prove that the boundary of $X$ is a metric spherical building under the hypothesis of Theorem~\ref{thm:NotGoedComplete}, in case of a totally disconnected isometry group. A key intermediate step consists in showing that the set $C$ of cocompact points in the visual boundary is a spherical building, see Proposition~\ref{prop:SphBldg} below. We have already obtained spherical caps of full dimension consisting entirely of cocompact points thanks to Propositions~\ref{prop:CregNonempty} and~\ref{prop:regular:1}. In fact the sets $\Sigma(\xi)$ described by Proposition~\ref{prop:regular:1} will be the chambers of the spherical building $C$. The next step is to construct fully maximal spheres all of whose points are cocompact. This is achieved by Proposition~\ref{prop:CocptSphere} below. The proof that $C$ is a spherical building will be presented thereafter. Finally, the last step in the proof of Theorem~\ref{thm:NotGoedComplete} will be to show that e
 very point of $\bd X$ is in fact contained in $C$.

For the sake of completeness, the formal definition of a metric spherical building was given in \S\ref{sec:Buildings} above. However, in proving that $\bd X$ is a spherical building, we will not check the building axioms directly, but will rather invoke the criterion due to Balser--Lytchak recalled in Theorem~\ref{thm:BalserLytchak}. Once again, we first need to assemble preliminaries.

%%%%%%%%%%%%%%%%%%%%%%%%%%%%%%%%%%%%%%%%%%%%%%%%%%%%%%%%%%%%%%%%%%%%%%%%%%%%%%
\subsection{Existence of tricycles}
%%%%%%%%%%%%%%%%%%%%%%%%%%%%%%%%%%%%%%%%%%%%%%%%%%%%%%%%%%%%%%%%%%%%%%%%%%%%%%

\begin{flushright}
\begin{minipage}[t]{0.75\linewidth}\itshape\small
Finally you see that while I was splitting the cycle up into finer and finer pieces, I was also building a structure.\upshape
\begin{flushright}
(R.M.~Pirsig, \emph{Zen and the Art of Motorcycle Maintenance}, 1974)
\end{flushright}
\end{minipage}
\end{flushright}

%The following supplements Proposition~\ref{prop:regular:1}.

\begin{lem}\label{lem:tricycle:1}
Let $(X, A)$ be a locally compact  \cat group. Assume that $A$ is amenable and totally disconnected. Let $S \subsetneqq \bd X$ be a fully maximal sphere containing an $A$-fixed point. Then there exists $g \in A$ such that $S \cup gS$ is a tricycle.
\end{lem}

\begin{proof}
We work by induction on $\dim(\bd X)$. In the base case $\dim(\bd X)=0$, the space $X$ is Gromov-hyperbolic by~\cite[Theorem~III.H.1.5]{Bridson-Haefliger}, so any two boundary points are opposite. We have $S = \{\xi, \xi'\}$ with $\xi \in (\bd X)^A$ and it suffices to find an element $g \in A$ which does not fix $\xi'$. Since $A$ is transitive on $\Op(\xi) = \bd X \setminus \{\xi\}$ by Corollary~\ref{cor:transitive}, the non-existence of such a $g$ implies that $\bd X = \{\xi, \xi'\}= S$, which contradicts the hypothesis that $S$ is properly contained in $\bd X$. 

Assume now that $\dim(\bd X) >0$. By Proposition~\ref{prop:regular:2} there is a regular cocompact point $\xi \in (\bd X)^A$. Let $B$ be the pointwise stabiliser of $\Sigma(\xi)$, which has finite index in $A$ by Proposition~\ref{prop:regular:1}(iv).

Since $\dim(\bd X) >0$, Proposition~\ref{prop:regular:1}(ii) ensures that there exists a non-regular point $\eta \in \Sigma(\xi)$.  In particular $\eta \neq \xi$ and the boundary $\bd X_{\eta}$ is not a round sphere. Let $\eta' \in S$ be opposite to $\eta$. We have a spherical join decomposition $\bd P(\eta, \eta') \cong \{\eta, \eta'\} \circ \bd X_\eta$ and a canonical map $\pi_\eta \colon \bd P(\eta, \eta') \setminus \{\eta, \eta'\} \to \bd X_\eta$. The   image of $S \setminus \{\eta, \eta'\} $ under $\pi_\eta$ is a  fully maximal sphere $S'$ containing the  point $\xi' = \pi_\eta(\xi)$ which is fixed by the image of $B$ in $\Isom(X_\eta)$. Since $B$ acts continuously and cocompactly on $X_\eta$ (see Proposition~\ref{prop:lifting}), we may apply the induction hypothesis. This affords an element $g' \in B$ such that  $S' \cup g'S'$ is a tricycle in $\bd X_\eta$.

By Theorem~\ref{thm:Levi}, we may write $g' = gu$ where $g \in B_{\eta'}$ and $u$ acts trivially on $X_\eta$. Thus $S' \cup g' S' = S' \cup g S'$. In view of the decomposition $\bd P(\eta, \eta') \cong \{\eta, \eta'\} \circ \bd X_\eta$, it now follows that $S \cup g S$ is isometric to $\{\eta, \eta'\} \circ (S' \cup g S' )$, and is thus a tricycle. 
\end{proof}

For the sake of future references, we also record the following.

\begin{lem}\label{lem:tricycle:2}
Let $(X, G)$ be a locally compact  \cat group with  $G$   totally disconnected.
Given a regular cocompact point $\xi \in \bd X$ and a non-regular point $\eta \in \Sigma(\xi)$, there exists a tricycle of full dimension containing $\xi$ and whose equator contains $\eta$. 
\end{lem}

\begin{proof}
Let $A$ be pointwise stabiliser of $\Sigma(\xi)$, so that $(X, A)$ is an amenable \cat group by Lemma~\ref{lem:regular:amen_stab} and Proposition~\ref{prop:regular:1}(iv). 

Let  $S$ be a fully maximal sphere containing $\xi$ and $\eta' \in S$ be the point antipodal to $\eta$. Then $A_{\eta'}$ acts cocompactly on $P(\eta, \eta')$ by Proposition~\ref{prop:lifting}. Since $\eta$ is not regular, we have $S \subsetneqq \bd P(\eta, \eta')$. Applying Lemma~\ref{lem:tricycle:1} to the locally compact \cat group $(P(\eta, \eta'), A_{\eta'})$, we obtain $g \in A_{\eta'}$ such that $S \cup gS$ is a tricycle. Since $\{\eta, \eta'\}$ is a spherical factor of that tricycle, it must be contained in its equator.
\end{proof}

%%%%%%%%%%%%%%%%%%%%%%%%%%%%%%%%%%%%%%%%%%%%%%%%%%%%%%%%%%%%%%%%%%%%%%%%%%%%%%
\subsection{Existence of spherical reflections}
%%%%%%%%%%%%%%%%%%%%%%%%%%%%%%%%%%%%%%%%%%%%%%%%%%%%%%%%%%%%%%%%%%%%%%%%%%%%%%

\begin{lem}\label{lem:refl}
Let $(X, G)$ be a locally compact  \cat group with  $G$   totally disconnected.  Let $S, T \subset \bd X$ be a fully maximal spheres such that $S \cup T$ is a tricycle with equator $E$.

If both open hemispheres of $S$ bounded by $E$ contain regular cocompact points,  then $\Stab_G(S)$ contains the orthogonal reflection through $E$. 
\end{lem}

We start with the following subsidiary claim. 

\begin{lem}\label{lem:RootFixator}
Let $(X, G)$ be a locally compact  \cat group with  $G$   totally disconnected.  

Let $H_0, H_1, H_2$ be the three  closed hemispheres of $Z$ with $\dim H_i = \dim \bd X$ and $H_0 \neq H_1, H_2$, having a common boundary equator, say $E$. If  $H_0$ contains a regular cocompact point, then there exists $g \in G$ fixing $H_0$ pointwise with $gH_1 = H_2$. 
\end{lem}

\begin{proof}
Let $\xi_0 \in H_0$ be regular cocompact. By Lemma~\ref{lem:criterion:tricycle}, both sets $H_0 \cup H_1$ and $H_0 \cup H_2$ are round spheres. In particular, for $i \in \{1, 2\}$, the hemisphere $H_i$ contains a unique antipode $\xi_i$ of   $\xi_0$. By Lemma~\ref{lem:regular:amen_stab}, the group $G_{\xi_0}$ is amenable. Moreover it acts cocompactly on $X$, and contains the pointwise stabiliser of $\Sigma(\xi_0)$, say $B$,  as a finite index subgroup, see Proposition~\ref{prop:regular:1}(iv). In particular $B$ is cocompact on $X$, hence transitive on $\Op(\xi_0)$ by Corollary~\ref{cor:transitive}. Since $\xi_1, \xi_2 \in \Op(\xi_0)$ by Proposition~\ref{prop:Leeb}, we find $g \in B$ with $g\xi_1 = \xi_2$. Since $\xi_0$ is regular, we have $H_0 \cup H_i = \bd P(\xi_0, \xi_i)$ for $i=1, 2$. Therefore $g(H_0 \cup H_1) = H_0 \cup H_2$. Since $g$ moreover fixes a neighbourhood of $\xi_0$ in $H_0$, it follows that $g$ fixes $H_0$ pointwise. 
\end{proof}

\begin{proof}[Proof of Lemma~\ref{lem:refl}]
Let now $H$ and $H'$ be the two closed hemispheres of $S$ bounded by $E$, and let $H''$ be the closure of $T \setminus S$. By hypothesis, both $H$ and $H'$ contain an interior point which is regular cocompact. By Lemma~\ref{lem:RootFixator}, there is $g \in G$ fixing $H$ pointwise and mapping $H'$ to $H''$. In particular $H''$ contains an interior point which is regular cocompact.

A second application of Lemma~\ref{lem:RootFixator} yields an element $h_1 \in G$ fixing $H'$ pointwise and with $h_1 H''=H$. Noting that $H \cup H' \cup g\inv(H')$ is a tricycle by Lemma~\ref{lem:criterion:tricycle}, we may apply Lemma~\ref{lem:RootFixator} a third time so as to obtain an element $h_2 \in G$ fixing $H'$ pointwise and with $h_2 H= g\inv H'$.

Now we set $r = h_1 g h_2 \in G$. Then $r$ fixes $E$ pointwise since $g, h_1  $ and $h_2$ all do. Moreover we have 
$$r H  = h_1g h_2 H  = h_1 g g\inv H' = h_1 H' = H'$$
and
$$r H' = h_1g h_2 H' =   h_1g   H'= h_1 H''= H.$$
Thus $r$ preserves $S =  H  \cup H'$ and acts on $S$ as a spherical reflection fixing the equator $E$ pointwise, as desired.
\end{proof}

%%%%%%%%%%%%%%%%%%%%%%%%%%%%%%%%%%%%%%%%%%%%%%%%%%%%%%%%%%%%%%%%%%%%%%%%%%%%%%
\subsection{Existence of a sphere of cocompact points}
%%%%%%%%%%%%%%%%%%%%%%%%%%%%%%%%%%%%%%%%%%%%%%%%%%%%%%%%%%%%%%%%%%%%%%%%%%%%%%

The main result of this subsection is the following where, as in Proposition~\ref{prop:CregNonempty} above, the set of regular cocompact points of $\bd X$ is denoted by $\Creg$.

\begin{prop}\label{prop:CocptSphere}
Let $(X, G)$ be a locally compact  \cat group with  $G$   totally disconnected without a fixed point at infinity.  Let $S$ be  a fully maximal sphere   containing a regular cocompact point. 

For each $s \in S$, there exists $\xi \in \Creg \cap S$ such that $s \in \Sigma(\xi)$. In particular all points of $S$ are cocompact. 
\end{prop}

The proof is divided into several steps, each stated in a separate lemma. The first one is valid without the hypothesis of absence of fixed points at infinity. 

 \begin{lem}\label{lem:CocptSphere:step1}
Let $(X, G)$ be a locally compact  \cat group with  $G$   totally disconnected. Let $S$ be  a fully maximal sphere and let $\xi \in \Creg \cap S$. Let $Z$ denote the closed convex hull of $\Creg \cap S$ and set $W = \Stab_G(S)/\Fix_G(S)$. 

For each $z \in Z$, there is $w \in W$ such that $wz \in \Sigma(\xi)$.
In particular the closure of $\Creg \cap S$ is convex, and all points of $Z$  are cocompact. 
\end{lem}
\begin{proof}
The set $\Creg  \cap \Sigma(\xi)$ is open and dense in $\Sigma(\xi)$ by Proposition~\ref{prop:regular:1}(ii). Moreover every point of $\Sigma(\xi)$ is cocompact by Proposition~\ref{prop:regular:1}(iv). In particular $\Sigma(\xi) \subset Z$.

By Lemma~\ref{lem:Bieberbach}, the  group $W = \Stab_G(S)/\Fix_G(S)$ is finite.  Therefore, we may choose $\zeta \in \{wz \mid w \in W\}$ at minimal  distance from $\xi$. 
We next consider a geodesic segment 
$\gamma$ joining $\xi$ to $\zeta$. If $\tangle \zeta \xi < \pi$, then  $\gamma$ is entirely contained in $Z$  since $Z$ is $\pi$-convex by definition. In case $\tangle \zeta \xi = \pi$, we remark that $\gamma$ contains some points of $\Sigma(\xi)$ other than $\xi$ since $\xi$ lies in the interior of $\Sigma(\xi)$. This  implies that $\gamma$ is entirely contained in $Z$ in that case as well. 

Let now $\eta$ the unique
point of $\gamma$ such that
$$\gamma \cap \Sigma(\xi) =
[\xi, \eta].$$ %
If $\eta = \zeta$, then $\zeta \in \Sigma(\xi)$
and we are done. Otherwise, Proposition~\ref{prop:regular:1}(ii) ensures that
the point $\eta$ is not regular. In particular $\eta \neq \xi$ and hence $\eta$ is an interior point of the geodesic segment $\gamma$. 

By Lemma~\ref{lem:tricycle:2}, there is a fully maximal sphere $T$ such that $S \cup T$ is a tricycle whose equator, say $E$, contains $\eta$. Remark that   $E$ can be expressed as the intersection of three different fully maximal spheres. Therefore Proposition~\ref{prop:regular:1} ensures that $E$ does not contain any regular cocompact point. In particular we have $\xi \not \in E$, so $\xi$ belongs to one of the two open hemispheres of $S$ bounded by $E$.  
This  implies that $\gamma \cap E = \{\eta\}$. In particular $\zeta$ lies in the other open hemisphere of $S$ bounded by $E$. Therefore $Z$ is note entirely contained in the closed hemisphere $H$ bounded by $E$ and containing $\xi$. It follows that $\Creg \cap S$   meets both of the open hemispheres of $S$ determined by $E$.  Hence Lemma~\ref{lem:refl} affords a reflection $r \in W$  fixing $E$ pointwise. In particular $r$ fixes $\eta$. Since neither $\xi $ nor $\zeta$ belongs to $E$, it follows that    $r\zeta$ is   strictly closer to $\xi$ than $\zeta$. This contradicts the definition of $\zeta$. 

The other claimed assertions follow from Proposition~\ref{prop:regular:1}(ii) and~(iv).
\end{proof}

In the proof of Proposition~\ref{prop:CocptSphere}, the hypothesis of absence of fixed points at infinity will be exploited through the following. 

\begin{lem}\label{lem:BdFixedPt}
Let $(X, G)$ be a locally compact  \cat group with  $G$   totally disconnected without a fixed point at infinity.  Let $S$ be  a fully maximal sphere   containing a regular cocompact point. 

For each cocompact point $z \in S$, there exists  $z' \in \Creg \cap S$ with $\tangle z {z'} > \frac \pi 2$.
\end{lem}

\begin{proof}
We first claim that there exists some  regular cocompact 
point $\eta \in \bd X$ such that  $\tangle z \eta > \pi /2$. In order to establish this, we let $\xi \in S$ be the regular   cocompact given by hypothesis and define
$$s = \sup_{g \in G} \tangle z {g \xi}.$$

If $s <  \pi/ 2$, then the $G$-orbit of $\xi$ is a $G$-stable set of radius~$< \pi/ 2$. By~\cite[II.2.7]{Bridson-Haefliger}, this orbit admits therefore a unique circumcentre which is $G$-fixed, contradicting the assumptions.

If $s  > \pi /2$, then there exists $g \in G$ such that $\tangle z {g\xi} > \pi /2$ and we may define $\eta= g\xi$.

Assume now that $s =  \pi/ 2$. Let $(g_n)$ be a sequence of elements of $G$ such that
$$\lim_n \tangle z {g_n\xi} =\frac \pi 2.$$
Let also $\vareps>0$ be such that the spherical cap $S \cap B(\xi, \vareps)$ is entirely contained in the interior of  $\Sigma(\xi)$ and thus consists of cocompact regular points, see Proposition~\ref{prop:regular:1}(ii). For $n$ large enough, we have $\tangle z {g_n\xi} >  (\pi-\vareps)/ 2$. By Lemma~\ref{lem:AntipodalPt}, the geodesic segment joining $z$ to $g_n\xi$ may be prolonged in $\Sigma(g_n \xi)$ beyond $g_n \xi$ by a piece of length $\vareps$ consisting entirely of regular cocompact points. By construction, the points lying of the second half of this geodesic extension are cocompact regular points  at distance~$> \pi /2$ of $z$, and we define $\eta$ to be one of these points. This proves the claim.

\smallskip
Set $\delta = \tangle z \eta - \pi /2$. Thus $\delta >0$. By Lemma~\ref{lem:CocptSphere:step1}, there exists $\xi_0 \in \Creg \cap S$ such that $\tangle
{\xi_0}  z < \delta /2$. Let $\xi'_0 \in S$ be the point antipodal to $\xi_0$. Since $\xi_0$ is regular, we have $S = \bd P(\xi_0, \xi'_0)$ so that Proposition~\ref{prop:2CocptPts}
provides a  point
$$\eta' \in
\overline{G_{\xi_0}\eta} \cap S$$
with  $\tangle {\xi_0} {\eta'}
= \tangle {\xi_0} \eta$. By Lemma~\ref{lem:CocptOrbit}, the
$G$-orbit of $\eta$ is closed, hence $\eta' \in G\eta$ and $\eta'$
is thus regular and cocompact. Invoking twice the triangle inequality, we obtain successively
$$\begin{array}{rcl}
\tangle z {\eta'} & \geq & \tangle {\xi_0} {\eta'} - \tangle {\xi_0} z\\
& = & \tangle {\xi_0} {\eta} - \tangle {\xi_0} z\\
& \geq & \tangle {z} {\eta} - 2 \tangle {\xi_0} z\\
& > &  \tangle {z} {\eta} -  \delta\\
& = & \frac \pi 2.
\end{array}$$
Thus we have constructed a regular cocompact point $\eta' \in \Creg \cap S$
with $\tangle z {\eta'}  > \pi /2$. 
\end{proof}

\begin{proof}[Proof of Proposition~\ref{prop:CocptSphere}]
Let $\Creg$ be the set of regular cocompact points and $Z$ denote the closed  convex hull of $\Creg \cap S$. In view Lemma~\ref{lem:CocptSphere:step1}, all points of $Z$ are cocompact, and it suffices to show that $Z = S$. By Lemma~\ref{lem:BdFixedPt} we have $\rad(Z)>\pi/2$. Hence Lemma~\ref{lem:RadSphericalConvex} ensures that $Z$ is a round sphere. Since $Z$ has non-empty interior by Proposition~\ref{prop:regular:1}, we conclude that $Z=S$, as required.
\end{proof}

%%%%%%%%%%%%%%%%%%%%%%%%%%%%%%%%%%%%%%%%%%%%%%%%%%%%%%%%%%%%%%%%%%%%%%%%%%%%%%
\subsection{Antipodal pairs of cocompact points are opposite}
%%%%%%%%%%%%%%%%%%%%%%%%%%%%%%%%%%%%%%%%%%%%%%%%%%%%%%%%%%%%%%%%%%%%%%%%%%%%%%

We recall that two boundary points of a \cat space are called antipodal if their Tits-distance is~$\geq \pi$, and opposite if they are the end points of a geodesic line. Opposite points are always antipodal, but the converse need not hold in general. As before, we denote by $C$ (resp. $\Creg$) the set of cocompact (resp. regular cocompact) points. The goal of this subsection is to prove the following. 

\begin{lem}\label{lem:AntipodalPair}
Let $(X, G)$ be a locally compact  \cat group with  $G$   totally disconnected. Any antipodal pair $\{\xi, \zeta\}$ with   $\xi \in \Creg$ and $\zeta \in C$  is contained in a common   fully maximal sphere. In particular    $\xi$ and $\zeta$ are opposite.
\end{lem}

We shall use the following fact.

\begin{lem}\label{lem:7.14}
Let $(X, G)$ be a locally compact  \cat group. Given a cocompact point $\zeta \in \bd X$,  the set of all cocompact points in $\Ant(\zeta)$ is contained in a single $G$-orbit.
\end{lem}

\begin{proof}
The following argument is borrowed from the proof of Lemma~7.14 from~\cite{Caprace-Monod_structure} (whose hypotheses are too strong to be invoked directly); we include it here for the sake of completeness. Let $\eta \in \Ant(\zeta)$ be cocompact. Choose a radial sequence $(g_n) \subset G_\zeta$ for $\eta$ and let $\eta'$ be an accumulation point of the sequence $(g_n \eta)$. By Proposition~\ref{prop:upper:semi}, the points $\zeta$ and $\eta'$ are now opposite. Since $\eta$ is cocompact, it follows from Lemma~\ref{lem:CocptOrbit} that its $G$-orbit is closed. Thus we have shown that the $G$-orbit of $\eta$ contains a point $\eta'$ which is opposite $\zeta$. The conclusion follows since by cocompactness of $\zeta$, the stabiliser $G_\zeta$ is transitive on $\Op(\zeta)$ (see Corollary~\ref{cor:transitive}).
\end{proof}

\begin{proof}[Proof of Lemma~\ref{lem:AntipodalPair}]
%By Proposition~\ref{prop:MaxFlat}, the point $\xi$ is contained in some fully maximal sphere $S$. 
By Proposition~\ref{prop:regular:1}(ii) and~(v), there is an open spherical cap $U \subset \Sigma(\xi)$ containing $\xi$ and such that $U \cap G.\xi = \{\xi\}$, where $G.\xi$ denotes  the $G$-orbit of $\xi$. Since every point of $U$ is cocompact by Proposition~\ref{prop:regular:1}(iv), it follows from Lemma~\ref{lem:7.14} that $U \cap \Ant(\zeta)= \{\xi \}$. %In other words we have $\tangle z \zeta < \pi $ for all $z \in U \setminus \{\xi\}$. 
Therefore, the Parachute Lemma~\ref{lem:para}  ensures that $\xi$ and $\zeta$ are contained in a common fully maximal sphere. Since any such sphere bounds a fully maximal flat by Proposition~\ref{prop:Leeb}, the points $\xi$ and $\zeta$ are indeed opposite.
\end{proof}

\begin{remark}
In case $X$ is geodesically complete, every cocompact point  from $\bd X$ has a discrete orbit for the Tits metric, see~\cite[Proposition~7.15]{Caprace-Monod_structure}. Therefore, Lemma~\ref{lem:7.14} shows in that case that the collection of cocompact antipodes of any given (possibly singular) cocompact point is Tits-discrete. Using the Parachute Lemma as above, this implies that the statement of Lemma~\ref{lem:AntipodalPair} holds for all cocompact points, and not only the regular ones. This will in fact follow \emph{a posteriori} from Proposition~\ref{prop:SphBldg}, even without the assumption that $X$ be geodesically complete. 
\end{remark}

%%%%%%%%%%%%%%%%%%%%%%%%%%%%%%%%%%%%%%%%%%%%%%%%%%%%%%%%%%%%%%%%%%%%%%%%%%%%%%
\subsection{The set of cocompact points is a  building}
%%%%%%%%%%%%%%%%%%%%%%%%%%%%%%%%%%%%%%%%%%%%%%%%%%%%%%%%%%%%%%%%%%%%%%%%%%%%%%

We are now ready to establish the following key step towards Theorem~\ref{thm:NotGoedComplete}.  

\begin{prop}\label{prop:SphBldg}
Let $(X, G)$ be a locally compact  \cat group. Assume that $G$ is totally disconnected, has a cocompact amenable subgroup and does not fix any point in $\bd X$. 

Then the set $C \subset \bd X$ of cocompact points, endowed with the Tits-metric, is a metric spherical building and $\dim C = \dim \bd X$.
\end{prop}

\begin{proof}%[Proof of Proposition~\ref{prop:SphBldg}]
As before, we let $C \subset \bd X$ denote the set of cocompact points and $\Creg \subset C$ the set of regular cocompact ones. 

By Proposition~\ref{prop:CregNonempty}, the set $\Creg$ is non-empty. Moreover every cocompact point is contained in a fully maximal sphere by Proposition~\ref{prop:MaxFlat}, and if such a sphere contains a point of $\Creg$, then it is entirely contained in $C$ by Proposition~\ref{prop:CocptSphere}. In particular, there exists at least one fully maximal sphere, say $S$, entirely  contained in $C$.

\setcounter{claim}{0}
\begin{claim}\label{claim:SphBdg1}
The $G$-orbit of every cocompact point meets $S$.
\end{claim}

Indeed, let $z \in C$. By Lemma~\ref{lem:AntipodalPt}, there exists  $\xi \in  \Ant(z) \cap S$. Since $\xi$ is cocompact, the set $\Ant(\xi) \cap C$ is contained in a single $G$-orbit by Lemma~\ref{lem:7.14}. Since $\Ant(\xi) \cap C$ clearly contains the unique element of $\Ant(\xi) \cap S$, we deduce that $G$-orbit of $z$ meets $S$, as claimed.

\begin{claim}\label{claim:SphBdg2}
Let $\xi \in \Creg \cap S$. The $G_\xi$-orbit of every cocompact point meets $S$.
\end{claim}

Let $\zeta \in C$. 
If $\zeta$ is antipodal to $\xi$, then $\zeta \in \Op(\xi)$ by Lemma~\ref{lem:AntipodalPair} and hence, there exists $g \in G_\xi$ with $g\zeta \in S$ by Corollary~\ref{cor:transitive}. We are done in this case.

Suppose now that $\tangle \xi  {\zeta}< \pi$. By Claim~\ref{claim:SphBdg1}, the point $\zeta$ belongs to some fully maximal sphere $S'$ which is a $G$-translate of $S$. In particular $S' \subset C$. By Lemma~\ref{lem:AntipodalPt}, there exist $\xi' \in S' \cap \Ant(\xi)$ and a geodesic segment $\gamma$ from $\xi$ to $\xi'$ passing through $\zeta$. By Lemma~\ref{lem:AntipodalPair}, the points  $\xi$ and $\xi'$ are contained in a common fully maximal sphere, say $S''$. Corollary~\ref{cor:OppositePair} implies that $\xi'$ is regular.  Proposition~\ref{prop:regular:1}(ii)  then ensures that  $\Sigma(\xi') \subset S' \cap S''$ contains a neighbourhood  of $\xi'$ in $S''$, so that  the geodesic segment $\gamma$ contains a sub-segment of positive length entirely contained in $S''$. Since $S''$ is convex and also contains both endpoints of $\gamma$, we infer that $\gamma$ is entirely contained in $S''$. In particular $\zeta \in S''$.  Lemma~\ref{lem:bij} affords an element $g \in G
 _\xi$ with $gS'' =S$, so that  $g\zeta \in S$, as required.   The claim stands proven. 

\begin{claim}\label{claim:SphBdg3}
Any two points of $C$ are contained in a common fully maximal sphere. In particular $C$ is convex. 
\end{claim}

Let   $\zeta_1, \zeta_2 \in C$. By Claim~\ref{claim:SphBdg1}, there is $g \in G$ with $g\zeta_1 \in S$. By Proposition~\ref{prop:CocptSphere},  there exists  $\xi \in  \Creg \cap S$ such that $g\zeta_1 \in \Sigma(\xi)$. By Claim~\ref{claim:SphBdg2} there exists $h \in G_\xi$ with $h g \zeta_2 \in S$. Since $h$ fixes $\xi$, it stabilises $\Sigma(\xi)$ so that $hg \zeta_1 \in \Sigma(\xi) \subset S$. Hence the pair $\{\zeta_1, \zeta_2\}$ is contained in the fully maximal sphere $(hg)\inv S$.

\begin{claim}\label{claim:SphBdg4}
Let $\xi \in \Creg \cap S$. Then there exists $\vareps >0$ such that $B(\xi, \vareps) \cap C \subset S$. In particular the Tits-closure of ${B(\xi, \vareps)} \cap C$ is Tits-compact. 
\end{claim}

By Proposition~\ref{prop:regular:1}(ii), there exists $\vareps >0$ such that $B(\xi, \vareps) \cap S \subset \Sigma(\xi)$. Given $z \in B(\xi, \vareps) \cap C$, there is $g \in G_\xi$ with $gz \in B(\xi, \vareps) \cap S$ by Claim~\ref{claim:SphBdg2}. Hence $z \in g\inv\Sigma(\xi) = \Sigma(\xi)$. The claim follows.

\medskip
We can now conclude the proof by applying Theorem~\ref{thm:BalserLytchak} to the set $C$. Indeed $C$ is a \catun space by Claim~\ref{claim:SphBdg3}. It contains the fully maximal sphere $S$, so that $\dim(\bd X) = \dim(S) \leq \dim(C) \leq \dim(\bd X)$. The other required hypotheses follow from Claims~\ref{claim:SphBdg3} and~\ref{claim:SphBdg4}. Hence $C$ is a spherical building.
\end{proof}

%%%%%%%%%%%%%%%%%%%%%%%%%%%%%%%%%%%%%%%%%%%%%%%%%%%%%%%%%%%%%%%%%%%%%%%%%%%%%%
\subsection{Every boundary point is cocompact}
%%%%%%%%%%%%%%%%%%%%%%%%%%%%%%%%%%%%%%%%%%%%%%%%%%%%%%%%%%%%%%%%%%%%%%%%%%%%%%

Here is the last step in the proof of that $\bd X$ is a building.  

\begin{prop}\label{prop:AllCocompact}
Let $(X, G)$ be a locally compact  \cat group. Assume that $G$ is totally disconnected, has a cocompact amenable subgroup and does not fix any point in $\bd X$. 
Then every $\xi \in \bd X$ is cocompact. In particular $\bd X$ is a spherical building.
\end{prop}

\begin{lem}\label{lem:OppCocompact}
Let $(X, G)$ be a locally compact  \cat group. Assume that $G$ is totally disconnected, has a cocompact amenable subgroup and does not fix any point in $\bd X$.

Any point opposite to a cocompact point is itself cocompact.
\end{lem}
\begin{proof}
Let $\xi$ be cocompact. By Proposition~\ref{prop:SphBldg}, the point $\xi$ is contained in a fully maximal sphere $S$ entirely consisting of cocompact points. The antipode of $\xi$ in $S$ is opposite $\xi$ by Proposition~\ref{prop:Leeb}. The conclusion follows since the stabiliser $G_\xi$ acts cocompactly and $X$ and is thus  transitive on $\Op(\xi)$ by Corollary~\ref{cor:transitive}.
\end{proof}

\begin{lem}\label{lem:Chamber}
Let $(X, G)$ be a locally compact  \cat group. Assume that $G$ is totally disconnected, has a cocompact amenable subgroup and does not fix any point in $\bd X$. 

Let $\xi \in \Creg$ and $z \in \bd X$. Then $\Sigma(\xi) \cap \Ant(z)$ contains at most one point. 
\end{lem}

\begin{proof}
We may assume that  $\Sigma(\xi) \cap \Ant(z)$ contains at least one point, say $\zeta$.
Let $B$ denote the pointwise stabiliser of $\Sigma(\xi)$ in $G$. By Proposition~\ref{prop:regular:1}(iv), the group $B$ acts cocompactly on $X$ and thus contains a sequence $(g_n)$ which is radial for $z$. Upon extracting we may assume that $(g_n z)$ converges to some $z' \in \bd X$.

Since $g_n \zeta = \zeta$ for all $n$, we deduce from Proposition~\ref{prop:upper:semi} that  $\zeta$ and $z'$ are opposite. Hence $z'$ is cocompact by Lemma~\ref{lem:OppCocompact}. By Proposition~\ref{prop:SphBldg},  there exists a fully maximal sphere $S$ containing $\xi$ and $z'$. Hence $S$ contains $\Sigma(\xi) \cup \{z'\}$. 

Let now $\eta \in \Sigma(\xi) \cap \Ant(z)$. Proposition~\ref{prop:upper:semi} again implies that $\eta$ and $z'$ are opposite. Hence $\eta = \zeta$ since $\eta \in \Sigma(\xi) \subset S$ and since $\zeta$ is the only point of $S$ antipodal to $z'$.
\end{proof}

\begin{proof}[Proof of Proposition~\ref{prop:AllCocompact}]
Let $z \in \bd X$. By Proposition~\ref{prop:SphBldg} the set $C$ of cocompact points contains a fully maximal sphere, say $S$, which is covered by a finite union of sets of the form $\Sigma(\xi)$ with $\xi \in \Creg$ (see Lemma~\ref{lem:Bieberbach} and Proposition~\ref{prop:CocptSphere}).  Each $\Sigma(\xi)$ contains at most one antipode of $z$ by Lemma~\ref{lem:Chamber}, so that the set $\Ant(z) \cap S$ is finite, hence discrete. Since the set $\Ant(z) \cap S$ is moreover non-empty by Lemma~\ref{lem:AntipodalPt}, we may pick $z' \in \Ant(z) \cap S$ and a neighbourhood $U$ of $z'$ in $S$ containing no other antipode of $z$. By the Parachute Lemma~\ref{lem:para}, the pair $\{z, z'\}$ is contained in a fully maximal sphere, and is thus opposite by Proposition~\ref{prop:Leeb}. We conclude via Lemma~\ref{lem:OppCocompact} that $z \in C$.
\end{proof}

%%%%%%%%%%%%%%%%%%%%%%%%%%%%%%%%%%%%%%%%%%%%%%%%%%%%%%%%%%%%%%%%%%%%%%%%%%%%%%
\subsection{The spherical building is Moufang}
%%%%%%%%%%%%%%%%%%%%%%%%%%%%%%%%%%%%%%%%%%%%%%%%%%%%%%%%%%%%%%%%%%%%%%%%%%%%%%

\begin{proof}[Proof of Theorem~\ref{thm:Moufang}]
We set $G= \Isom(X)$. By hypothesis $(X, G)$ is a locally compact \cat group.

Any two points of $\bd X$ are contained in a common fully maximal sphere. In particular any two antipodal points are opposite by Proposition~\ref{prop:Leeb}. 
Given $\eta \in \bd X$, there exists $\xi \in \Ant(\eta)$. By Corollary~\ref{cor:transitive}, the  group $G_\xi$ is transitive on $\Op(\xi) = \Ant(\xi)$, and the latter set contains at least two points since the building $\bd X$ is thick. In particular $G$ does not fix any point in $\bd X$. 

We may therefore invoke Theorem~\ref{thm:structure}. Since the spherical building of a Lie group of rank~$\geq 2$ is Moufang (see \S\ref{sec:Moufang}), we are reduced to the case where $X$ is an irreducible \cat space and $G= \Isom(X)$ is totally disconnected. 
Moreover, by Lemma~\ref{lem:MetricVsCombinatorial}, the metric building $\bd X$ may also be viewed as an irreducible combinatorial spherical building. Therefore, the Moufang condition is automatic if $\dim(\bd X) \geq 2$  by~\cite[Satz~1]{Tits77} or~\cite[Theorem~11.6]{Weiss03}. We   focus henceforth on the case $\dim(\bd X) = 1$. 

\setcounter{claim}{0}
\begin{claim}\label{cl:1}
Every interior point $\xi$ of a chamber $c$ of $\bd X$ is regular.  In particular $c = \Sigma(\xi)$.
\end{claim}

Let $\xi' \in \Op(\xi)$. Since belongs to the interior of $c$, it follows from~\cite[Lemma~3.10.1]{Kleiner-Leeb} that $\Link(\xi, \xi')$ is a round sphere. Therefore $\xi$
is regular by Lemma~\ref{lem:regular:criterion}. Hence $c = \Sigma(\xi)$ by Lemma~\ref{lem:Chamber=Intersection}.

\begin{claim}\label{cl:2}
For each half-apartment $\alpha  \subset  \bd X$, the pointwise stabiliser of $\alpha$ in $G$  is transitive on the set of apartments containing $\alpha$. 
\end{claim}

By Claim~\ref{cl:1},  the interior of $\alpha$ contains a regular point. Hence Claim~\ref{cl:2} holds by virtue of Lemma~\ref{lem:RootFixator}.

\begin{claim}\label{cl:3}
For each $\xi \in \bd X$, the group $G_\xi^{\mathrm u}$ fixes pointwise each chamber containing $\xi$.
\end{claim}

Since $\bd X$ is an irreducible building of dimension~$1$, each chamber  is a geodesic segment of length~$<\pi/2$.

If $\xi$ belongs to the interior of a chamber $c$,  then $c$ is the only chamber containing $\xi$ by Lemma~\ref{lem:Chamber=Intersection}. Hence $G_\xi$ stabilises $c$ and therefore fixes its midpoint $\xi_0$. If $G_\xi^{\mathrm u}$ does not fix $c$ pointwise, then it contains an element acting on $c$ as the reflection through $\xi_0$. Since $\bd X$ is unidimensional, it follows that the convex set $(\bd X)^{G_\xi^{\mathrm u}}$ coincides with $\{ \xi_0\}$. This is impossible by virtue of Proposition~\ref{prop:TitsNb}.

We assume  henceforth  that $\xi$ is not an interior point of a chamber. 
%Since $\bd X$ is unidimensional, it follows that the union of all chambers containing $\xi$ is convex and isometric to a star around $\xi$. 
Let then $\eta \in \bd X$ be an interior point of a chamber $c$ containing $\xi$, so $\Td(\eta, \xi) < \frac \pi 2$. If $\eta$ is not fixed by $G_\xi^{\mathrm u}$, then its projection $\eta'$ to $(\bd X)^{G_\xi^{\mathrm u}}$ is different from $\xi $ by Proposition~\ref{prop:TitsNb}. Hence $\Td(\eta, \eta') < \Td(\eta, \xi) < \frac \pi 2$. Therefore the geodesic triangle $(\eta, \eta', \xi)$ has perimeter~$<\pi$, and must thus be a tripod since $\bd X$ is a  unidimensional \catun space. Therefore $\eta' \in [\eta, \xi]$ since  $[\eta', \xi]$ is pointwise fixed by ${G_\xi^{\mathrm u}}$. As $\eta' \neq \xi$, it follows that $\eta'$ is an interior point of the chamber $c$. This shows that ${G_\xi^{\mathrm u}}$ fixes a subsegment of $c$ of positive length; hence it fixes $c$ pointwise, contradicting the assumption that $\eta$ is not fixed. This proves that every interior point of every chamber containing $\xi$ is fixed by  ${G_\xi^{\mathrm u}}$. The claim follows.

\medskip
We now conclude the proof of the theorem as follows. Fix a   half-apartment $\alpha \subset \bd X$ and denote by $G_\alpha$ the pointwise stabiliser of $\alpha$ in $G$. Let also $A(\alpha)$ denote the set of apartments containing $\alpha$.  By Claim~\ref{cl:2}, the group $G_\alpha$ is transitive on $A(\alpha)$.

Since $\bd X$ is $1$-dimensional, the half-apartment $\alpha$ is  a geodesic segment of length~$\pi$. Its chambers can be numbered successively by $c_1, \dots, c_n$. The unique common point of $c_i$ and $c_{i+1}$ is denoted by $\eta_i$ for $i =1, \dots, n-1$. The group $G_\alpha$ is contained in $\bigcap_{i=1}^{n-1} G_{\eta_i}$.

Fix an apartment $S \in A(\alpha)$, so that $S$ is a fully maximal sphere of $\bd X$. 
%By Proposition~\ref{prop:Leeb}, there is a fully maximal flat $F \subset X$ invariant under $\Stab_G(S)$ and with $\bd F = S$. 
By Claim~\ref{cl:1}, there exists a regular cocompact point $\xi \in S$. Let $\xi' \in \Op(\xi)$ be its unique antipode contained in $S$. By Proposition~\ref{prop:regular:1}(iv), the pointwise fixator $B$ of $\Sigma(\xi)$ acts cocompactly on $X$. Hence the group $T = B_{\xi'}$ acts cocompactly on $P(\xi, \xi')$ by Proposition~\ref{prop:lifting}. Since $T$ fixes pointwise the chamber  $\Sigma(\xi)$, it fixes pointwise the sphere $S = \bd P(\xi, \xi')$.

Pick a base point $p \in P(\xi, \xi')$ and for each $i$, choose a sequence $(t_{i, n})$ of elements of $T$ which is radial for $\eta_i$. Upon extracting, we may assume that   $\lim_n t_{i, n} p$ exists. Hence this limit coincides with the unique point $\eta'_i \in S$ opposite $\eta_i$. We define inductively a chain of subgroups
$$G_\alpha = V_0 > V_1 > \dots > V_{n-1}$$
by setting 
$$V_{i} = V_{i-1} \cap G^{\mathrm u}_{\eta_i},$$ 
It follows from the definition that $V_i$ is normal in $G_\alpha$ for all $i$. Since $T$ acts trivially on $S$ and is thus contained in 
$G_\alpha$, it follows that $T$ normalises $G^{\mathrm u}_{\eta_i}$ for all $i$. In particular  $t_{j, n}$ normalises $V_i$ for all $i, j, n$. 

By Claim~\ref{cl:3}, the group $V_i$ fixes pointwise each chamber containing $\eta_i$. By induction, the group $V_{i-1}$ is transitive on $A(\alpha)$ and is normalised by each element of the sequence $(t_{i, n})_n$. We are thus in a position to apply  Theorem~\ref{thm:Levi} to the sequence $ (t_{i, n})_n$ and the subgroup $V_i$ with respect to the boundary point $\eta_i$.   This provides the  decomposition
$$
V_{i-1} = V_i \cdot V_{i-1, \eta'_i}.
$$
Since $V_{i-1}$ is transitive on $A(\alpha)$ by induction and since there is only one element of $A(\alpha)$ which contains $\eta'_i$ (namely the convex hull of $\alpha \cup \{\eta'_i\}$), we deduce that $V_i$ is transitive on $A(\alpha)$ as well. 

Now the group group $V_{n-1}$ fixes pointwise each chamber containing $\eta_i$ for all $i=1, \dots, n-1$. It  is thus contained in the root group $U_\alpha$. The transitivity of $V_{n-1}$ on $A(\alpha)$ therefore implies that $\bd X$ satisfies the Moufang condition, as desired. 
\end{proof}

\begin{proof}[Proof of Corollary~\ref{cor:Building}]
The visual boundary $\bd X$ is a metric spherical building (see~\cite[Proposition~4.2.1]{Kleiner-Leeb}) which is thick and irreducible because $X$ is so. The desired conclusion will follow from Theorem~\ref{thm:Moufang} if we prove   that   $G_\xi$ acts cocompactly on $X$ for each $\xi \in \bd X$. 

Let $Z \subset \bd X$ the set of circumcentres of chambers at infinity. Then $Z$ is closed for the c\^one topology: indeed, fixing a base point $x \in X$, there is a one-to-one correspondence between the chambers at infinity and the sectors based at $x$. Given a sequence of chambers, the corresponding sequence of sectors has a convergent subsequence; the closedness of $Z$ follows. 

Since $G$ is transitive on the set of chambers at infinity, it is transitive on $Z$. Therefore for each $\eta \in Z$ the stabiliser $G_\eta$ acts cocompactly on $X$  by Lemma~\ref{lem:CocptOrbit}. The isometry group of each chamber at infinity is finite. Therefore, the pointwise stabiliser of each chamber still acts cocompactly on $X$. Every boundary point $\xi$ is contained in a chamber, and thus has a cocompact stabiliser.
\end{proof}

%%%%%%%%%%%%%%%%%%%%%%%%%%%%%%%%%%%%%%%%%%%%%%%%%%%%%%%%%%%%%%%%%%%%%%%%%%%%%%
\subsection{Compact spherical buildings}
%%%%%%%%%%%%%%%%%%%%%%%%%%%%%%%%%%%%%%%%%%%%%%%%%%%%%%%%%%%%%%%%%%%%%%%%%%%%%%

The proof of Corollary~\ref{cor:ModelSpace} requires to use the notion of  a \textbf{compact building}\index{compact building}\index{spherical building!compact}. Following Burns and Spatzier~\cite{Burns-Spatzier1}, this is defined as a combinatorial spherical building $\Delta$ of type $I$ such that for each $i \in I$, the set $V_i$ of vertices to type $i$ is given a compact topology satisfying the condition that the chamber set $\Ch(\Delta) $ is closed in the product space $\prod_i V_i$. 

\begin{lem}\label{lem:CompactBuilding}
Let $(X, G)$ be a locally compact  \cat group such that $\bd X$ is irreducible. Assume that $G$ is totally disconnected, has a cocompact amenable subgroup and does not fix any point in $\bd X$. 

Then $\bd X$ is a compact spherical building, where the topology on the vertex set of each type is inherited from the c\^one topology on $\bd X$. 
\end{lem}

\begin{proof}
The argument below is an adaption of the proof of~\cite[Proposition~6.4]{GKVMW}.  We know from Proposition~\ref{prop:AllCocompact} that $\bd X$ is a metric spherical building, which is irreducible by Lemma~\ref{lem:IrredBuilding}. By Lemma~\ref{lem:MetricVsCombinatorial} $\bd X$   also carries naturally the structure of a combinatorial building; we let  $I$ denote  its type set and  for $i \in I$, we let $V_i$ denote the set of vertices of type $i$. Upon replacing $G$ by a finite index subgroup, we may assume that $G$ acts on $\bd X$ by type-preserving automorphisms. 

For each vertex $v$ of $\bd X$, the stabiliser $G_v$ is cocompact in $G$ (see Proposition~\ref{prop:AllCocompact}). Therefore the $G$-orbit of $v$ is closed in $\bd X$ with respect to the c\^one topology (see Lemma~\ref{lem:CocptOrbit}), and we infer that the orbit map defines a homeomorphism $G/G_v \to Gv = V_i$, where $i$ is the type of $v$. 

Now, for each chamber $c \in \Ch(\bd X)$, the stabilizer $G_c$ is also  cocompact   in $G$. Therefore, if $K \leq G$ is any compact open subgroup, there is a finite set of left cosets of $K$ whose union maps onto the compact quotient $G/G_c$. In other words, this means  that $K$ acts with finitely many orbits on $\Ch(\bd X)$. Therefore, in order to show that $\Ch(\bd X)$ is closed in the product $\prod_{i \in I} V_i$, it is enough to show that the $K$-orbit of $c$ is closed. 

Let us denote by $\{v_i \; | \; i \in I\}$ the set of vertices of $c$. Then $K_{v_i}$ is closed and contains $K_c$. Therefore, the map 
$$K/K_c \to \prod_{i \in I} K/K_{v_i}$$
has closed image. Since $ \prod_{i \in I} K/K_{v_i}$ is canonically homeomorphic to a closed subset of $ \prod_{i \in I} G/G_{v_i}$, the desired assertion follows from the existence of a homeomorphism  $G/G_{v_i}  \to   V_i$ which was established above.
\end{proof}

%%%%%%%%%%%%%%%%%%%%%%%%%%%%%%%%%%%%%%%%%%%%%%%%%%%%%%%%%%%%%%%%%%%%%%%%%%%%%%
\subsection{Building buildings by means of means}
%%%%%%%%%%%%%%%%%%%%%%%%%%%%%%%%%%%%%%%%%%%%%%%%%%%%%%%%%%%%%%%%%%%%%%%%%%%%%%

\begin{flushright}
\begin{minipage}[t]{0.75\linewidth}\itshape\small
There was an elliptical precision about its perfect pairs of parts that was graceful and shocking, like good modern art, and at times Yossarian wasn't quite sure that he saw it all, just the way he was never quite sure about good modern art. \hfill\upshape (J.~Heller, \emph{Catch-22}, 1961)
\end{minipage}
\end{flushright}

\begin{proof}[Proof of Theorem~\ref{thm:NotGoedComplete}]
We apply Theorem~\ref{thm:structure}. We may replace $X$ by its subspace $X'$  since they have the same visual boundary. Each irreducible factor in the canonical product decomposition of $X$ afforded by Theorem~\ref{thm:structure} admits a cocompact action of an amenable locally compact group: indeed, one may take the closure of the projection of the  amenable cocompact subgroup of $\Isom(X)$ given by hypothesis. Therefore, it suffices to treat one irreducible factor of $X$ at a time. 

For a factor whose isometry group is Lie, the desired conclusions follow from Theorem~\ref{thm:structure} (the spherical building of a simple Lie group is automatically Moufang).

For a factor $Y_j$ with a totally disconnected isometry group, we apply Propositions~\ref{prop:SphBldg} and~\ref{prop:AllCocompact}, ensuring that 
$\bd Y_j$ is a spherical building all of whose points are cocompact. 

Assume now that $\dim(\bd Y_j) \geq 1$. Since $\bd Y_j$ is an irreducible \catun space, we deduce from Lemma~\ref{lem:IrredBuilding}  that $\bd Y_j$ satisfies the hypotheses of  Theorem~\ref{thm:Moufang}. Hence the building $\bd Y_j$ is Moufang.
\end{proof}

\begin{proof}[Proof of Theorem~\ref{cor:CoctAmen}]
Any geodesically complete locally compact \cat space is proper by Hopf--Rinow (see~\cite[Theorem~2.4]{BallmannLN}). By Theorem~\ref{thm:NotGoedComplete}, the visual boundary of $X$ is a spherical building, and every point of $\bd X$ has a cocompact stabiliser. By~\cite[Theorem~1.3]{Caprace-Monod_structure}, this implies that $X$ is a product of flats, irreducible Euclidean buildings and Bass--Serre (\emph{i.e} {edge-transitive}) trees. Corollary~\ref{cor:Building} ensures that each irreducible component of dimension~$\geq 2$ is in fact a Bruhat--Tits building.
\end{proof}

\begin{proof}[Proof of Corollary~\ref{cor:ModelSpace}]
As in the proof of Theorem~\ref{thm:NotGoedComplete}, it suffices by Theorem~\ref{thm:structure} to consider the case where $X$ is irreducible and $\Isom(X)$ is totally disconnected and acts minimally. We know from Theorem~\ref{thm:NotGoedComplete} that $\bd X$ is a thick irreducible metric spherical building. If $\dim(\bd X) = 0$, then $X$ is Gromov hyperbolic and the existence of a continuous, proper, edge-transitive $\Isom(X)$-action on a locally finite tree follows from~\cite[Theorem~D]{CCMT}. If $\dim(\bd X) \geq 1$, then $\bd X$ is a Moufang building. By Lemma~\ref{lem:CompactBuilding}, it is also a compact building in the sense of Burns--Spatzier~\cite{Burns-Spatzier1}. Therefore, it follows from~\cite[Theorem~1.1]{GKVMW} that there is a Bruhat--Tits building $X_{\mathrm{model}}$ whose building at infinity $\bd X_{\mathrm{model}}$ is isomorphic (both as a compact building and as a \catun space) to $\bd X$. Moreover~\cite[Theorem~1.1]{GKVMW} ensures that the canonical homomorphism of $\Aut(X_{\mathrm{model}})$ to $\Aut(\bd X_{\mathrm{model}})$ is an isomorphism of topological groups,  where $\Aut(X_{\mathrm{model}})$ (resp. $\Aut(\bd X_{\mathrm{model}})$) is endowed with the compact-open topology for its action on $X$ (resp. on the set of chambers of $\bd X_{\mathrm{model}}$).  We may therefore identify $\Aut(X_{\mathrm{model}})$ with $\Aut(\bd X_{\mathrm{model}})$. 

The group $\Isom(X)$ acts continuously on the space $\bd X$ endowed with the c\^one topology. The kernel of this action is compact by Proposition~\ref{prop:BoundaryAction},  so the isomorphism $\bd X \to \bd X_{\mathrm{model}}$ yields a continuous homomorphism with compact kernel from $\Isom(X)$ to  $\Aut(\bd X_{\mathrm{model}}) = \Aut(X_{\mathrm{model}})$. We emphasize that there is a subtlety in checking that the isomorphism $\bd X \to \bd X_{\mathrm{model}}$ induces an isomorphism of topological groups $\Aut(\bd X) \to \Aut(\bd X_{\mathrm{model}})$. Indeed, the group $\Aut(\bd X)$ can be topologized in two ways which are a priori different, namely the compact-open topology with respect to its action on the space $\bd X$ endowed with the c\^one topology, and the compact-open topology with respect to its action on the space of chambers of $\bd X$. However, mapping each chamber to its circumcentre, the space of chambers can be identified with a closed $\Aut(\bd X)$-invariant 
 subset of $\bd X$, and one deduces that these two topologies coincide, so that the map $\Aut(\bd X) \to \Aut(\bd X_{\mathrm{model}})$ is indeed an isomorphism of topological groups. 

Since the image of $\Isom(X)$ in $\Aut(\bd X)$ contains all the root subgroups by Theorem~\ref{thm:Moufang} it follows that the $\Isom(X)$-action on the Bruhat--Tits building $\Aut(X_{\mathrm{model}})$ is chamber-transitive, hence cocompact. The fact that it is proper follows from the closedness of image of $\Isom(X)$ in the group $\Aut(\bd X)$ endowed with the compact-open topology, asserted by Proposition~\ref{prop:BoundaryAction}. \end{proof}

%=========================================================================================================

\bibliographystyle{plain}
\bibliography{../IsomCAT0}

\begin{thebibliography}{10}

\bibitem{AbramenkoBrown}
Peter Abramenko and Kenneth~S. Brown.
\newblock {\em Buildings}, volume 248 of {\em Graduate Texts in Mathematics}.
\newblock Springer, New York, 2008.
\newblock Theory and applications.

\bibitem{AB98}
Scot Adams and Werner Ballmann.
\newblock Amenable isometry groups of {H}adamard spaces.
\newblock {\em Math. Ann.}, 312(1):183--195, 1998.

\bibitem{Anderson87}
Michael~Thomas Anderson.
\newblock On the fundamental group of nonpositively curved manifolds.
\newblock {\em Math. Ann.}, 276(2):269--278, 1987.

\bibitem{Arens46}
Richard Arens.
\newblock Topologies for homeomorphism groups.
\newblock {\em Amer. J. Math.}, 68:593--610, 1946.

\bibitem{Avez}
Andr{\'e} Avez.
\newblock Vari{\'e}t{\'e}s riemanniennes sans points focaux.
\newblock {\em C. R. Acad. Sci. Paris S{\'e}r. A-B}, 270:A188--A191, 1970.

\bibitem{AzencottWilson}
Robert Azencott and Edward~N. Wilson.
\newblock Homogeneous manifolds with negative curvature. {I}.
\newblock {\em Trans. Amer. Math. Soc.}, 215:323--362, 1976.

\bibitem{BallmannLN}
Werner Ballmann.
\newblock {\em Lectures on spaces of nonpositive curvature}, volume~25 of {\em
  DMV Seminar}.
\newblock Birkh\"auser Verlag, Basel, 1995.
\newblock With an appendix by Misha Brin.

\bibitem{BallmannBrin_Duke99}
Werner Ballmann and Michael Brin.
\newblock Diameter rigidity of spherical polyhedra.
\newblock {\em Duke Math. J.}, 97(2):235--259, 1999.

\bibitem{BalserLytchak_Centers}
Andreas Balser and Alexander Lytchak.
\newblock Centers of convex subsets of buildings.
\newblock {\em Ann. Global Anal. Geom.}, 28(2):201--209, 2005.

\bibitem{BaLy_BuildingLikeSpaces}
Andreas Balser and Alexander Lytchak.
\newblock Building-like spaces.
\newblock {\em J. Math. Kyoto Univ.}, 46(4):789--804, 2006.

\bibitem{Bridson-Haefliger}
Martin~R. Bridson and Andr\'e Haefliger.
\newblock {\em {Metric spaces of non-positive curvature}}.
\newblock {Grundlehren der Mathematischen Wissenschaften 319, Springer,
  Berlin}, 1999.

\bibitem{BurgerSchroeder87}
Marc Burger and Viktor Schroeder.
\newblock Amenable groups and stabilizers of measures on the boundary of a
  {H}adamard manifold.
\newblock {\em Math. Ann.}, 276(3):505--514, 1987.

\bibitem{Burns-Spatzier1}
Keith Burns and Ralf Spatzier.
\newblock On topological {T}its buildings and their classification.
\newblock {\em Inst. Hautes \'Etudes Sci. Publ. Math.}, 65:5--34, 1987.

\bibitem{PCMI}
Pierre-Emmanuel Caprace.
\newblock Lectures on proper {CAT}($0$) spaces and their isometry groups.
\newblock To appear in the proceedings of PCMI 2012.

\bibitem{CCMT}
Pierre-Emmanuel Caprace, Yves Cornulier, Nicolas Monod, and Romain Tessera.
\newblock Amenable hyperbolic groups.
\newblock To appear in J. Europ. Math. Soc.

\bibitem{Caprace-Monod_structure}
Pierre-Emmanuel Caprace and Nicolas Monod.
\newblock {Isometry groups of non-positively curved spaces: structure theory}.
\newblock {\em J Topology}, 2(4):661--700, 2009.

\bibitem{amenis}
Pierre-Emmanuel Caprace and Nicolas Monod.
\newblock Fixed points and amenability in non-positive curvature.
\newblock {\em Math. Ann.}, 356(4):1303--1337, 2013.

\bibitem{FoertschLytchakGAFA}
Thomas Foertsch and Alexander Lytchak.
\newblock The de {R}ham decomposition theorem for metric spaces.
\newblock {\em Geom. Funct. Anal.}, 18(1):120--143, 2008.

\bibitem{Geoghegan-Ontaneda}
Ross Geoghegan and Pedro Ontaneda.
\newblock Boundaries of cocompact proper {${\rm CAT}(0)$} spaces.
\newblock {\em Topology}, 46(2):129--137, 2007.

\bibitem{Gromoll-Wolf}
Detlef Gromoll and Joseph~A. Wolf.
\newblock Some relations between the metric structure and the algebraic
  structure of the fundamental group in manifolds of nonpositive curvature.
\newblock {\em Bull. Amer. Math. Soc.}, 77:545--552, 1971.

\bibitem{GKVMW}
Theo Grundh{\"o}fer, Linus Kramer, Hendrik Van~Maldeghem, and Richard~M. Weiss.
\newblock Compact totally disconnected {M}oufang buildings.
\newblock {\em Tohoku Math. J. (2)}, 64(3):333--360, 2012.

\bibitem{GuralnikSwenson}
Dan~P. Guralnik and Eric~L. Swenson.
\newblock A `transversal' for minimal invariant sets in the boundary of a
  {CAT}(0) group.
\newblock {\em Trans. Amer. Math. Soc.}, 365(6):3069--3095, 2013.

\bibitem{Heintze74}
Ernst Heintze.
\newblock On homogeneous manifolds of negative curvature.
\newblock {\em Math. Ann.}, 211:23--34, 1974.

\bibitem{Kleiner}
Bruce Kleiner.
\newblock The local structure of length spaces with curvature bounded above.
\newblock {\em Math. Z.}, 231(3):409--456, 1999.

\bibitem{Kleiner-Leeb}
Bruce Kleiner and Bernhard Leeb.
\newblock Rigidity of quasi-isometries for symmetric spaces and {E}uclidean
  buildings.
\newblock {\em Inst. Hautes \'Etudes Sci. Publ. Math.}, 86:115--197 (1998),
  1997.

\bibitem{Lawson-Yau}
H.~Blaine Lawson, Jr. and Shing-Tung Yau.
\newblock Compact manifolds of nonpositive curvature.
\newblock {\em J. Differential Geometry}, 7:211--228, 1972.

\bibitem{Leeb}
Bernhard Leeb.
\newblock {\em A characterization of irreducible symmetric spaces and
  {E}uclidean buildings of higher rank by their asymptotic geometry}.
\newblock Bonner Mathematische Schriften, 326. Universit\"at Bonn
  Mathematisches Institut, Bonn, 2000.

\bibitem{Lytchak_RigidityJoins}
Alexander Lytchak.
\newblock Rigidity of spherical buildings and joins.
\newblock {\em Geom. Funct. Anal.}, 15(3):720--752, 2005.

\bibitem{LytchakSchroeder}
Alexander Lytchak and Viktor Schroeder.
\newblock Affine functions on {${\rm CAT}(\kappa)$}-spaces.
\newblock {\em Math. Z.}, 255(2):231--244, 2007.

\bibitem{Monod_superrigid}
Nicolas Monod.
\newblock Superrigidity for irreducible lattices and geometric splitting.
\newblock {\em J. Amer. Math. Soc.}, 19(4):781--814, 2006.

\bibitem{Monod-Py}
Nicolas Monod and Pierre Py.
\newblock An exotic deformation of the hyperbolic space.
\newblock {\em Amer. J. Math.}, 136(5):1249--1299, 2014.

\bibitem{MuhlherrPeterssonWeiss}
Bernhard M\"uhlherr, Holger Petersson, and Richard~M. Weiss.
\newblock Descent in buildings.
\newblock Book to appear in Annals of Mathematics Studies.

\bibitem{Oliver}
Richard~K. Oliver.
\newblock On {B}ieberbach's analysis of discrete {E}uclidean groups.
\newblock {\em Proc. Amer. Math. Soc.}, 80(1):15--21, 1980.

\bibitem{PapasogluSwenson}
Panos Papasoglu and Eric Swenson.
\newblock Boundaries and {JSJ} decompositions of {CAT}(0)-groups.
\newblock {\em Geom. Funct. Anal.}, 19(2):559--590, 2009.

\bibitem{Swenson13}
Eric Swenson.
\newblock On cyclic {${\rm CAT}(0)$} domains of discontinuity.
\newblock {\em Groups Geom. Dyn.}, 7(3):737--750, 2013.

\bibitem{Tits77}
J.~Tits.
\newblock Endliche {S}piegelungsgruppen, die als {W}eylgruppen auftreten.
\newblock {\em Invent. Math.}, 43(3):283--295, 1977.

\bibitem{Tits74}
Jacques Tits.
\newblock {\em Buildings of spherical type and finite {BN}-pairs}.
\newblock Lecture Notes in Mathematics, Vol. 386. Springer-Verlag, Berlin,
  1974.

\bibitem{Tits_resume}
Jacques Tits.
\newblock {\em R\'esum\'es des cours au {C}oll\`ege de {F}rance 1973--2000}.
\newblock Documents Math\'ematiques (Paris) [Mathematical Documents (Paris)],
  12. Soci\'et\'e Math\'ematique de France, Paris, 2013.

\bibitem{TitsWeiss}
Jacques Tits and Richard~M. Weiss.
\newblock {\em Moufang polygons}.
\newblock Springer Monographs in Mathematics. Springer-Verlag, Berlin, 2002.

\bibitem{VanMaldeghem-VanSteen}
H.~Van~Maldeghem and K.~Van~Steen.
\newblock Characterizations by automorphism groups of some rank {$3$}
  buildings. {I}. {S}ome properties of half strongly-transitive triangle
  buildings.
\newblock {\em Geom. Dedicata}, 73(2):119--142, 1998.

\bibitem{Weiss03}
Richard~M. Weiss.
\newblock {\em The structure of spherical buildings}.
\newblock Princeton University Press, Princeton, NJ, 2003.

\bibitem{Weiss09}
Richard~M. Weiss.
\newblock {\em The structure of affine buildings}, volume 168 of {\em Annals of
  Mathematics Studies}.
\newblock Princeton University Press, Princeton, NJ, 2009.

\end{thebibliography}

\clearpage
\addcontentsline{toc}{section}{Index}

{\small 
\printindex
}
\end{document}